\documentclass[a4paper, reqno, 10pt]{amsart}

\usepackage[utf8]{inputenc}
\usepackage{amsthm, amsfonts, amssymb, amsmath}
\usepackage{enumitem}
\usepackage{eqparbox}
\usepackage{etoolbox}
\usepackage{pdfsync}
\usepackage{stmaryrd}
\usepackage{mathtools}

\usepackage{mathrsfs}
\usepackage{tikz}
\usetikzlibrary{matrix, arrows, cd}
\usepackage[colorlinks=False, citecolor=Black, linkcolor=Red, urlcolor=Blue]{hyperref}
\usepackage{cleveref}

\newtheorem{theoremalphabetical}{Theorem}[section]

\newtheorem{theorem}{Theorem}[section]
\newtheorem{proposition}[theorem]{Proposition}
\newtheorem{lemma}[theorem]{Lemma}

\newtheorem{corollary}[theorem]{Corollary}

\theoremstyle{definition}
\newtheorem{definition}[theorem]{Definition}
\newtheorem{example}[theorem]{Example}

\newtheorem{notation}[theorem]{Notation}
\newtheorem{setting}[theorem]{Setting}
\newtheorem{remark}[theorem]{Remark}

\newtheorem*{remarksnonumber}{Remarks}

\newenvironment{altenumerate}
{\begin{list}
		{\textup{(\theenumi)} }
		{\usecounter{enumi}
			\setlength{\labelwidth}{0pt}
			\setlength{\labelsep}{0pt}
			\setlength{\leftmargin}{0pt}
			\setlength{\itemsep}{0pt}
			\setlength{\topsep}{0pt}
			\renewcommand{\theenumi}{\roman{enumi}}
	}}
	{\end{list}}

\newenvironment{altitemize}
{\begin{list}
		{$\bullet$}
		{\setlength{\labelwidth}{0pt}
			\setlength{\itemindent}{5pt}
			\setlength{\labelsep}{5pt}
			\setlength{\leftmargin}{0pt}
			\setlength{\itemsep}{0pt}
			\setlength{\topsep}{0pt}
	}}
	{\end{list}}
\newenvironment{altenumeratelevel2}
{\begin{list}
		{\textup{(\theenumi)} }
		{\usecounter{enumii}
			\setlength{\labelwidth}{2em}
			\setlength{\labelsep}{0pt}
			\setlength{\leftmargin}{2em}
			\setlength{\itemsep}{2pt}
			\setlength{\topsep}{2pt}
			\setlength{\itemindent}{0pt}
			\renewcommand{\theenumi}{\arabic{enumii}}
	}}
	{\end{list}}

\numberwithin{equation}{section}

\makeatletter
\def\@seccntformat#1{%
	\protect\textup{\protect\@secnumfont
		\ifnum\pdfstrcmp{subsection}{#1}=0 \bfseries\fi% subsection # in \bfseries
		\csname the#1\endcsname
		\protect\@secnumpunct
	}%
}  
\makeatother

\makeatletter
\def\@tocline#1#2#3#4#5#6#7{\relax
	\ifnum #1>\c@tocdepth % then omit
	\else
	\par \addpenalty\@secpenalty\addvspace{#2}%
	\begingroup \hyphenpenalty\@M
	\@ifempty{#4}{%
		\@tempdima\csname r@tocindent\number#1\endcsname\relax
	}{%
		\@tempdima#4\relax
	}%
	\parindent\z@ \leftskip#3\relax \advance\leftskip\@tempdima\relax
	\rightskip\@pnumwidth plus4em \parfillskip-\@pnumwidth
	#5\leavevmode\hskip-\@tempdima
	\ifcase #1
	\or\or \hskip 1em \or \hskip 2em \else \hskip 3em \fi%
	#6\nobreak\relax
	\hfill\hbox to\@pnumwidth{\@tocpagenum{#7}}\par% <---- \dotfill -> \hfill
	\nobreak
	\endgroup
	\fi}
\makeatother

\newcommand{\Aut}{{\mathrm{Aut}}}
\newcommand{\abs}[1]{{\lvert{#1}\rvert}}

\newcommand{\ad}{{\mathrm{ad}}}

\newcommand{\blank}{{\,\_\,}}

\newcommand{\Coker}{{\mathrm{Coker}}}

\newcommand{\cont}{{\mathrm{cts}}}

\renewcommand{\det}{{\mathrm{det}}}

\newcommand{\defeq}{\vcentcolon=}

\newcommand{\eqdef}{=\vcentcolon}
\newcommand{\ev}{{\mathrm{ev}}}
\newcommand{\Ext}{{\mathrm{Ext}}}

\newcommand{\GL}{{\mathrm{GL}}}

\newcommand{\Hom}{{\mathrm{Hom}}}

\renewcommand{\Im}{{\mathrm{Im}}}

\newcommand{\id}{{\mathrm{id}}}

\newcommand{\Ker}{{\mathrm{Ker}}}

\newcommand{\lra}{\longrightarrow}

\newcommand{\lto}{\longmapsto}

\DeclareRobustCommand\longtwoheadrightarrow
{\relbar\joinrel\twoheadrightarrow}
\DeclareRobustCommand\longhookrightarrow
{\lhook\joinrel\longrightarrow}

\renewcommand{\mod}{{\,\,\mathrm{mod}\,\,}}

\newcommand{\mto}{\mapsto}

\newcommand{\ov}[1]{\overline{#1}}

\newcommand{\pr}{{\mathrm{pr}}}
\newcommand{\Proj}{{\mathrm{Proj}}}
\newcommand{\PGL}{{\mathrm{PGL}}}
\newcommand{\Pic}{{\mathrm{Pic}}}

\newcommand{\ra}{\rightarrow}

\newcommand{\res}[1]{{\!\,\mid_{#1}}}
\newcommand{\rig}{{\mathrm{rig}}}

\newcommand{\Tr}{{\mathrm{Tr}}}

\newcommand{\ul}[1]{{\underline{#1}}}

\let\originalmiddle\middle
\renewcommand{\middle}[1]{\,\originalmiddle#1\,}

\newcommand{\BF}{{\mathbb {F}}}

\newcommand{\BN}{{\mathbb {N}}}

\newcommand{\BP}{{\mathbb {P}}}

\newcommand{\BR}{{\mathbb {R}}}

\newcommand{\BZ}{{\mathbb {Z}}}

\newcommand{\CB}{{\mathcal {B}}}

\newcommand{\CL}{{\mathcal {L}}}
\newcommand{\CM}{{\mathcal {M}}}

\newcommand{\CO}{{\mathcal {O}}}
\newcommand{\CP}{{\mathcal {P}}}

\newcommand{\CS}{{\mathcal {S}}}
\newcommand{\CT}{{\mathcal {T}}}
\newcommand{\CU}{{\mathcal {U}}}
\newcommand{\CV}{{\mathcal {V}}}

\newcommand{\bA}{{\mathrm{\bf A}}}

\newcommand{\bE}{{\mathrm{\bf E}}}

\newcommand{\bV}{{\mathrm{\bf V}}}

\title{Classification of Equivariant Line Bundles on the Drinfeld Upper Half Plane}
\author{Georg Linden}
\date{March 31, 2026}


\newcommand{\Prof}{{\mathrm{Prof}}}

\newcommand{\unif}{{\pi}}
\newcommand{\1}{{\mathbf{1}}}

\newcommand{\Cur}{F}
\newcommand{\conjelt}{{s}}

\begin{document}

\begin{abstract}
	We explicitly determine the group of isomorphism classes of equivariant line bundles on the non-archimedean Drinfeld upper half plane for $\GL_2(F)$, for its subgroups of matrices whose determinant has even (respectively trivial) valuation, and for $\GL_2(\CO_F)$.
	Our results extend a recent classification of torsion equivariant line bundles with connection due to Ardakov and Wadsley, but we use a different approach.
	A crucial ingredient is a construction due to Van der Put which relates invertible analytic functions on the Drinfeld upper half plane to currents on the Bruhat--Tits tree.
	Another tool we use is condensed group cohomology.
\end{abstract}
\maketitle
\tableofcontents

\section{Introduction}

Let $F$ be a non-archimedean local field with ring of integers $\CO_F$, and let $K$ be a complete extension of $F$ contained in the completion of the algebraic closure of $F$.
The Drinfeld upper half plane associated to $F$ is the rigid $F$-analytic space 
\begin{equation*}
	\Omega_F = \BP_F^1 \setminus \BP^1 (F) .
\end{equation*}
We let $\Omega$ denote its base change to $K$.
The space $\Omega$ carries a natural action of $G \defeq \GL_2(F)$ by restricting the action on the projective line.

For $p$-adic $F$ and when $K$ contains the quadratic unramified extension of $F$, Ardakov and Wadsley \cite{ArdakovWadsley23EquivLineBun} explicitly described the torsion subgroup of $\mathrm{PicCon}^{G^0}(\Omega)$, i.e.\ of the group of isomorphism classes of $G^0$-equivariant line bundles with integrable connection\footnote{Such equivariant line bundles (or more generally vector bundles) with integrable connection on $\Omega$ are closely related to the tower of $G$-equivariant \'etale coverings of $\Omega$ defined by Drinfeld \cite{Drinfeld76CovpAdicSymmDom} when $K$ is $\breve{F}$, i.e.\ the completion of the maximally unramified extension of $F$. See for example the introduction of \cite{ArdakovWadsley23EquivLineBun} for details.} on $\Omega$.
Here, $G^0$ is the subgroup of matrices $g \in G$ whose determinant has trivial valuation.

The space $\Omega_F$ also admits a formal model $\widehat{\Omega}$ due to Deligne.
Junger \cite{Junger23CohModpFibresEquivDrinfeld} classified $G$- and $G^{(2)}$-equivariant line bundles on the base change $\widehat{\Omega}_{\CO_{\breve{F}}}$ of $\widehat{\Omega}$ where $G^{(2)}$ is the subgroup of matrices $g \in G$ whose determinant has even valuation.
\\

In this article our aim lies between these two directions.
We determine the structure of the group $\Pic^G(\Omega)$ of isomorphism classes of $G$-equivariant\footnote{Like Ardakov and Wadsley, we implicitly mean that the equivariant structure is continuous in a natural way as well, see \Cref{Def - Continuous equivariant module}.} line bundles on $\Omega$ as well as those of $\Pic^{G^{(2)}}(\Omega)$, $\Pic^{G^0}(\Omega)$ and $\Pic^{G_0}(\Omega)$ where $G_0 \defeq \GL_2(\CO_F)$.
Here, $F$ and $K$ are as in the beginning.

To state our results we fix some notations and recall two families of such equivariant line bundles:
For $n\in \BZ$, we let $\CO(n)$ denote the $n$-th twisting sheaf\footnote{See \Cref{Example - Cohomology class of twisting sheaf} for our convention concerning the $G$-equivariant structure on $\CO(n)$.} on $\BP_K^1$ restricted to $\Omega$.
Furthermore, for any continuous character $\chi\colon G \ra K^\times$ (or for one of $G^{(2)}$, $G^0$ or $G_0$), we obtain an equivariant line bundle $\CO_\chi$ on $\Omega$ by twisting the canonical equivariant structure of the structure sheaf by $\chi$.

Let $q$ be the cardinality of the residue field of $F$.
Let $\mu_{q-1}(F)$ denote the $(q-1)$-st roots of unity of $F$.
Moreover, we write $\CO_F^{\times \times}$ for the principal units of $F$, i.e.\ those elements of $\CO_F$ which are congruent to $1$ modulo its maximal ideal, and define $\CO_K^{\times \times}$ analogously.
For $x\in F^\times$, we let $\widehat{x}$ (resp.\ $\langle x \rangle$, $\widetilde{x}$) denote its image under the projection to $\mu_{q-1}(F)$ (resp.\ to $\CO_F^{\times \times}$, $\unif^\BZ \times \CO_F^{\times \times}$).

\begin{theoremalphabetical}\label{Thm - A}
	There are isomorphisms of abelian groups
	\begin{equation*}
	\arraycolsep=1.4pt
	\def\arraystretch{1.2}
	\begin{array}{rrccrrl}
		\text{(i)}\quad &\BZ &\oplus &&&\Hom_\cont(F^\times, K^\times) &\,\overset{\sim}{\lra}\, \Pic^{G}(\Omega) \,, \hspace{20cm}	\\
		&&&&&{(n, \chi) } &\,\lto\, \big[ \CO(n) \otimes \CO_{\chi \circ \det} \big] \,,	 \\[9pt]
		\text{(ii)}\quad&\BZ &\oplus & \BZ/ (q^2 - 1) \, \BZ  &\oplus &\Hom_\cont\big( \unif^{2\BZ} \times \CO_F^{\times \times}, K^{\times} \big) &\,\overset{\sim}{\lra}\, \Pic^{G^{(2)}}(\Omega) \,, \\
		&&&&&{(n,k, \chi)} &\,\lto\, \big[ \CO(n) \otimes \CL^{\otimes k} \otimes \CO_{\chi \circ\widetilde{ \det}} \big] \,, \\[9pt]
		\text{(iii)}\quad&\BZ &\oplus & \BZ/ (q^2 - 1) \, \BZ  &\oplus &\Hom_\cont\big( \CO_F^{\times \times}, \CO_K^{\times \times} \big) &\,\overset{\sim}{\lra}\, \Pic^{G^0}(\Omega) \,, \\
		&&&&&{(n,k, \chi)} &\,\lto\, \big[ \CO(n) \otimes \CL^{\otimes k} \otimes \CO_{\chi \circ \langle \det \rangle} \big] \,,	 \\[9pt]
		\text{(iv)}\quad&\BZ_p &\oplus &\BZ/ (q^2 - 1) \, \BZ  &\oplus &\Hom_\cont \big( G_0 , \CO_K^{\times \times} \big) &\,\overset{\sim}{\lra} \,\Pic^{G_0}(\Omega) \,, \\
		&&&&& {(\lambda, k , \psi) } &\,\lto\, \big[ \CL \otimes \CO(1) \big]^{\otimes \lambda} \otimes \big[ \CL^{\otimes k} \otimes \CO_\psi \big] .	
	\end{array}
	\end{equation*}
	Here, $\CL$ is a certain $G^{(2)}$-equivariant line bundle which satisfies $[\CL]^{\otimes q+1} = [\CO_{\widehat{\det} {}^{-1}}]$.
\end{theoremalphabetical}

\begin{remarksnonumber}
	\begin{altenumerate}
		\item
			In fact, we prove something stronger than what is stated above.
			We identify $\Pic^G(\Omega)$ (and analogously for $G^{(2)}$, $G^0$ and $G_0$) with the underlying abelian group of the condensed group cohomology $\ul{H}^1 \big( G, \CO^\times (\Omega) \big)$.
			We establish isomorphisms on the level of condensed abelian groups between these cohomology groups and the right hand side terms of (i) -- (iv) in the Theorems \ref{Thm - Group cohomology for G_0}, \ref{Thm - Group cohomology of G^0} and \ref{Thm - Group cohomology for G}.
			From those we deduce the isomorphisms of \Cref{Thm - A}.
	
		\item 
			The above isomorphisms are compatible with respect to the canonical forgetful homomorphism on the right hand side and the following homomorphisms on the left hand side:
			From (i) to (ii)
			\begin{equation*}
				(n,\chi) \lto \big(n ,\, (q+1) \,l ,\, \chi\res{\unif^\BZ \times \CO_F^{\times \times}} \big) \,,
			\end{equation*}
			where $l \in \BZ/ (q-1) \, \BZ$ such that $\chi(\zeta) = \zeta^l$ for $\zeta \in \mu_{q-1}(F)$.
			From (ii) to (iii)
			\begin{equation*}
				(n,k, \chi) \lto \big( n , k , \chi\res{\CO_F^{\times \times}} \big) \,,
			\end{equation*}
			and from (iii) to (iv)
			\begin{equation*} 
				(n,k,\chi)\lto \big(n,\,k-n ,\, \chi \circ \langle \det \rangle \big) . 
			\end{equation*}
				
	\end{altenumerate}
\end{remarksnonumber}
\medskip

Our description of $\Pic^{G^0} (\Omega)$ relates to the one of Ardakov and Wadsley \cite{ArdakovWadsley23EquivLineBun} for $\mathrm{PicCon}^{G^0} (\Omega)_\mathrm{tors}$ in their setting as follows:
They construct an isomorphism of abelian groups
\begin{equation*}
	\mathrm{PicCon}^{G^0} (\Omega)_\mathrm{tors} \overset{\sim}{\lra} \Hom \big( \CO_D^\times , K^\times \big)_\mathrm{tors}
\end{equation*}
where $D$ is the quaternion division algebra over $F$ and $\CO_D$ its maximal order.
This isomorphism depends on the choice of a point $z \in \Omega_F(L)$ and an embedding $L \hookrightarrow D$ of $F$-algebras.
However, after quotienting out natural actions of $G$ and $D^\times$ on both sides respectively the resulting bijection is independent of these choices.

Forgetting the integrable connection induces a group homomorphism
\begin{equation*}
	\mathrm{PicCon}^{G^0}(\Omega) \lra \Pic^{G^0}(\Omega) .
\end{equation*}
It follows from a result of Taylor \cite{Taylor25CatLubinTateDrinfeldBun} that this map yields an isomorphism between the respective torsion subgroups.
In this way, \Cref{Thm - A} yields an alternative description of $\mathrm{PicCon}^{G^0} (\Omega)_\mathrm{tors}$.
We moreover recover the statement that $\mathrm{PicCon}^{G^0} (\Omega)_\mathrm{tors}$ and $\Hom \big( \CO_D^\times , K^\times \big)_\mathrm{tors}$ are isomorphic, see \Cref{Rmk - Connection to Ardakov-Wadsley}.
\\

The description of \cite{ArdakovWadsley23EquivLineBun} for $\mathrm{PicCon}^{G^0}(\Omega)_\mathrm{tors}$ in terms of representations of $\CO_D^\times$ has been generalised by Taylor \cite{Taylor25EquivVecBun} to Drinfeld symmetric spaces and vector bundles of arbitrary dimension and rank (although in a way which is not as explicit).
We expect that \Cref{Thm - A} and our proof also can be generalised to Drinfeld symmetric spaces of higher dimension.

\subsection*{Overview of the Proof}

In the following, let $H$ be one of $G$, $G^{(2)}$, $G^0$ or $G_0$. 
Since every line bundle on $\Omega$ is trivial, $\Pic^H (\Omega)$ is isomorphic to the group cohomology $H^1 \big( H, \CO^\times(\Omega) \big)$ of continuous\footnote{The continuity condition on the $H$-equivariant structures translates to considering continuous group cohomology.} cochains (\Cref{Cor - Equivariant Picard group and group cohomology on DHP}).
For example, the isomorphism class $[\CO(1)]$ corresponds to the group cohomology class $[j]$ represented by the $1$-cocycle (\Cref{Example - Cohomology class of twisting sheaf})
\begin{equation*}
	j \colon G \lra \CO^\times(\Omega) \,,\quad g \lto \bigg[ \, \Omega \ni [z_0 \!:\! z_1] \mto a\, \frac{z_1}{z_0} - c \, \bigg] \,, \quad\text{where $g = \bigg( \begin{matrix} a & b \\ c & d \end{matrix} \bigg)$.}
\end{equation*}

To handle the occurring homological algebra for topological groups we found it convenient to employ the language of condensed mathematics due to Clausen and Scholze.
More precisely, we pass to the condensed $\ul{H}$-module associated to $\CO^\times(\Omega)$ and consider its condensed group cohomology $\ul{H}^i \big( H, \CO^\times(\Omega) \big)$ (\Cref{Def - Condensed Group Cohomology}).
Because the topological abelian group $\CO^\times(\Omega)$ -- and in fact all topological groups relevant here -- is well behaved (meaning separable and completely ultrametrisable, see \Cref{Lemma - Group of invertible functions is Polish}), we can recover $H^1 \big( H, \CO^\times(\Omega) \big)$ as the underlying abelian group of $\ul{H}^1\big( H, \CO^\times(\Omega) \big)$ (\Cref{Prop - Comparison condensed and continuous group cohomology}).
\\

Our first crucial ingredient is a construction due to Van der Put \cite{vanderPut92DiscGrpsMumfordCurv}.
Let $\CT$ denote the Bruhat--Tits tree of $\PGL_2(F)$ with vertices $\bV$, undirected edges $\bE$ and directed edges (``arrows'') $\bA$. 
Van der Put defines the $G$-module  $\Cur(\bA, \BZ)$ of currents (harmonic cochains) on $\CT$ as the abelian group of functions $\varphi \colon \bA \ra \BZ$ such that
\begin{altenumeratelevel2}
	\item 
		$\varphi((v,w)) = - \varphi((w,v))$, for all $(v,w)\in \bA$, and 
		
	\item 
		$\sum_{(v,w)\in \bA} \varphi((v,w)) = 0 $, for all $v\in \bV$, summing over the edges originating from $v$.
\end{altenumeratelevel2}
He then constructs a group homomorphism $P \colon \CO^\times(\Omega) \ra \Cur(\bA, \BZ)$ which fits into a short exact $G$-equivariant sequence
\begin{equation*}\label{Eq - Introduction 1}
	1 \lra K^\times \lra \CO^\times(\Omega) \overset{P}{\lra} \Cur(\bA, \BZ) \lra 0 .  \tag{$\ast$}
\end{equation*}
We verify that this sequence is a strictly exact sequence of topological abelian groups with respect to a natural topology on $\Cur(\bA, \BZ) $ (\Cref{Cor - Solid Van der Put sequence}).

Via the long exact sequence of (condensed) group cohomology associated to (\ref{Eq - Introduction 1}) we can relate $\ul{H}^1\big( H, \CO^\times(\Omega) \big)$ to the group cohomology of $\Cur(\bA,\BZ)$.
The latter is more accessible to computation.
For convenience, instead of $\Cur(\bA,\BZ)$ we work with $\Cur(\bE,\BZ)$, i.e.\ the abelian group of functions on $\bE$ which satisfy the analogue of property (2):
We have $\Cur(\bE,\BZ) \cong \Cur(\bA, \BZ)$ as $G^0$-modules (\Cref{Lemma - Isomorphism between currents on directed and undirected edges}).
With a homomorphism $\Sigma$ which expresses this condition (2), we obtain a short strictly exact $G^0$-equivariant sequence
\begin{equation*}\label{Eq - Introduction 2}
	0 \lra \Cur(\bE, \BZ) \lra C(\bE, \BZ) \overset{\Sigma}{\lra}  C(\bV, \BZ) \lra 0 . \tag{$\dagger$}
\end{equation*}
Here, $C(\bE, \BZ)$ (resp.\ $C(\bV, \BZ)$) denotes the group of $\BZ$-valued functions on $\bE$ (resp.\ on $\bV$).

Using the long exact sequence associated to (\ref{Eq - Introduction 2}) we deduce the vanishing of the $G^0$- and $G_0$-invariants of $\Cur(\bA,\BZ)$ as well as (\Cref{Prop - Cohomology of currents} and \Cref{Cor - Cohomology of currents for maximal compact subgroup})
\begin{equation*}
\arraycolsep=1.4pt
\def\arraystretch{1.2}
\begin{array}{rcccl}
	\ul{H}^1 \big( G^0 , \Cur(\bA,\BZ) \big) &\cong &\textstyle \frac{1}{q-1} \BZ &\oplus &\BZ/ (q+1)\,\BZ \,,\\
	\ul{H}^1 \big( G_0, \Cur(\bA,\BZ) \big) &\cong &\BZ_p &\oplus &\BZ/ (q+1)\, \BZ \,.
\end{array}
\end{equation*}
From the long exact sequence associated to (\ref{Eq - Introduction 1}) we thus extract the exact sequences
\begin{equation*}\label{Eq - Introduction 3}
\arraycolsep=1.4pt
\def\arraystretch{1.2}
\begin{array}{cccccccl}
	1 \lra &\ul{H}^1 ( G^0 ,K^\times ) &\lra& \ul{H}^1 \big( G^0, \CO^\times(\Omega) \big) &\overset{P_\ast}{\lra}& \textstyle \frac{1}{q-1} \BZ &\oplus &\BZ/ (q+1)\,\BZ \,, \\
	1 \lra &\ul{H}^1 ( G_0 ,K^\times ) &\lra& \ul{H}^1 \big( G_0, \CO^\times(\Omega) \big) &\overset{P_\ast}{\lra}& \BZ_p &\oplus &\BZ/ (q+1)\, \BZ 
\end{array}
\tag{$\ddagger$}
\end{equation*}
which give a ``bound'' on the shape of $\ul{H}^1 \big( G^0, \CO^\times(\Omega) \big)$ and $\ul{H}^1 \big( G_0, \CO^\times(\Omega) \big)$.
It remains to determine the precise image of $P_\ast$ and the class of the extension induced by (\ref{Eq - Introduction 3}) in both cases.
\\

To this end, the next step is to construct some explicit classes in $H^1 \big( G_0, \CO^\times(\Omega) \big)$.
The basis here is a continuous $1$-cocycle $\alpha \colon G_0 \ra \CO^\times(\Omega)$ which essentially was defined by Ardakov and Wadsley in \cite{ArdakovWadsley23EquivLineBun} (see \Cref{Prop - Prime to p torsion of group cohomology for affinoid} and \Cref{Def - Certain torsion cocycle} for a recapitulation).
The class $[\alpha]$ generates the prime-to-$p$-torsion subgroup of $H^1 \big( G_0 , \CO^\times(\Omega) \big)$ and corresponds to $[\CL] \in \Pic^{G_0}(\Omega)$ in \Cref{Thm - A}.

We then prove that the condensed subgroup generated by $[j \alpha]$ extends to a direct summand of $\ul{H}^1 \big( G_0, \CO^\times(\Omega) \big)$ which is isomorphic to $\BZ_p$ (\Cref{Cor - Zp-family in group cohomology}).
In particular, in (\ref{Eq - Introduction 3}) the homomorphism $P_\ast$ for $G_0$ is surjective.
At this stage we can deduce the decomposition of \Cref{Thm - A} for $\ul{H}^1 \big( G_0, \CO^\times(\Omega) \big)$ (\Cref{Thm - Group cohomology for G_0}).

To show the decompositions for $G^0$, $G^{(2)}$ and $G$, we analyse which classes of $H^1 \big( G_0, \CO^\times(\Omega) \big)$ can be ``lifted'' to these groups, i.e.\ lie in the image of the canonical forgetful homomorphisms.

The group $G^0$ is the amalgamated free sum
\begin{equation*}
	G^0 = G_0 \ast_{I} {}^\conjelt G_0
\end{equation*}
of $G_0$ and its conjugate ${}^\conjelt G_0$ along the Iwahori subgroup $I = G_0 \cap {}^\conjelt G_0$.
Here, $\conjelt$ is the element $\big( \begin{smallmatrix} 0 & 1\\ \unif & 0 \end{smallmatrix} \big)$ for a fixed uniformiser $\unif$ of $F$.
This allows us to relate the group cohomology of $G^0$ to the one of $G_0$, ${}^\conjelt G_0$ and $I$ via a Mayer--Vietoris-type sequence for condensed group cohomology (\Cref{Cor - Mayer-Vietoris sequence}).

Concretely (see \eqref{Eq - Diagram induced by Mayer-Vietoris}) and together with an explicit computation in $H^1 \big( I , \CO^\times(\Omega'_0) \big)$ (\Cref{Prop - Shape of prime to p torsion class restricted to affinoid over edge}), we deduce that $[\alpha]$ lifts to $H^1 \big( G^0, \CO^\times(\Omega) \big)$ (\Cref{Prop - Lifting of prime-to-p torsion}).
Here, $\Omega'_0$ denotes the affinoid subdomain of $\Omega_K$ lying over the edge of $\CT$ which is stabilised by $I$.
Subsequently, we show the decomposition for $\ul{H}^1 \big( G^0, \CO^\times(\Omega) \big)$ (\Cref{Thm - Group cohomology of G^0}).
We note that in (\ref{Eq - Introduction 3}) the map $P_\ast$ for $G^0$ is not surjective but rather its image is $\BZ \oplus \BZ/ (q+1)\,\BZ$.

Finally, we have semi-direct product decompositions $G= G^0 \rtimes \conjelt^\BZ$ and $G^{(2)}= G^0 \rtimes \conjelt^{2\BZ}$.
Via the associated Hochschild--Serre spectral sequences and knowledge of the conjugation action by $\conjelt$ on $\ul{H}^1 \big( G^0, \CO^\times(\Omega) \big)$ we establish that $[\alpha]$ (and hence $[\CL]$) ``lifts'' to $G^{(2)}$, as well as the decompositions for $\ul{H}^1 \big( G^{(2)}, \CO^\times(\Omega) \big)$ and $\ul{H}^1 \big( G, \CO^\times(\Omega) \big)$ (\Cref{Thm - Group cohomology for G}).

\subsection*{Acknowledgements}

I would like to thank Konstantin Ardakov for his detailed answers to my questions (in particular about \cite{ArdakovWadsley23EquivLineBun}) and his interest in this work. 
Moreover, I would like to thank Vytautas Paškūnas, James Taylor, Simon Wadsley and Yingying Wang for helpful comments and discussions.

This work was done while the author was a member of the Research Training Group \textit{Symmetries and Classifying Spaces – Analytic, Arithmetic and Derived} which is funded by the DFG (Deutsche Forschungsgemeinschaft).

\subsection*{Notation and Conventions}\label{Sect - Notation and conventions}

Throughout, $F$ is a non-archimedean local field with fixed uniformiser $\unif$ and residue field $\BF_q$ of order $q$ and characteristic $p$. 
Let $\CO_F$ denote the ring of integers and $\CO_F^{\times \times} \defeq \{ x \in F \mid \abs{x-1}< 1 \}$ the subgroup of principal units.

Furthermore, $K$ will be a complete extension of $F$ contained in $C$.
Here $C$ is a completion of an algebraic closure of $F$ both of which we fix.
Then $\CO_K$ and $\CO_K^{\times\times}$ are defined analogously to above, and $\kappa$ will denote the residue field of $K$.
We denote the valuation with respect to $\unif$ by $v_\unif \colon F^\times \ra \BZ$ and normalise the absolute value $\abs{\blank}$ of $C$ and its subfields such that $\abs{\unif} = \frac{1}{q}$.

For a rigid $K$-analytic variety $X$, we will denote its structure sheaf simply by $\CO$ if the context allows, and write $\lVert \blank \rVert_X$ for the supremum seminorm on $\CO(X)$.
Regarding the \emph{sheaf of units} $\CO^\times$, for reduced affinoid subdomains $U\subset X$, we always consider $\CO^\times(U)$ as a topological group with the subspace topology $\CO^\times(U) \subset \CO(U)$, see \Cref{Lemma - Group of invertible functions is Polish}.
We let $\CO^{\times\times} \subset \CO^\times$ denote the \emph{sheaf of principal units} (also called \emph{sheaf of small units}) with $\CO^{\times\times}(U)= \{f \in \CO(U) \mid \lVert f - 1 \rVert_U < 1 \}$.

For topological spaces $A$ and $B$, we let $C(A,B)$ denote the space of continuous maps from $A$ to $B$ which we always consider with the compact-open topology.
When $A$ and $B$ are topological groups, we write $\Hom(A,B)$ for the subspace of continuous group homomorphisms.

The category of topological abelian groups is quasi-abelian in the sense of \cite{Schneiders98QuasiAbCat}.
In particular, a continuous homomorphism $\varphi \colon A \ra B$ between topological abelian groups is called \emph{strict} if the induced map $A/\Ker(\varphi) \ra \Im(\varphi)$ is a topological isomorphism with respect to the quotient and subspace topology respectively.
A \emph{short strictly exact sequence} is an algebraically short exact sequence whose homomorphisms are continuous and strict.

The continuous group cohomology we consider is the one defined by Tate in \cite[\S 2]{Tate76RelK2GalCohom} via continuous cochains. 
For a topological group $G$ and a topological $G$-module $M$, we let $C^n(G,M)$ (resp.\ $Z^n(G,M)$, $B^n (G,M)$) denote the subgroup of continuous $n$-cochains (resp.\ $n$-cocycles, $n$-coboundaries), and $H^n(G,M) \defeq Z^n(G,M) / B^n(G,M)$ the $n$-th continuous group cohomology group.

We treat notions for condensed group cohomology in \Cref{Sect - Condensed group cohomology} and only mention the following at this point:
For $G$ and $M$ as above, we let $\ul{H}^n (G, M) $ denote the $n$-th condensed group cohomology of the associated condensed $\ul{G}$-module $\ul{M}$.
Moreover, when $A$ is a topological abelian group, we write $\ul{\Hom}(G, A) \defeq \ul{H}^1 (G, A) $ where $A$ is endowed with the trivial $G$-action.
In the situation occurring in the following (namely locally profinite $G$, and ultrametrisable Polish $M$ and $A$) the condensed abelian groups $\ul{H}^n (G, M) $ and $\ul{\Hom}(G, A) $ are solid.
In addition, their underlying abelian groups satisfy
\begin{equation*}
	\ul{H}^n (G, M)(\ast) = H^n (G, M) \qquad\text{and}\qquad \ul{\Hom}(G, A)(\ast) = \Hom (G, A) 
\end{equation*}
and $\ul{\Hom}(G, A)$ agrees with the condensation of $\Hom(G,A)$.
Furthermore, we will use the open mapping theorem in this situation, i.e.\ the fact that every surjective continuous homomorphism between Polish groups is strict.

When $X$ is a set and $S\subset X$ a subset, we let $\1_S$ denote the characteristic function of $S$ with $\1_S(x) = 1$ if $x\in S$ and $\1_S(x)=0$ if $x\in X\setminus S$.

\section{Currents on the Bruhat--Tits Tree}

Recall that $G \defeq \GL_2(F)$, the subgroup $G^0$ is the kernel of the continuous group homomorphism $G \ra \BZ$, $g \mto v_\unif\big( \det(g) \big)$, while $G^{(2)}$ is the preimage of $2  \BZ$ under this homomorphism.
Moreover, we have the compact open subgroups
\begin{equation*}
	G_0 \defeq \GL_2(\CO_F) \qquad\text{and} \qquad {}^\conjelt G_0 \defeq \conjelt G_0 \conjelt^{-1} \quad\text{where $\conjelt \defeq \bigg( \begin{matrix} 0 & 1 \\ \unif & 0 \end{matrix} \bigg)$}
\end{equation*}
of $G^0$.
Their intersection
\begin{equation*}
	I \defeq G_0 \cap {}^\conjelt G_0 = \bigg\{ \bigg( \begin{matrix} a & b \\ c & d \end{matrix} \bigg) \in G_0   \,\bigg\vert\, c \equiv 0 \mod (\unif) \bigg\}
\end{equation*}
is the \emph{Iwahori subgroup} of $G_0$.

For a topological $G$-module $M$ and a closed subgroup $H$ of $G$, the conjugation by $\conjelt$ induces a topological isomorphism $\conjelt^\ast$ of complexes with
\begin{equation*}
	\conjelt^\ast \colon C^n( H, M) \lra C^n ( {}^\conjelt H , M ) \,,\quad \varphi \lto \big[ (g_1,\ldots,g_n) \mto \conjelt\cdot \varphi( \conjelt^{-1} g_1 \conjelt , \ldots, \conjelt^{-1} g_n \conjelt ) \big] .
\end{equation*}
If $M$ is ultrametrisable Polish, it follows from \Cref{Prop - Comparison condensed and continuous group cohomology} (ii) that $\conjelt^\ast$ gives rise to isomorphisms of condensed group cohomology.

\begin{definition}\label{Def - Conjugation action on cohomology}
	We denote the induced isomorphisms by
	\begin{equation*}
		\conjelt^\ast \colon \ul{H}^n \big( H , M \big) \lra \ul{H}^n \big( {}^\conjelt H , M \big) .
	\end{equation*}
	The underlying isomorphism of abelian groups is given by $[\varphi] \mto \conjelt^\ast [\varphi] \defeq [\conjelt^\ast \varphi]$.
\end{definition}

We remark that ${}^\conjelt G^0 = G^0$ and ${}^\conjelt I = I$ so that $\conjelt^\ast$ is an endomorphism in these cases. 
Moreover, if the centre of $G$ acts trivially on $M$, we have $(\conjelt^\ast)^2 = \id$ because $\conjelt^2$ lies in the centre.

\subsection{The Bruhat--Tits Tree}

We recall the definition of the Bruhat--Tits tree $\CT$ of $\PGL_2(F)$ following \cite[Sect.\ I.1]{BoutotCarayol91ThmCerednikDrinfeld}.

A \emph{lattice} in $F^2$ is a free $\CO_F$-submodule of rank $2$.
Two lattices $M$ and $M'$ are \emph{homothetic} if there exists $\lambda \in F^\times$ such that $M' = \lambda M$. 
We let $[M]$ denote the homothety class of a lattice $M$.

The \emph{Bruhat--Tits tree of $\PGL_2(F)$} is the graph $\CT$ with vertices $\bV$ defined to be the set of homothety classes of lattices in $F^2$.
Two vertices are joined by an edge if and only if there exist lattices $M$ and $M'$ representing them such that $\unif M \subsetneq M' \subsetneq M$.
We let $\bE$ denote the set of (undirected) edges and write $\bA$ for the set directed edges (``arrows'') of $\CT$. 
Then $\CT$ indeed is a tree and each vertex has exactly $q+1$ edges incident on it; they are in bijection with the lines in $M/\unif M$, i.e.\ with $\BP^1(\BF_q)$.

On $F^2$ we consider the $G$-action by matrix multiplication on column vectors.
%\begin{equation*}
%	g \cdot \left( \begin{matrix}	x_0 \\ x_1 	\end{matrix} \right) \defeq \left(\begin{matrix} a x_0 + b x_1 \\ c x_0 + d x_1 	\end{matrix} \right) \, \quad\text{, for $g = \left( \begin{matrix} a & b \\ c& d	\end{matrix} \right) \in G$, $\left( \begin{matrix} x_0 \\x_1 	\end{matrix} \right) \in F^2$.}
%\end{equation*}
This induces an action of $G$ on the set of homothety classes of lattices in $F^2$ and therefore on the tree $\CT$.
The centre of $G$ acts trivially.

We distinguish the family of vertices\footnote{The apartment of the $(v_i)_{i\in \BZ}$ is the one stabilised by the diagonal matrices of $G$.} and edges
\begin{equation*}
	v_i \defeq [ \CO_F \oplus \unif^i  \CO_F ]  \qquad\text{and}\qquad e_i \defeq \{v_i,v_{i+1}\}\,, \qquad\text{for $i \in \BZ$.}
\end{equation*}
Then $\conjelt \cdot v_i = v_{-i+1}$ and the stabiliser of $v_0$ (resp.\ $v_1$) is the maximal compact open subgroup $G_0$ (resp.\ ${}^\conjelt G_0$).

For $n\in \BN$, let $\CT_n = (\bV_n, \bE_n)$ denote the finite subtree of $\CT$ whose vertices $v$ are precisely the ones with $d(v,v_0) \leq n$, i.e.\ with distance less or equal to $n$ from $v_0$.
Moreover, let $\CT'_n = (\bV'_n, \bE'_n)$ be the finite subtree of $\CT$ whose vertices $v$ satisfy $d(v, v_0) \leq n $ or $d(v,v_1) \leq n $.
Then $\CT'_n = \CT_{n+1} \cap \conjelt \CT_{n+1} = \CT_n \cup \conjelt\CT_n$.

\begin{lemma}\label{Lemma - Action on finite subtree}
	\begin{altenumerate}
		\item
		The stabiliser in $G^0$ of $v_0$ (resp.\ $e_0$) is equal to $G_0$ (resp.\ $I$).
		Moreover, $\bV = G^0 \cdot v_0 \cup G^0 \cdot v_1$ is a decomposition into $G^0$-orbits, and $G^0$ acts transitively on $\bE$.
		
		\item 
		The subtree $\CT_n$ (resp.\ $\CT'_n$) is stable under the action of $G_0$ (resp.\ $I$).
		
		\item
		We have the following decompositions into $G_0$-orbits (resp.\ $I$-orbits)
		\begin{align*}
			\bV_n = \bigcup_{i=0}^n  G_0\cdot v_i \,,			\qquad	\bE_n = \bigcup_{i=0}^{n-1} G_0 \cdot e_i \,, \qquad
			\bV'_n = \bigcup_{i=-n}^{n+1}  I \cdot v_i			\,,\qquad	\bE'_n = \bigcup_{i=-n}^{n} I \cdot e_i .
		\end{align*}
		
		\item 
		The congruence subgroup $G_n \defeq 1 + \unif^n M_2 (\CO_F)$ acts trivially on $\CT_n$ and on $\CT'_{n-1}$.
	\end{altenumerate}
\end{lemma}
\begin{proof}
	The statements of (i) follow from II.1.3 Lemma 1 and II.1.4 Thm.\ 2 of \cite{Serre80Trees}.

	The action of $G$ on $\CT$ preserves the distance between vertices.
	Furthermore, there is a $G$-equivariant bijection between $\BP^1(F)$ and infinite, non-back\-track\-ing sequences of adjacent vertices originating from $v_0$, see \cite[Sect.\ II.1.1]{Serre80Trees}.
	By construction this bijection descends to a $G_0$-equivariant bijection between $\BP^1(\CO_F/\unif^i)$ and the vertices whose distance to $v_0$ is exactly $i$.
	From this one deduces the remaining assertions.
\end{proof}

Let $\bA_{\leq n}$ (resp.\ $\bA'_{\leq n}$) denote the subset of directed edges of $\CT$ which originate from a vertex of $\CT_n$ but possibly end in $\CT_{n+1}$ (resp.\ originate from a vertex of $\CT'_n$ but possibly end in $\CT'_{n+1}$).
We consider several modules consisting of functions arising from $\CT$ and the above subtrees.

\begin{definition}\label{Def - Cochains}
	The topological abelian group $C(\bA, \BZ)$ (where $\bA$ and $\BZ$ are endowed with the discrete topology) becomes a $G$-module via
	\begin{equation*}
		(g \cdot \varphi) ( a ) \defeq \varphi ( g^{-1} \cdot a ) \,,\qquad\text{for $g \in G$, $\varphi \in C(\bA,\BZ)$, $a\in \bA$.}
	\end{equation*}
	We call $C(\bA, \BZ)$ the group of \emph{cochains} on $\CT$.
	Analogously we define $C(\bE, \BZ)$ and $C(\bV, \BZ)$.

	In this way, we also obtain the $G_0$-modules $C(\bA_{\leq n}, \BZ)$, $C(\bE_n, \BZ)$ and $C(\bV_n, \BZ)$, and the $I$-modules $C(\bA'_{\leq n}, \BZ)$, $C(\bE'_n, \BZ)$ and $C(\bV'_n, \BZ)$. 
\end{definition}

\begin{lemma}\label{Lemma - Cochains for finite trees are topological modules}
	For $n \in \BN$, the action of $G_0$ (resp.\ of $I$) is continuous on $C(\bA_{\leq n}, \BZ)$, $C(\bE_n, \BZ)$ and $C(\bV_n, \BZ)$ (resp.\ on $C(\bA'_{\leq n}, \BZ)$, $C(\bE'_n, \BZ)$ and $C(\bV'_n, \BZ)$), and these topological modules carry the discrete topology.
\end{lemma}
\begin{proof}
	Since the above spaces of functions all have finite discrete domains and discrete codomains, the compact-open topology on them is the discrete one.
	Moreover, the group actions are continuous because by \Cref{Lemma - Action on finite subtree} (iv) the stabiliser of any element contains the open subgroup $G_{n+2}$.
\end{proof}

\begin{proposition}\label{Prop - Inverse limit description for cochains}
	The group of cochains $C(\bA, \BZ)$ is a Polish $G$-module and
	\begin{equation*}
		C(\bA,\BZ) \overset{\sim}{\lra} \textstyle \varprojlim_{n\in \BN} C(\bA_{\leq n}, \BZ)   \,,\quad \varphi \lto ( \varphi\res{\bA_{\leq n}} )_{n\in \BN} ,
	\end{equation*}
	is an isomorphism of topological $G_0$-modules.
	Similarly, there is an isomorphism of toplogical $I$-modules $C(\bA, \BZ) \cong \varprojlim_{n\in \BN} C(\bA'_{\leq n} , \BZ)$.

	Moreover, the analogous assertions hold for $C(\bE,\BZ)$ and $C(\bV,\BZ)$.
\end{proposition}
\begin{proof}
	For the claimed isomorphisms involving $C(\bA, \BZ)$, we first note that forming $C(\blank, \blank)$ is an internal Hom-functor in the category of compactly generated Hausdorff spaces. 
	Since this internal Hom-functor preserves limits\footnote{But recall that the first entry of the internal Hom-functor is in the opposite category.}, we deduce from $\bA = \bigcup_{n \in \BN} \bA_{\leq n} = \bigcup_{n \in \BN} \bA'_{\leq n}$ that the claimed maps are homeomorphisms.
	In particular, $C(\bA,\BZ)$ is Polish because the $C(\bA_{\leq n}, \BZ)$ are. 
	Moreover, one directly verifies that the map is a homomorphisms of $G_0$-modules (resp.\ $I$-modules).

	It remains to show that the action of $G$ on $C(\bA, \BZ)$ is continuous.
	We first show that $G$ acts by continuous automorphisms.
	Indeed, fix $g\in G$ and let $n\in \BN$. 
	Then the finite subtree $g^{-1} \cdot \CT_n$ is contained in $\CT_m$, for some $m\geq n$. 
	We obtain continuous homomorphisms
	\begin{equation*}
		C(\bA_{\leq m}, \BZ) \lra C(\bA_{\leq n}, \BZ)\,,\quad \varphi \lto \big[ a \mto \varphi(g^{-1} \cdot a ) \big] ,
	\end{equation*}
	which induce the automorphism by which $g$ acts on $C(\bA, \BZ)$.

	It now suffices to prove that the action restricted to an open subgroup of $G$ on $C(\bA, \BZ)$ is continuous. 
	But the action of $G_0$ is the inverse limit of the continuous actions on the $C(\bA_{\leq n}, \BZ)$.

	The proofs for $C(\bE,\BZ)$ and $C(\bV, \BZ)$ are completely analogous.	
\end{proof}

\subsection{Currents}

The currents that van der Put defines in \cite[Sect.\ 2]{vanderPut92DiscGrpsMumfordCurv} are a subset of $C(\bA, \BZ)$. 
However, we will also consider functions with domain $\bE$ which satisfy a ``current-like'' property.

\begin{definition}\label{Def - Currents}
	\begin{altenumerate}
		\item 
		A cochain $\varphi$ in $ C(\bA,\BZ)$ (resp.\ in $ C(\bA^{(\prime)}_{\leq n}, \BZ)$) is called a \emph{current} (or \emph{harmonic cochain}) if
		\begin{altenumeratelevel2}
			\item 
			$\varphi( (v,w) ) = -\varphi( (w,v))$, for every directed edge $(v,w) \in \bA$  (resp.\ for all $(v,w) \in \bA^{(\prime)}_{\leq n}$ with $\{v,w\} \in \bE^{(\prime)}_n$), and
			
			\item 
			$\sum_{(v,w)\in \bA} \varphi( (v,w) ) = 0$, for all vertices $v\in \bV$ (resp.\ for all $v\in \bV^{(\prime)}_n$), where one sums over all directed edges originating from $v$.
		\end{altenumeratelevel2}
		The currents on $\CT$ form a $G$-submodule of $C(\bA,\BZ)$ which we denote by $\Cur(\bA,\BZ)$.
		Likewise we denote the $G_0$-submodule of currents on $\CT_n$ by $\Cur(\bA_{\leq n}, \BZ)$, and the $I$-submodule of currents on $\CT'_n$ by $\Cur(\bA'_{\leq n}, \BZ)$.
		
		\item 
			A function $\varphi \in C(\bE,\BZ)$ (resp.\ $\varphi \in C(\bE^{(\prime)}_{n+1}, \BZ)$) is called a \emph{current} (or \emph{harmonic cochain}) if
		\begin{equation*}
			\textstyle \sum_{ \{v,w\} \in \bE} \varphi(\{v,w\}) = 0 \,,\qquad\text{for all $v\in \bV$ (resp.\ for all $v\in \bV^{(\prime)}_{n}$).}
		\end{equation*}
		Here the sum is taken over all edges which have $v$ as one of their endpoints.
		We let $\Cur(\bE, \BZ)$, $\Cur(\bE_{n+1}, \BZ)$ and $\Cur(\bE'_{n+1}, \BZ)$ denote the corresponding $G$-, $G_0$- and $I$-submodules respectively.
				
	\end{altenumerate}
\end{definition}

\begin{remark}\label{Rmk - Inverse limit description for currents}
	Being closed subspaces, all these submodules of currents are Polish.
	Moreover, the isomorphisms from \Cref{Prop - Inverse limit description for cochains} induce isomorphisms
	\begin{alignat*}{3}
		\Cur(\bA, \BZ) &\cong \textstyle \varprojlim_{n\in \BN} \Cur(\bA_{\leq n}, \BZ)  &&\cong \textstyle\varprojlim_{n\in \BN} \Cur(\bA'_{\leq n}, \BZ) , \\
		\Cur(\bE, \BZ) &\cong \textstyle \varprojlim_{n\in \BN} \Cur(\bE_{ n}, \BZ)  &&\cong \textstyle\varprojlim_{n\in \BN} \Cur(\bE'_{ n}, \BZ)
	\end{alignat*}
	of topological $G_0$- respectively $I$-modules.
\end{remark}

\begin{notation}
	The vertices $\bV$ can be partitioned into two classes such that the distance between vertices of the same class is even.
	We say that $v\in \bV$ is \emph{even} if $d(v,v_0)$ is even, and \emph{odd} otherwise, i.e.\ to call the class even which contains $v_0$.
	The subgroup $G^0$ preserves this parity of vertices \cite[II.1.2 Cor.\ to Prop.\ 1]{Serre80Trees}.

	Furthermore, given an edge $e\in \bE$ when we write $e = \{ v_+, v_- \}$, we mean that $v_+$ is the even vertex of $e$ and $v_-$ the odd one.
\end{notation}

\begin{lemma}\label{Lemma - Isomorphism between currents on directed and undirected edges}
	\begin{altenumerate}
		\item 
		There is an isomorphism of topological $G^0$-modules
		\begin{equation*}
			\Cur(\bA,\BZ) \overset{\sim}{\lra} \Cur(\bE, \BZ) \,,\quad \varphi \lto \big[ e= \{v_+,v_-\} \mto \varphi((v_+,v_-))  \big] .
		\end{equation*}
		
		\item 
		For $n\in \BN$, there is an isomorphism of $G_0$-modules (resp.\ $I$-modules)
		\begin{equation*}
			\Cur \big(\bA^{(\prime)}_{\leq n},\BZ \big) \overset{\sim}{\lra} \Cur \big(\bE^{(\prime)}_{n+1}, \BZ \big) \,,\quad \varphi \lto \Bigg[ e = \{v_+,v_-\} \mto \begin{cases}	\varphi((v_+,v_-)) & \text{, if $v_+\in \bV^{(\prime)}_n$,} \\
				-\varphi((v_-,v_+)) & \text{, if $v_+\notin \bV^{(\prime)}_n$,}
			\end{cases} \Bigg] .
		\end{equation*}
		
	\end{altenumerate}
\end{lemma}
\begin{proof}
	For (ii) we consider the homomorphism
	\begin{equation*}
		\Cur(\bE_{n+1},\BZ) \lra \Cur (\bA_{\leq n}, \BZ) \,,\quad \psi \lto \Bigg[  (v,w) \mto \begin{cases}	\psi(\{v,w\}) & \text{, if $v$ is even,} \\
			-\psi(\{v,w\}) & \text{, if $v$ is odd,}
		\end{cases} \Bigg] .
	\end{equation*}
	One verifies that this homomorphism and the one from (ii) are well-defined and inverse to each other. 
	That they are $G_0$-equivariant is a direct computation which uses that $G_0$ preserves the parity of vertices. 
	For $\Cur(\bE'_{n+1},\BZ) \cong \Cur (\bA'_{\leq n}, \BZ)$ one argues analogously.

	The homomorphism in (i) then is the inverse limit of the ones in (ii). 
	Hence it is a topological isomorphism. 
	Using that $G^0$ preserves parity, one computes that it is $G^0$-equivariant. 
\end{proof}

\begin{remark}\label{Rmk - conjugation action on currents on undirected edges}
	The topological $G$-module $\Cur(\bA, \BZ)$ canonically carries an action by the element $\conjelt$.
	On $\Cur(\bE, \BZ)$ we prescribe an $\conjelt$-action via 
	\begin{equation*}
		\conjelt \cdot \varphi \defeq \big[ e \mto - \varphi( \conjelt^{-1} \cdot e ) \big] \,,\qquad\text{for $\varphi \in \Cur(\bE, \BZ)$.}
	\end{equation*}
	Then, the $G^0$-equivariant isomorphism of \Cref{Lemma - Isomorphism between currents on directed and undirected edges} (i) becomes $\conjelt$-equivariant since $\conjelt$ interchanges the parity of vertices.
	Analogous statements hold for an $\conjelt$-action on $\Cur(\bE'_n, \BZ)$ and an isomorphism $\conjelt \colon \Cur(\bE_n, \BZ) \ra \Cur(\conjelt\bE_n, \BZ)$.
\end{remark}

We can realise $\Cur(\bE,\BZ)$ as the kernel of the following homomorphism of $G^0$-modules:
\begin{equation*}
	\Sigma  \colon C(\bE, \BZ) \lra  C(\bV, \BZ) \,,\quad \varphi \lto  \bigg[ v \mto \sum_{ \{v,w\} \in \bE} \varphi(\{v,w\}) \bigg] .
\end{equation*}
Similarly, we define the homomorphism $\Sigma_n \colon C(\bE_{n+1}, \BZ) \ra C(\bV_{n}, \BZ)$ of $G_0$-modules and the homomorphism $\Sigma_n^{\prime} \colon C(\bE_{n+1}^{\prime}, \BZ) \ra C(\bV_{n}^{\prime}, \BZ)$ of $I$-modules so that $\Ker(\Sigma_n) = \Cur(\bE_{n+1}, \BZ)$ (resp.\ $\Ker(\Sigma'_n) = \Cur(\bE'_{n+1}, \BZ)$).
Then $\Sigma$ is the inverse limit of the $\Sigma_n$ and therefore continuous.

\begin{proposition}\label{Prop - Short exact sequence for currents}
	\begin{altenumerate}	
		\item 
		There is a short strictly exact sequence of topological $G^0$-modules
		\begin{equation*}
			0 \lra \Cur(\bE,\BZ) \lra C(\bE,\BZ) \overset{\Sigma}{\lra}  C(\bV,\BZ) \lra 0 .
		\end{equation*}

		\item 
		For $n\in \BN$, there is a short strictly exact sequence of discrete $G_0$-modules (resp.\ $I$-modules)
		\begin{equation*}
			0 \lra \Cur \big(\bE^{(\prime)}_{n+1},\BZ \big) \lra C \big(\bE^{(\prime)}_{n+1},\BZ \big) \overset{ \Sigma^{(\prime)}_n}{\lra} C \big(\bV^{(\prime)}_n,\BZ \big) \lra 0 .
		\end{equation*}

	\end{altenumerate}
\end{proposition}
\begin{proof}
	Once we show that $\Sigma$ is surjective, strictness of $\Sigma$ follows from the open mapping theorem.
	Because the projections from $C(\bV, \BZ)$ to $C(\bV_n, \BZ)$ (resp.\ to $C(\bV'_n, \BZ)$) are surjective, this also implies surjectivity of $\Sigma_n$ (resp.\ of $\Sigma'_n$).
	The remaining assertions are clear then.

	Given $ \eta \in  C(\bV,\BZ)$, we will inductively construct a preimage $\varphi$ under $\Sigma$.
	For the initial step set $\varphi(\{v_0,v_1\})\defeq \eta(v_0)$ and $\varphi(\{v_0,w\}) \defeq 0$, for all other vertices $w$ neighbouring $v_0$, so that $\varphi$ is defined on $\bE_1$.

	For the induction step, we assume that $\varphi$ is defined on $\bE_{n}$ and satisfies 
	\begin{equation*}
		(\ast_n) \qquad\quad	\sum_{ \{v,w\} \in \bE} \varphi(\{v,w\}) = \eta(v)\,,\quad\text{for all $v \in \bV_{n-1}$.}
	\end{equation*}
	Let $v \in \bV_{n} \setminus \bV_{n-1}$. 
	We write $w_0$ for the unique vertex neighbouring $v$ with $w_0 \in \bV_{n-1}$ and $w_1,\ldots,w_q$ for the other neighbouring vertices which necessarily are elements of $\bV_{n+1}\setminus \bV_{n}$. 
	We set
	\begin{equation*}
		\varphi( \{v, w_i \}) \defeq \begin{cases}		\eta(v) - 	 \varphi(\{v, w_0\} )  &\text{, for $i=1$,} \\
			0												&\text{, for $i=2,\ldots,q$.}
			
		\end{cases}
	\end{equation*}
	This procedure applied to all $v\in \bV_{n} \setminus \bV_{n-1}$ extends $\varphi$ to $\bE_{n+1}$ such that $(\ast_{n+1})$ is fulfilled.	
\end{proof}

\begin{corollary}\label{Cor - Mittag-Leffler property for currents}
	For $n\in \BN$, the restriction maps 
	\begin{align*}
		\Cur(\bE_{n+1}, \BZ) \lra  \Cur(\bE_n, \BZ) \qquad\text{and}\qquad \Cur(\bE'_{n+1}, \BZ) \lra \Cur(\bE'_n, \BZ)
	\end{align*} 
	are surjective.
	In particular, the inverse systems $\big( \Cur(\bE_{n+1}, \BZ) \big)_{n\in \BN}$ and $\big( \Cur(\bE'_{n+1}, \BZ) \big)_{n\in \BN}$ satisfy the Mittag--Leffler condition.
\end{corollary}
\begin{proof}
	The construction in the proof of the above proposition shows that for $\eta$ in the kernel of $C(\bV_n, \BZ) \ra C(\bV_{n-1}, \BZ)$, we can find a preimage under $\Sigma_n$ contained in the kernel of $C(\bE_{n+1}, \BZ) \ra C(\bE_{n}, \BZ)$. 
	Since the latter homomorphism also is surjective, the snake lemma applied to the short exact sequences for $\Sigma_n$ and $\Sigma_{n-1}$ shows that the cokernel of $ \Cur(\bE_{n+1}, \BZ) \ra  \Cur(\bE_n, \BZ) $ vanishes.
	For $\Cur(\bE'_{n+1}, \BZ) \ra \Cur(\bE'_n, \BZ)$ the same reasoning is valid.
\end{proof}

\subsection{Group Cohomology of Currents}

We now use the strictly short exact sequences of \Cref{Prop - Short exact sequence for currents} to compute the zeroth and first condensed group cohomology groups with coefficients in $\Cur(\bE, \BZ)$.

\begin{proposition}\label{Prop - Vanishing of group cohomology for cochains}
	For $n\in \BN$, we have
	\begin{align*}
		\ul{H}^1  \big( {G^0}, C(\bE ,\BZ) \big)  = \{0\} \,,\qquad \ul{H}^1  \big( {G_0}, C(\bE_{n+1},\BZ) \big)  = \{0\}  \,,\qquad \ul{H}^1  \big( {I}, C(\bE'_n ,\BZ) \big)  = \{0\}.
	\end{align*}
\end{proposition}
\begin{proof}
	The decompositions into orbits from \Cref{Lemma - Action on finite subtree} (i) and (iii) together with the orbit-stabiliser theorem yield equivariant isomorphisms
	\begin{alignat*}{3}
		C(\bE ,\BZ) \,&= \hspace{7pt}C(G^0 \cdot e_0 , \BZ)\,  &&\cong \hspace{7pt}C \big( G^0 / I, \BZ \big) ,\\
		C(\bE_{n+1}, \BZ) \,&= \hspace{7pt} \bigoplus_{i=0}^{n} \hspace{7pt} C( G_0 \cdot e_i, \BZ) \, &&\cong \hspace{7pt} \bigoplus_{i=0}^{n} \hspace{7pt} C \big( G_0 / \mathrm{Stab}_{G_0}(e_i) , \BZ \big) , \\
		C(\bE'_n, \BZ) \,&=\bigoplus_{i=-n+1}^{n} C(I \cdot e_i , \BZ) \, &&\cong \bigoplus_{i=-n+1}^n C \big( I / \mathrm{Stab}_I (e_i) , \BZ \big) .
	\end{alignat*}
	It thus suffices to consider a topological $H$-modules of the form $C \big(H/ H', \BZ \big)$ where $H$ is any of $G^0$, $G_0$ or $I$, and $H' \subset H$ is a compact open subgroup.
	Indeed, for example $\mathrm{Stab}_{G_0}(e_i) = G_0 \cap \mathrm{Stab}_{G^0}(e_i)$ is compact open in $G_0$ because $\mathrm{Stab}_{G^0}(e_i)$ is a $G^0$-conjugate of $I$.

	\Cref{Lemma - Coinduction and condensation} then implies that there is an isomorphism $\ul{C \big(H/H', \BZ \big) } \cong \mathrm{coind}_{\ul{H'}}^\ul{H} (\ul{\BZ})$ of condensed $\ul{H}$-modules where $H'$ acts trivially on $\BZ$.
	With Shapiro's lemma (\Cref{Prop - Shapiro's lemma}) and \Cref{Rmk - Notation for condensed group cohomology} this gives
	\begin{equation*}
		\ul{H}^1 \big( {H} , C \big(H/H', \BZ \big) \big) \cong \ul{H}^1 ( {H'} , \BZ )  = \ul{\Hom ( H', \BZ )} .
	\end{equation*}
	But since $H'$ is compact, $\Hom ( H', \BZ )$ vanishes. 
\end{proof}

The short exact sequence of \Cref{Prop - Short exact sequence for currents} (i) induces a long exact sequence of condensed group cohomology.
Using \Cref{Prop - Vanishing of group cohomology for cochains} we extract the following exact sequence of solid abelian groups.
\begin{equation}\label{Eq - Exact sequence for group cohomology of currents}
	0 \lra \Cur(\bE,\BZ)^\ul{G^0} \lra C(\bE,\BZ)^\ul{G^0} \xrightarrow{\Sigma^\ul{G^0}} C(\bV,\BZ)^\ul{G^0}  \overset{\delta}{\lra} \ul{H}^1  \big( {G^0}, \Cur(\bE, \BZ) \big) \lra 0. 
\end{equation}
Using \Cref{Lemma - Action on finite subtree} (i) we find that $C(\bV,\BZ)^{G^0} $ is a discrete group and equal to $\BZ \,\1_{G^0 v_0} \oplus \BZ \, \1_{G^0 v_1}$ where $\1_{G^0 v_i}$ denotes the characteristic function of the respective $G^0$-orbit.
We also deduce that $C(\bE, \BZ)^{G^0} = \BZ \, \1_{G^0 e_0}$.
Since every $v\in \bV$ is the endpoint of precisely $(q+1)$ edges, we find that 
\begin{equation*}\label{Eq - Description of Sigma}
	\Sigma(\1_{G^0 e_0}) = (q+1) \big( \1_{G^0 v_0} +\1_{G^0 v_1} \big) .
\end{equation*}
It follows that $F(\bE, \BZ)^{G^0} = \Ker \big( \Sigma^{G^0}  \big) = \{0\}$.
Moreover, we have
\begin{align*}
	\Coker \big( \Sigma^{G^0} \big) \cong  \Big( \BZ \,\1_{G^0 v_0} \oplus \BZ \, \1_{G^0 v_1} \Big) \Big/  (q+1) \big( \1_{G^0 v_0} + \1_{G^0 v_1} \big) \, \BZ &\overset{\sim}{\lra } \textstyle\frac{1}{q-1} \BZ \, \oplus \, \BZ/(q+1)\,\BZ 
\end{align*}
where $ \1_{G^0 v_0} $ is mapped to $\big(\textstyle \frac{-1}{q-1}, 1 \mod (q+1)  \big)$ and $\1_{G^0 v_1} $ to $ \big(\textstyle \frac{1}{q-1}, 0 \mod (q+1)   \big)$ under the second isomorphism.
Since $\Coker \big( \Sigma^{G^0} \big) $ is Polish, \Cref{Lemma - Strictly exact sequence gives exact sequence of condensed modules} and \eqref{Eq - Exact sequence for group cohomology of currents} then yield
\begin{equation*}
	\ul{\Coker \big( \Sigma^{G^0} \big) } \cong C(\bV,\BZ)^\ul{G^0} \big/ \Im \big( \Sigma^\ul{G^0} \big) \overset{\delta}{\overset{\sim}{\lra}} \ul{H}^1  \big( {G^0}, \Cur(\bE, \BZ) \big) .
\end{equation*}
Furthermore, the exact sequence \eqref{Eq - Exact sequence for group cohomology of currents} becomes $\conjelt$-equivariant when $\ul{H}^i \big( {G^0}, \Cur(\bE, \BZ) \big)$ carries the $\conjelt$-action induced from \Cref{Rmk - conjugation action on currents on undirected edges} and $\conjelt$ acts on $C(\bE, \BZ)$ and $C(\bV,\BZ)$ by
\begin{equation}\label{Eq - Modified conjugation action}
	\conjelt \cdot \varphi \defeq \big[ e \mto - \varphi( \conjelt^{-1} \cdot e ) \big]\,, \qquad
	\conjelt \cdot \eta \defeq \big[ v \mto - \eta( \conjelt^{-1} \cdot v ) \big] ,
\end{equation}
for $\varphi \in C(\bE, \BZ)$ and $\eta \in C(\bV, \BZ)$.
This $\conjelt$-action permutes $\1_{G^0 v_0}$ with $- \1_{G^0 v_1}$.
In total we have shown:

\begin{proposition}\label{Prop - Cohomology of currents}
	We have $\ul{H}^0 \big( {G^0 } , \Cur(\bE,\BZ) \big) = \{0\}$ and $\ul{H}^1 \big( {G^0 } , \Cur(\bE,\BZ) \big) $ is a discrete condensed abelian group with
	\begin{align*}
		 \ul{H}^1 \big( {G^0}, \Cur(\bE,\BZ) \big) &\overset{\sim}{\lra} \textstyle\frac{1}{q-1} \BZ \oplus \BZ/(q+1)\,\BZ \qquad\text{where} \\
		  \delta \big(\1_{G^0 v_0} \big) &\lto \big(\textstyle \frac{-1}{q-1}, 1 \mod (q+1) \big) , \\
			\delta \big( \1_{G^0 v_1} \big) &\lto \big(\textstyle \frac{1}{q-1}, 0  \mod (q+1) \big) 
	\end{align*}
	on the underlying abelian groups.
	Under this isomorphism the automorphism $\conjelt^\ast$ becomes $\big(\textstyle \frac{1}{q-1}, 0 \big) \mto \big(\textstyle \frac{1}{q-1}, -1 \big)$, $\big(0, 1 \big) \mto \big( 0, -1 \big)$ on the right hand side.
\end{proposition}

To compute the zeroth and first condensed group cohomology of the $G_0$-module $\Cur(\bE_{n+1}, \BZ)$ we can proceed similarly.

\begin{proposition}\label{Prop - Cohomology of currents on finite subtree}
	For $n\in \BN$, we have $\ul{H}^0 \big( {G_0} , \Cur(\bE_{n+1},\BZ) \big)  = \{0\}$ and $\ul{H}^1 \big( {G_0} , \Cur(\bE_{n+1},\BZ) \big)$ is a discrete condensed abelian group with
	\begin{align*}
		 \ul{H}^1 \big( {G_0}, \Cur(\bE_{n+1},\BZ) \big)	\overset{\sim}{\lra} \BZ/ q^n (q+1) \, \BZ \,,\quad \delta_n(\1_{G_0 v_0}) \lto 1 \mod q^n(q+1)  .
	\end{align*}
	Together with the canonical maps these isomorphisms fit into commutative diagrams
	\begin{equation*}
		\begin{tikzcd}[row sep=scriptsize]
			\ul{H}^1 \big(  {G_0}, \Cur(\bE_{n+2},\BZ) \big) \ar[r, "\sim"] \ar[d] &  \BZ/ q^{n+1} (q+1) \, \BZ \ar[d, two heads] \\
			\ul{H}^1 \big(  {G_0}, \Cur(\bE_{n+1},\BZ) \big) \ar[r, "\sim"] &	 \BZ/ q^{n} (q+1) \, \BZ  .
		\end{tikzcd}
	\end{equation*}
\end{proposition}
\begin{proof}
	Here we use the short strictly exact sequence of $G_0$-modules from \Cref{Prop - Short exact sequence for currents} (ii).
	It gives rise to the exact sequence of solid abelian groups (using \Cref{Prop - Vanishing of group cohomology for cochains})
	\begin{equation*}\label{Eq - Long exact sequence for currents on finite subtree}
		0 \lra \Cur(\bE_{n+1}, \BZ)^\ul{G_0} \lra C (\bE_{n+1}, \BZ)^\ul{G_0} \overset{\Sigma_n^\ul{G_0}}{\lra} C (\bV_{n}, \BZ)^\ul{G_0} \overset{\delta_n}{\lra} \ul{H}^1 \big( {G_0}, \Cur(\bE_{n+1}, \BZ) \big) \lra 0 .
	\end{equation*}
	Moreover, we have
	\begin{align*}
		C (\bE_{n+1}, \BZ)^{G_0} \cong \bigoplus_{i=0}^n \BZ \, \1_{G_0 e_i} \qquad\text{and}\qquad C (\bV_{n}, \BZ)^{G_0} \cong \bigoplus_{i=0}^n \BZ \, \1_{G_0 v_i} .
	\end{align*}
	With respect to these bases, $\Sigma_n^{G_0}$ is given by
	\begin{equation*}\label{Eq - Description of Sigma_n}
		\Sigma_n^{G_0} (\1_{G_0 e_i}) 
		= \begin{cases}		(q+1) \, \1_{G_0 v_0} + \1_{G_0 v_1} &		 \text{, for $i=0$,}				\\
			q \, \1_{G_0 v_i} + \1_{G_0 v_{i+1}} & 	\text{, for $i=1,\ldots,n-1$,}	\\
			q \, \1_{G_0 v_n}									&		\text{, for $i=n$.}
		\end{cases}
	\end{equation*}
	This implies that modulo $\Im\big(\Sigma_n^{G_0} \big)$
	\begin{equation}\label{Eq - Relations module the image of Sigma_n}
		\1_{G_0 v_i} \equiv (-1)^i q^{i-1} (q+1) \, \1_{G_0 v_0} \quad\text{, for $i=1,\ldots,n$,}\quad\text{and}\qquad (-1)^n q^n (q+1) \, \1_{G_0 v_0} \equiv 0 .
	\end{equation}
	From this we deduce the claims as before.
\end{proof}

\begin{corollary}\label{Cor - Cohomology of currents for maximal compact subgroup}
	We have $\ul{H}^0 \big( {G_0} , \Cur(\bE,\BZ) \big) = \{0\}$ and an isomorphism of solid abelian groups
	\begin{align*}
		\ul{H}^1 \big( {G_0} , \Cur(\bE,\BZ) \big)	&\overset{\sim}{\lra} \BZ_p \oplus \BZ / (q+1)\,\BZ 
	\end{align*}
	whose underlying homomorphism of abelian groups maps $\delta(\1_{G_0 v_0})$ to $ \big(1, 1 \mod (q+1) \big)$.
	Together with the isomorphisms from \Cref{Prop - Cohomology of currents} and \ref{Prop - Cohomology of currents on finite subtree}, the isomorphism \eqref{Eq - Isomorphism with p-adic numbers} below and the canonical horizontal maps we obtain commutative diagrams
	\begin{equation}\label{Eq - Commutative diagram for isomorphisms for group cohomology of currents}
		\begin{tikzcd}[row sep= scriptsize, column sep= scriptsize]
			\ul{H}^1 \big( {G^0} , \Cur(\bE,\BZ) \big) \ar[r]\ar[d, "{\cong}"] & \ul{H}^1 \big( {G_0} , \Cur(\bE,\BZ) \big) \ar[r] \ar[d, "\cong"] & \ul{H}^1 \big( {G_0}, \Cur(\bE_{n+1}, \BZ) \big) \ar[d, "\cong"] \\
			\frac{1}{q-1} \BZ  \oplus \BZ/ (q+1) \,\BZ \ar[r, hook ]  & \BZ_p  \oplus \BZ/ (q+1) \,\BZ	 \ar[r, two heads] & \BZ/ q^n (q+1) \, \BZ .																				
		\end{tikzcd}
	\end{equation}
\end{corollary}
\begin{proof}
	We have seen in \Cref{Cor - Mittag-Leffler property for currents} that the inverse system $\big( \Cur(\bE_{n+1}, \BZ) \big)_{n\in \BN}$ satisfies the Mittag--Leffler condition.
	It is therefore acyclic, i.e.\ the homomorphism\footnote{The kernel of this homomorphism is $\varprojlim_{n\in \BN} \Cur(\bE_{n+1}, \BZ) \cong \Cur(\bE, \BZ)$.}
	\begin{equation*}
		\prod_{n\in \BN} \Cur(\bE_{n+1}, \BZ) \lra \prod_{n\in \BN} \Cur(\bE_{n+1}, \BZ) \,,\quad (\varphi_n)_{n\in \BN} \lto \big( \varphi_n - \varphi_{n+1}\res{\bE_{n+1}} \big)_{n\in \BN} ,
	\end{equation*}
	is surjective. 
	Because $\prod_{n\in \BN} \Cur(\bE_{n+1}, \BZ)$ is Polish, the inverse system $\big( \ul{ \Cur(\bE_{n+1}, \BZ) } \big)_{n\in \BN}$ is acyclic as well by \Cref{Lemma - Strictly exact sequence gives exact sequence of condensed modules}. 
	We may therefore apply \Cref{Lemma - Group cohomology and inverse system} and deduce from \Cref{Prop - Cohomology of currents on finite subtree} that
	\begin{align*}
		\ul{H}^0 \big( {G_0} , \Cur(\bE,\BZ) \big) &\cong \textstyle \varprojlim_{n\in \BN} \ul{H}^0 \big( {G_0} , \Cur(\bE_{n+1} ,\BZ) \big) = \{0\} ,\\
		\ul{H}^1 \big( {G_0} , \Cur(\bE,\BZ) \big) &\cong \textstyle \varprojlim_{n\in \BN}  \ul{H}^1 \big(  {G_0} , \Cur(\bE_{n+1} ,\BZ) \big) \overset{\sim}{\lra} \varprojlim_{n\in \BN}  \BZ/ q^n (q+1) \, \BZ .
	\end{align*}
	Here the latter isomorphism maps $\delta(\1_{G_0 v_0} )$ to $\big( 1 \mod q^n (q+1)  \big)_{n\in \BN}$.
	Furthermore, we have
	\begin{align}\label{Eq - Isomorphism with p-adic numbers}
		\textstyle \varprojlim_{n\in \BN} \BZ / q^n (q+1)\,\BZ &\overset{\sim}{\lra} \BZ_p \oplus \BZ/ (q+1)\,\BZ \,,\quad \\
		\big( a_n \mod q^n(q+1) \big)_{n\in \BN} &\lto  \Big( \big( a_n \mod q^n \big)_{n\in \BN} \,,\, a_0 \mod (q+1) \Big) , \nonumber
	\end{align}
	which also yields an isomorphism for the associated condensed abelian groups.
	This shows the claimed isomorphism and that the right hand side square of \eqref{Eq - Commutative diagram for isomorphisms for group cohomology of currents} is commutative.

	To verify the commutativity of the left hand side square it suffices to show that the outer square is commutative for all $n\in \BN$. 
	Through
	\begin{equation*}
		( G^0 \cdot v_0 ) \cap \bV_n = \bigcup_{i=0}^{ \lfloor \frac{n}{2} \rfloor } G_0 \cdot v_{2i} \qquad\text{and}\qquad ( G^0 \cdot v_1 ) \cap \bV_n = \bigcup_{i=0}^{ \lfloor \frac{n-1}{2} \rfloor } G_0 \cdot v_{2i+1}
	\end{equation*}
	together with the relations \eqref{Eq - Relations module the image of Sigma_n}, we see that 
	\begin{alignat*}{4}
		\ul{H}^1 \big( G^0 , \Cur(\bE,\BZ) \big) &\lra \ul{H}^1 \big( {G_0}, \Cur(\bE_{n+1}, \BZ) \big) &&\overset{\sim}{\lra}  \BZ/ q^n (q+1) \, \BZ && \\
		\delta( \1_{G^0 v_0} ) &\lto \sum_{i=0}^{ \lfloor \frac{n}{2} \rfloor } \delta_n( \1_{G_0 v_{2i}} )  &&\lto  \sum_{j=0}^{2 \lfloor \frac{n}{2} \rfloor } q^{j}  \mod q^n(q+1) &&\eqdef b_n , \\
		\delta( \1_{G^0 v_1} ) &\lto\sum_{i=0}^{ \lfloor \frac{n-1}{2} \rfloor } \delta_n( \1_{G_0 v_{2i+1}} ) &&\lto  - \sum_{j=0}^{2 \lfloor \frac{n-1}{2} \rfloor +1 } q^{j}  \mod q^n (q+1) &&\eqdef c_n.
	\end{alignat*}
	Using that $\sum_{j = 0}^{\infty} q^j = \frac{-1}{q-1}$ in $\BZ_p$ as well as $b_0 \equiv 1 $ and $c_0 \equiv 0 \mod (q+1)$, one shows that under the isomorphism \eqref{Eq - Isomorphism with p-adic numbers}
	\begin{equation*}
		\big( b_n \big)_{n\in \BN} \lto \big( \textstyle \frac{-1}{q-1} , 1 \mod (q+1) \big)  \quad\text{and}\quad
		\big( c_n \big)_{n\in \BN} \lto \big( \textstyle \frac{1}{q-1} , 0 \mod (q+1) \big) .
	\end{equation*}
	This proves the sought commutativity of \eqref{Eq - Commutative diagram for isomorphisms for group cohomology of currents}.
\end{proof}

\begin{proposition}\label{Prop - Cohomology of currents on Iwahori finite subtree}
	For $n\in \BN$, the solid abelian group $\ul{H}^0 \big( {I} , \Cur(\bE'_{n+1}, \BZ) \big)$ is discrete and its underlying abelian group contains an element $\psi_n$ (defined in \eqref{Eq - Definition of cochain in the kernel of Sigma_n} below) which induces an isomorphism
	\begin{equation*}
		\ul{H}^0 \big( {I} , \Cur(\bE'_{n+1}, \BZ) \big) \overset{\sim}{\lra} \BZ \,,\quad \psi_{n} \lto 1 .
	\end{equation*}
	Moreover, this element satisfies $\conjelt^\ast \psi_n = - \psi_n$, and the above isomorphisms fit into commutative diagrams
	\begin{equation*}
	\begin{tikzcd}[row sep= scriptsize]
		\ul{H}^0 \big( {I} , \Cur(\bE'_{n+2}, \BZ) \big) \ar[r , "\sim"] \ar[d] & \BZ \ar[d, "\cdot q"] \\
		\ul{H}^0 \big( {I} , \Cur(\bE'_{n+1}, \BZ) \big) \ar[r, "\sim"] &	\BZ 
	\end{tikzcd}
	\end{equation*}
	where the left map is the canonical one and the right map is multiplication with $q$.
\end{proposition}
\begin{proof}
	Again, the starting point is the short strictly exact sequence of $I$-modules form \Cref{Prop - Short exact sequence for currents} (ii).
	From it we obtain the exact sequence
	\begin{equation*}
		0 \lra \Cur(\bE'_{n+1}, \BZ)^\ul{I} \lra C(\bE'_{n+1}, \BZ)^\ul{I} \overset{\Sigma_n^{\prime \ul{I}}}{\lra} C( \bV'_n , \BZ)^\ul{I} \overset{\delta'_n}{\lra} \ul{H}^1 \big( {I} , \Cur(\bE'_{n+1}, \BZ) \big) \lra 0 
	\end{equation*}
	of solid abelian groups. 
	Furthermore, we deduce that
	\begin{equation*}
		C(\bE'_{n+1}, \BZ)^I \cong \bigoplus_{i=-(n+1)}^{n+1} \BZ\, \1_{Ie_i} \qquad\text{and}\qquad C(\bV'_{n}, \BZ)^I \cong \bigoplus_{i=-n}^{n+1} \BZ\, \1_{Iv_i}  ,
	\end{equation*}
	from \Cref{Lemma - Action on finite subtree} (iii).
	Then $\Sigma_n^{\prime I}$ is given by
	\begin{equation*}\label{Eq - Relations for the image of Iwahori Sigma_n}
		\Sigma_n^{\prime I} ( \1_{Ie_i} ) = \begin{cases}
			q \,\1_{Iv_{-n}} & \text{, for $i=-(n+1)$,} \\
			\1_{Iv_i} + q \, \1_{Iv_{i+1}} & \text{, for $i=-n,\ldots,-1$,} \\
			\1_{Iv_0} + \1_{Iv_1} & \text{, for $i=0$,} \\
			q\, \1_{Iv_i} + \1_{Iv_{i+1}} & \text{, for $i=1,\ldots,n$,} \\
			q\, \1_{Iv_{n+1}} & \text{, for $i=n+1$.}
		\end{cases}
	\end{equation*}
	We now define the element
	\begin{equation}\label{Eq - Definition of cochain in the kernel of Sigma_n}
		\psi_n \defeq q^{n+1} \, \1_{I e_0} + \sum_{i=1}^{n+1} (-1)^i q^{(n+1)-i}  ( \1_{I e_i} + \1_{I e_{-i}})
	\end{equation}
	of $C(\bE'_{n+1}, \BZ)^I$.
	One directly verifies that $\psi_n$ lies in $\Ker\big(\Sigma_n^{\prime I} \big)$.
	This kernel is a free $\BZ$-submodule of rank $1$.
	Because the coefficient of $\psi_n$ for the basis element $\1_{I e_{n+1}}$ of $C(\bE'_{n+1}, \BZ)^I$ is $(-1)^{n+1} \in \BZ^\times$, it follows that $\Ker\big(\Sigma_n^{\prime I} \big) = \BZ \,\psi_n$.

	Since $\conjelt^\ast \1_{Ie_i} = -\1_{I e_{-i}}$ under \eqref{Eq - Modified conjugation action}, we deduce that $\conjelt^\ast \psi_n = - \psi_n$.
	Finally, the map induced by restriction from $C(\bE'_{n+2}, \BZ)$ to $C(\bE'_{n+1}, \BZ)$ is given by
	\begin{equation*}
		C(\bE'_{n+2}, \BZ)^I \lra C(\bE'_{n+1}, \BZ)^I \,,\quad \1_{Ie_i} \lto \begin{cases}	\1_{Ie_i} & \text{, for $i=-(n+1),\ldots,0, \ldots, n+1$,} \\
			0			&\text{, for $i=-(n+2)$ or $i=n+2$.}
			
		\end{cases}
	\end{equation*}
	Hence, this map sends $\psi_{n+1}$ to $q \, \psi_{n}$ which shows the claimed commutative square.
\end{proof}

\begin{corollary}\label{Cor - Cohomology of currents for Iwahori}
	We have $\ul{H}^0 \big( {I}, \Cur(\bE, \BZ) \big) = \{0\}$.
\end{corollary}
\begin{proof}
	Similarly to the proof of \Cref{Cor - Cohomology of currents for maximal compact subgroup}, we may apply \Cref{Lemma - Group cohomology and inverse system} to the inverse system $\big( \ul{\Cur(\bE'_{n+1}, \BZ)} \big)_{n\in \BN}$ of $\ul{I}$-modules.
	This yields an isomorphism
	\begin{equation*}
		\ul{H}^0 \big( {I}, \Cur(\bE, \BZ) \big) \overset{\sim}{\lra} \varprojlim_{n\in \BN} \ul{H}^0 \big( {I}, \Cur(\bE'_{n+1}, \BZ) \big) .
	\end{equation*}
	By \Cref{Prop - Cohomology of currents on Iwahori finite subtree} this inverse limit is isomorphic to $\varprojlim_{n\in \BN} \BZ$ with transition maps given by multiplication with $q$.
	It therefore vanishes.
\end{proof}

\begin{remark}
	With the same methods as for $G_0$ one can also prove an isomorphism $\ul{H}^1 \big( I, \Cur(\bE'_{n+1}, \BZ) \big)  \cong \BZ/ q^{n+1} \, \BZ$, for $n\in \BN$. 
	Furthermore, there is a short exact sequence
	\begin{equation*}
		0 \lra \ul{\BZ_p} \big/ \ul{\BZ} \lra \ul{H}^1 \big( {I}, \Cur(\bE, \BZ) \big) \lra \ul{\BZ_p} \lra 0 
	\end{equation*}
	of solid abelian groups.
	It is $\conjelt$-equivariant when $\conjelt$ acts by inversion on the first term and trivially on the last.
\end{remark}

\section{The Van der Put Transform}

Before specialising to the Drinfeld upper half plane, we record some notions concerning (continuous) equivariant line bundles on rigid $K$-analytic spaces more generally.
We follow Ardakov and Wadsley \cite{ArdakovWadsley23EquivLineBun}.

\subsection{Equivariant Line Bundles}

In this section, let $X$ be a rigid $K$-analytic space and $G$ a topological group which acts on $X$ by $K$-linear automorphisms, i.e.\ via a group homomorphism $\rho \colon G \to  \Aut_K(X,\CO_X)$.
By abuse of notation we let $g\colon X \ra X$ also denote the automorphism $\rho(g)$ induced by $g\in G$.
We assume that this action is continuous in the following sense.

\begin{definition}[{\cite[Def.\ 3.1.8]{Ardakov21EquivDmod}}]
	The action $\rho$ of $G$ on $X$ is \emph{continuous} if, for every quasi-compact, quasi-separated admissible open subset $U$ of $X$
	\begin{altenumeratelevel2}
		\item 
		the stabiliser $G_U$ of $U$ in $G$ is open in $G$,
		
		\item 
		the induced group homomorphism $\rho_U \colon G_U \to \Aut_K(U, \CO_U)$ is continuous with respect to the subspace topology on $G_U$ and a certain topology on $\Aut_K(U, \CO_U)$, see \cite[Thm.\ 3.1.5]{Ardakov21EquivDmod}.
	\end{altenumeratelevel2}
\end{definition}

Condition (2) can be illustrated as follows:
Let $U$ be an affinoid subdomain of $X$.
For any coherent $\CO_X$-module $\CM$, the space of sections $\CM(U)$ carries a canonical $K$-Banach space topology. 
Moreover, the algebra $\CB(\CM(U))$ of continuous $K$-linear endomorphisms is a $K$-Banach algebra via the operator norm so that the group of automorphisms $\CB(\CM(U))^\times$ becomes a topological group.

When $\CM$ is the structure sheaf $\CO_X$, the group $\CB(\CO_X(U))^\times$ is canonically identified with $\Aut_K(U,\CO_U)$.
Then the topology on $\Aut_K(U, \CO_U)$ in (2) is finer than the topology on $\CB(\CO_X(U))^\times$, see the proof of \cite[Lemma 3.2.4]{ArdakovWadsley23EquivLineBun}.
In particular, for a continuous group action $\rho$ and affinoid $U$ the induced homomorphism $\rho_U \colon G_U \to \CB(\CO_X(U))^\times$ is continuous.

\begin{definition}[{cf.\ \cite[Def.\ 2.4.1]{ArdakovWadsley23EquivLineBun}}]
	A \emph{$G$-equivariant structure} on an $\CO_X$-module $\CM$ is a collection of $\CO_X$-module homomorphisms
	\begin{align*}
		g^\CM \colon \CM \lra g^\ast \CM \,,\quad\text{for all $g\in G$,}
	\end{align*}
	such that
	\begin{align*}
		(gh)^\CM = h^\ast (g^\CM) \circ h^\CM \,,\quad\text{for all $g,h \in G$,}\qquad \text{and} \qquad 1^\CM = \id_\CM .
	\end{align*}
\end{definition}

For a coherent $\CO_X$-module $\CM$, a $G$-equivariant structure $(g^\CM)_{g\in G}$ yields $K$-linear maps
\begin{equation*}
	g^\CM(U) \colon \CM(U) \lra \CM( g(U) )\,, \quad\text{for all $g\in G$ and affinoid subdomains $U \subset X$.}
\end{equation*}
With respect to the canonical $K$-Banach space topologies on domain and codomain, these maps are continuous and therefore induce group homomorphisms $G_U \ra \CB(\CM(U))^\times$ \cite[Lemma 3.2.1]{ArdakovWadsley23EquivLineBun}.

\begin{definition}\label{Def - Continuous equivariant module}
	A \emph{continuous $G$-equivariant coherent $\CO_X$-module} is a coherent $\CO_X$-module $\CM$ together with a $G$-equivariant structure $(g^\CM)_{g\in G}$ such that, for all affinoid subdomains $U$ of $X$, the induced homomorphism $G_U \ra \CB(\CM(U))^\times$ is continuous.
\end{definition}

If $\CL$ is a continuous $G$-equivariant coherent $\CO_X$-module such that the underlying $\CO_X$-module is invertible, $\CL$ is also called a \emph{$G$-equivariant line bundle on $X$} \cite[Def.\ 3.2.3]{ArdakovWadsley23EquivLineBun}.

\begin{example}[{\cite[Def.\ 3.2.9]{ArdakovWadsley23EquivLineBun}}]
	For any continuous character $\chi \colon G \ra K^\times$, we can equip the trivial line bundle $\CO$ with a new $G$-equivariant structure
	\begin{equation*}
		g^{\CO_\chi}(U) : \CO(U ) \lra \CO(g(U)) \,,\quad f \lto \chi(g) \, g^{\CO}(U)(f) \,, \qquad\text{for $g\in G$.}
	\end{equation*}
	This defines a $G$-equivariant line bundle on $X$ which we denote by $\CO_\chi$.
\end{example}

\begin{definition}
	The \emph{$G$-equivariant Picard group $\Pic^G(X)$ of $X$} is the abelian group of isomorphism classes of $G$-equivariant line bundles on $X$.
	Its group law is induced by the tensor product of $G$-equivariant line bundles, and its unit element is (the class of) the structure sheaf $\CO$ with the $G$-equivariant structure induced from the action of $G$ on $X$.
\end{definition}

Note that there is a natural group homomorphism $\Pic^G(X) \ra \Pic(X)$ where $\Pic(X)$ denotes the \emph{Picard group of $X$}, i.e.\ the group of isomorphism classes of invertible $\CO_X$-modules on $X$.
This homomorphism is induced by forgetting the $G$-equivariant structure of a $G$-equivariant line bundle.
\\

Ardakov and Wadsley describe the isomorphism classes of $G$-equivariant structures in terms of continuous cocycles of $G$ acting on $\CO^\times(X)$ when $X$ is connected affinoid \cite[Lemma 3.3.1]{ArdakovWadsley23EquivLineBun}.
We need a slight generalisation of their result.

\begin{setting}\label{Set - At most countable affinoid covering}
	We consider a connected, reduced and quasi-separated rigid $K$-analytic space $X$ on which a topological group acts continuously via $K$-linear automorphisms.
	Additionally, we assume that $X$ affords an at most countable, admissible affinoid covering $\CU= (U_i)_{i\in I}$ such that the intersection of the stabilisers $G_0 \defeq \bigcap_{i\in I} G_{U_i}$ is an open subgroup of $G$.

	We consider $\CO^\times(X)$ endowed with the subspace topology induced from the canonical inclusions $\CO^\times(X) \hookrightarrow \CO(X) \hookrightarrow \prod_{i \in I} \CO(U_i)$.
	As the following lemma shows, $\CO^\times(X)$ then becomes a topological $G$-module.
\end{setting}

\begin{lemma}\label{Lemma - Group of invertible functions is Polish}
	The action of $G$ on $\CO^\times (X)$ is continuous and $\CO^\times (X)$ is a Polish abelian group.
	Moreover, the topology of $\CO^\times (X)$ does not depend on the choice of the at most countable, admissible affinoid covering.
\end{lemma}
\begin{proof}
	For an affinoid subdomain $U$ of $X$, the supremum seminorm makes $\CO(U)$ into a $K$-Banach algebra since $U$ is reduced. 
	Therefore, the multiplication map of $\CO^\times(U)$ is continuous with respect to the subspace topology.
	Since on $\CO^{\times \times}(U)$ inversion is given by the geometric series, it follows that $\CO^\times(U)$ is a topological group.

	On the other hand, we may also consider $\CO^\times(U)$ with the topology $\tau$ derived from the inclusion $\Lambda \colon \CO^\times(U) \hookrightarrow \CO(U)^2$, $f \mto (f, f^{-1})$.
	Because $(\CO^\times(U), \tau)$ is the preimage of $\{1\}$ under the continuous map $\CO(U)^{2} \ra \CO(U)$, $(f,f') \mto ff'$, it is complete with respect to the induced metric.
	Moreover, $\CO(U)$ is the quotient of some Tate algebra which is homeomorphic to the space of null sequences with values in $K$ and thus separable.
	%$C$ contains the algebraic closure of $\BQ_p$ or $\BF_q((t))$ as a countable dense subset, hence also $K$ has a dense countable subset $A$. Then $c_0(\BN, K)$ contains the dense countable subset $c_{00}(\BN, A)$.
	Hence, $\CO(U)$ and in turn $( \CO^\times(U), \tau)$ are Polish spaces.

	Clearly, $\tau$ is finer than the subspace topology $\CO^\times(U) \subset \CO(U)$.
	But since the inversion map is continuous with respect to the subspace topology, it follows that $\tau$ agrees with the latter.

	Having seen that $\CO^\times(U)$ is a Polish group, for every affinoid subdomain $U$ of $X$, it follows that the countable product $\prod_{i \in I} \CO^\times(U_i)$ is Polish as well.
	When canonically identifying $\CO(X)$ with the kernel of of $\prod_{i \in I} \CO(U_i) \ra \prod_{i,j \in I} \CO(U_i\cap U_j)$, the subspace $\CO^\times(X)$ is identified with the kernel of $\prod_{i \in I} \CO^\times(U_i) \ra \prod_{i,j \in I} \CO^\times(U_i\cap U_j)$.
	Therefore, $\CO^\times(X)$ is a Polish group when endowed with this choice of topology.

	We want to compare this topology on $\CO^\times (X)$ to the one induced by another at most countable admissible covering $\CV = (V_j)_{j\in J}$ by affinoid subdomains.
	Because $X$ is quasi-separated, there exists an at most countable, admissible affinoid covering which refines the one given by $U_i \cap V_j$, for $(i,j)\in I\times J$.
	We may thus assume that $\CV$ is a refinement of $\CU$.
	It follows that the topology on $\CO^\times (X)$ induced by $\CU$ is finer than the one induced by $\CV$.
	The open mapping theorem then implies that both topologies agree.

	By \cite[Lemma 3.2.1]{ArdakovWadsley23EquivLineBun} every fixed $g \in G$ acts by continuous $K$-algebra homomorphisms $\CO(U) \ra \CO(g(U))$, for all affinoid subdomains $U$ of $X$, and hence by a continuous automorphism on $\CO^\times(X)$. 
	To show that the action of $G$ on $\CO^\times(X)$ is continuous, it therefore suffices to show that its restriction to $G_0$ is.
	It is a consequence of the proof of \cite[Lemma 3.2.4]{ArdakovWadsley23EquivLineBun} that the map $\rho_i \colon G_0 \ra \CB(\CO(U_i))$ is continuous, for all $i\in I$.
	Therefore the action
	\begin{align*}
		G_0 \times \CO(U_i) \xrightarrow{\rho_i \times \id} \CB(\CO(U_i)) \times \CO(U_i) &\lra \CO(U_i) , \\
		(\varphi, f) &\lto \varphi(f) ,
	\end{align*}
	is continuous.
	From this we deduce that indeed $G_0 \times \CO^\times(X) \ra \CO^\times(X)$ is continuous. 	
\end{proof}

\begin{proposition}\label{Prop - Equivariant structures and cocycles}
	\begin{altenumerate}
		\item 
		In the above setting, there is a natural bijection between the set of continuous $G$-equivariant structures on a given trivial line bundle and the set of continuous $1$-cocycles $Z^1 \big(G, \CO^\times(X) \big)$. 
		Concretely, a continuous $G$-equivariant structure $(g^\CL)_{g\in G}$ on a trivial line bundle $\CL = \CO \cdot v$ with generating global section $v$ is mapped to the function $\alpha \colon G \ra \CO^\times(X)$ determined by
		\begin{align*}
			g^\CL(v) = \alpha(g) \, v \,, \quad\text{for all $g\in G$.}
		\end{align*}
		
		\item 
		The bijection from (i) induces an isomorphism of abelian groups
		\begin{align*}
			\Ker\big( \Pic^G(X) \ra \Pic(X) \big)\overset{\sim}{\lra} H^1 \big(G, \CO^\times(X) \big) .
		\end{align*}
	\end{altenumerate}
\end{proposition}
\begin{proof}
	All but one arguments in the proof of \cite[Lemma 3.3.1]{ArdakovWadsley23EquivLineBun} hold true verbatim in our more general setting.
	The exception is the reasoning for the statement that a given $G$-equivariant structure $(g^\CL)_{g\in G}$ is continuous if and only if the associated $1$-cocycle $\alpha$ is. 
	We therefore adapt this argument.
	
	Let $U$ be an affinoid subdomain of $X$.
	The trivialisation $\CL \cong \CO \cdot v$ induces an isomorphism $\CB(\CO(U))^\times \cong \CB(\CL(U))^\times$ of topological groups.
	Since the $g^\CL$ are $\CO_X$-linear, we have
	\begin{equation*}
		g^\CL(f v\res{U}) = g^\CO(f) \, g^\CL(v\res{U}) = g^\CO(f) \, \alpha(g)\res{U} \, v\res{U} \,, \quad\text{for all $g\in G_U$, $f\in \CO(U)$.}
	\end{equation*}
	Hence, the homomorphism $G_U \ra \CB(\CL(U))^\times $, $g\mto g^\CL(U)$, is equal to the composition of $G_U \ra \CB(\CO(U))^\times$, $g \mto  g^\CO(U) \, \alpha(g)\res{U} $, and the above isomorphism. 
	We recall that $g \mto g^\CO(U)$ is continuous by the assumption that $G$ acts continuously on $X$. 
	This shows that $g \mto g^\CL(U)$ is continuous whenever $\alpha$ is.
	
	Conversely assume that $(g^\CL)_{g\in G}$ is continuous.
	From the above reasoning for $U_i$ it follows that the homomorphisms $G_{U_i} \ra \CO^\times(U_i)$, $g \mto \alpha(g)\res{U_i}$, are continuous, for all $i \in I$. 
	This implies that $G_0  \ra \CO^\times(X)$, $g \mto \alpha(g)$, is continuous with respect to the topology prescribed on $\CO^\times(X)$.
	Because $G_0$ was assumed to be open in $G$, the claim for $\alpha$ follows. 
\end{proof}

\subsection{The Drinfeld Upper Half Plane}

We now consider the projective line $\BP_F^1$ with the (left) group action\footnote{Even an action of the algebraic group $\GL_{2,F}$ on $\BP_F^1$.} of $G\defeq \GL_2(F)$ which is given on $C$-valued points by
\begin{equation}\label{Eq - Action on projective line}
	\left( \begin{matrix} a & b \\ c & d \end{matrix} \right) \cdot [ z_0 : z_1 ] = [a z_0 + b z_1 : c z_0 + dz_1 ] . %\quad\text{, for $\left( \begin{matrix} a & b \\ c & d \end{matrix} \right) \in \GL_2(\BC_p)$, $[z_0: z_1] \in \BP_F^1(\BC_p)$.}
\end{equation}
%\begin{equation*}
%	\left( \begin{matrix} a & b \\ c & d \end{matrix} \right) \cdot
%	\arraycolsep=2pt
%	\def\arraystretch{0.6}
%	\underbracket[0.7pt][1.5pt]{\overbracket[0.7pt][1.5pt]{\begin{array}{c} z_0 \\[0pt] \cdot \cdot \\[-2pt]z_1	\end{array} }} = 
%	\underbracket[0.7pt][1.5pt]{\overbracket[0.7pt][1.5pt]{\begin{array}{c} a z_0 + b z_1  \\[0pt] \cdot \cdot \\[-2pt] c z_0 + dz_1 	\end{array} }} 
%\end{equation*}
We identify $\BP_F^1$ with $\Proj \, \mathrm{Sym} \big( (F^2)^\ast \big)$ and fix on the dual space $(F^2)^\ast$ the $G$-action 
\begin{equation*}
	g \cdot \ell \defeq \ell( g^\ad \cdot \blank ) \,,\quad\text{for $g \in G$, $\ell \in (F^2)^\ast$,}
\end{equation*}
where $g^\ad = \det(g) \, g^{-1}$ denotes the adjunct of $g$.
Then the induced $G$-action on $\mathrm{Sym}\big( (F^2)^\ast \big)$ gives rise\footnote{We remark that twists of the above $G$-action on $(F^2)^\ast$ by integer powers of $\det$ induce \eqref{Eq - Action on projective line} as well, but lead to a different choice of $G$-equivariant structure on $\CO(1)$ in \Cref{Example - Cohomology class of twisting sheaf}.} to the action \eqref{Eq - Action on projective line} on $\BP_F^1$.
% in the sense that $g \big(\Fm_{[z_0:z_1]}\big)= \Fm_{g\cdot [z_0:z_1]}$ where $\Fm_{[z_0:z_1]}$ is the maximal homogeneous ideal corresponding to $[z_0:z_1]\in \BP_F^1(C)$

The \emph{Drinfeld upper half plane} over $F$ is defined as the (rigid $F$-analytic) projective line $\BP_F^1 \defeq \big(\BP^1_F \big)^\rig$ with all $F$-rational points removed
\begin{equation*}
	\Omega_F \defeq \BP^1_F \setminus \BP^1(F) .
\end{equation*}
We use the simplified notation $\Omega \defeq \Omega_F \times_F K \subset \BP_K^1$ to denote its base change to $K$.
Then $\Omega$ comes with an action of $G$ induced from \eqref{Eq - Action on projective line}.

In \Cref{Def - Admissible covering of DHP}, we will define an admissible affinoid covering $\Omega_0 \subset \Omega_1 \subset \ldots$ which exhibits $\Omega$ as a (separated) admissible open subset of $\BP_K^1$.
Moreover, $\Omega$ is connected \cite[Thm.\ 2.4]{Kohlhaase11LubinTateDrinfeldBun}, and the $\Omega_n$ are all stabilised by $G_0$.

\begin{lemma}
	The action of $G$ on $\Omega$ is continuous.
\end{lemma}
\begin{proof}
	This follows from Proposition 3.1.12 (b) and Lemma 3.1.9 of \cite{Ardakov21EquivDmod} by considering the natural action of $\GL_{2,\CO_K}$ on $\BP^1_{\CO_K}=\BP^1_K$, and the inclusion $\GL_2(\CO_F) \hookrightarrow \GL_2(\CO_K)$.
\end{proof}

It is a classical result\footnote{In fact, this has been generalised by Junger to Drinfeld symmetric spaces of any dimension \cite[Thm.\ A]{Junger23CohomAnaArrHyper}.} that $\Pic(\Omega) \cong H^1(\Omega,\CO^\times)$ is trivial \cite[Thm.\ 2.7.6]{FresnelvanderPut04RigidAnGeom}. 
We thus deduce from \Cref{Prop - Equivariant structures and cocycles}:

\begin{corollary}[{cf.\ \cite[Cor.\ 2.11]{Junger23CohModpFibresEquivDrinfeld}}]\label{Cor - Equivariant Picard group and group cohomology on DHP}
	Let $H$ be a closed subgroup of $G$. 
	There is a natural isomorphism of abelian groups
	\begin{align*}
		\Pic^H(\Omega) \overset{\sim}{\lra} H^1 \big( H, \CO^\times(\Omega) \big) .
	\end{align*}
\end{corollary}

\begin{example}\label{Example - Cohomology class of twisting sheaf}
	We say that a non-trivial $K$-linear form $\ell \colon K^2 \ra K$ is \emph{$F$-rational} if there is a $1$-dimensional subspace $U\subset F^2$ such that $\Ker(\ell )= U_K$.
	For two such forms $\ell$ and $ \ell'$ the quotient $\frac{\ell}{\ell'}$ yields a well-defined function in $\CO^\times(\Omega)$ via $[z_0 : z_1] \mto \frac{\ell(z_0, z_1)}{\ell'(z_0,z_1)} $.

	In particular, we obtain a map
	\begin{align*}
		j_\ell \colon G \lra \CO^\times(\Omega) \,,\quad g \lto \frac{g \cdot \ell}{\ell} ,
	\end{align*}
	which one verifies to be a $1$-cocycle.
	For a different choice $\ell'$ instead of $\ell$ the resulting $1$-cocycles only differ by the $1$-coboundary $g \mto  \big( g \cdot \frac{\ell}{\ell'} \big) \big(\frac{\ell'}{\ell} \big)$.

	We now consider the restriction $\CO(1)$ of the twisting sheaf on $\BP_K^1$ to $\Omega$.
	Then $\CO(1)$ acquires a $G$-equivariant structure via the above $G$-action on $\mathrm{Sym}( (F^2)^\ast )$. 
	For any choice of trivialisation $\CO(1) = \CO \cdot \ell$ with a non-trivial, $F$-rational linear form $\ell$, this structure is given by $g^{\CO(1)}(\ell) = g \cdot \ell$.
	In particular, the $1$-cocycle associated to this trivialisation in \Cref{Prop - Equivariant structures and cocycles} (i) is equal to $\frac{g^{\CO(1)}(\ell)}{\ell} = j_\ell$ which thus is continuous.
	We conclude that the isomorphism class $[\CO(1)]$ is mapped to $[j_\ell]$ under the isomorphism of \Cref{Cor - Equivariant Picard group and group cohomology on DHP}.
\end{example}

\begin{definition}
	For any closed subgroup $H \subset G$ and some non-trivial, $F$-rational linear form $\ell \colon K^2 \ra K$, we let $[j]$ denote the class of $j_\ell$ in $H^1 \big( H, \CO^\times (\Omega) \big)$ (or even its image in $H^1 \big( H, \CO^\times (U) \big)$, for affinoid subdomains $U\subset \Omega$ when this is valid.
	We remark that under the conjugation action $\conjelt^\ast [ j] = [j]$ whenever this makes sense.
\end{definition}

Using the Bruhat--Tits tree $\CT$, there is an explicit construction of a formal model of $\Omega_F$ (see also \cite[Sect.\ I.3]{BoutotCarayol91ThmCerednikDrinfeld}).
However, it is mostly the special fibre $\ov{\Omega}$ of this formal model that is relevant to us. 
This special fibre comes with a specialisation (or reduction) map
\begin{equation*}
	r \colon \Omega_F \lra \ov{\Omega} \,,
\end{equation*}
and is an $\BF_q$-scheme locally of finite type.
The irreducible components $\ov{\Omega}_v$ of $\ov{\Omega}$ are all copies of $\BP_{\BF_q}^1$ indexed by the vertices $v \in \bV$. 
Two components $\ov{\Omega}_v$ and $\ov{\Omega}_{v'}$ intersect if and only if $v$ and $v'$ are joined by an edge of $\CT$.
These intersections are ordinary double points and occur at the $\BF_q$-rational points of $\ov{\Omega}_v$.
In other words, the $\BF_q$-rational points of $\ov{\Omega}_v$ correspond to the edges originating from $v$.

The union $\ov{\Omega}_v(\BF_q)$ of the $\BF_q$-rational points is a closed subscheme of $\ov{\Omega}_v$, and the preimage under $r$ of its complement
\begin{equation*}
	\Omega_{F, v} \defeq r^{-1} \big( \ov{\Omega}_v \setminus \ov{\Omega}_v(\BF_q) \big)\,, \quad\text{for $v\in \bV$,}
\end{equation*}
is an affinoid subdomain of $\Omega_F$.
The subdomain $\Omega_{F,v}$ is in fact isomorphic to a closed unit disc centred in $0$ with the $q$ open unit discs centred in the $F$-rational points removed.
Moreover, for any edge $e=\{v,w\} \in \bE$ one defines an affinoid subdomain $\Omega_{F,e}$ of $\Omega_F$ which contains $\Omega_{F,v}$ and $\Omega_{F,w}$, and which specialises under $r$ to $\ov{\Omega}_v \cup \ov{\Omega}_w$ with all $\BF_q$-rational points removed except the one corresponding to $e$.
The $(\Omega_{F,e})_{e\in \bE}$ constitute an admissible affinoid covering of $\Omega_F$.

When passing to the base change $\Omega \defeq \Omega_K$, we obtain specialisation maps
\begin{equation*}
	\Omega_v \defeq ( \Omega_{F,v})_K \lra \ov{\Omega}_{v,\kappa}  \setminus \ov{\Omega}_v(\BF_q)_\kappa \,, \quad\text{for $v\in \bV$,}
\end{equation*}
where $\kappa$ is the residue field of $K$. 
Under the identification $\ov{\Omega}_{v,\kappa} = \BP_{\kappa}^1$, the closed subscheme $ \ov{\Omega}_v(\BF_q)_\kappa$ is the scheme theoretic union of the closed points of $\BP_{\kappa}^1$ corresponding to the $\BF_q$-rational lines in $\kappa^2 $.
In particular the points of $ \ov{\Omega}_v(\BF_q)_\kappa$ canonically correspond to the edges originating from $v$.

Associated with the exhaustions of $\CT$ by the finite subtrees $\CT_n$ and $\CT'_n$, we also obtain an admissible coverings of $\Omega$.

\begin{definition}\label{Def - Admissible covering of DHP}
	For $e\in \bE$, we write $\Omega_e \defeq (\Omega_{F,e})_K$. 
	We set
	\begin{equation}\label{Eq - Definition of affinoid}
		\Omega_{F,n} \defeq \bigcup_{e \in \bE_n} \Omega_{F,e}	\qquad\text{and}\qquad	\Omega_n \defeq (\Omega_{F,n})_K = \bigcup_{e \in \bE_n} \Omega_{e} \,,\qquad\text{for $n \in \BN$,}
	\end{equation}
	so that $\Omega_0 \subset \Omega_1 \subset \ldots$ (resp.\ $\Omega_{F,0} \subset \Omega_{F,1} \subset \ldots$) is a $G_0$-stable, admissible affinoid covering of $\Omega$ (resp.\ $\Omega_F$).
	Likewise, we have $I$-stable affinoid subdomains $\Omega'_n \defeq \bigcup_{e \in \bE'_n} \Omega_e$ of $\Omega$ which satisfy $\Omega'_n = \Omega_{n+1} \cap \conjelt \Omega_{n+1} =  \Omega_n \cup \conjelt \Omega_n $.	
\end{definition}

\begin{remark}\label{Rmk - Inverse limit description for invertible functions}
	With these admissible coverings of $\Omega$ we find ourselves in the \Cref{Set - At most countable affinoid covering}.
	It follows form \Cref{Lemma - Group of invertible functions is Polish} that $\CO^\times (\Omega)$ is a Polish $G$-module topologised via the isomorphism
	\begin{equation*}
		\CO^\times(\Omega) \overset{\sim}{\lra} \textstyle \varprojlim_{n\in \BN} \CO^\times(\Omega_n) \,,\quad f \lto ( f\res{\Omega_n} )_{n\in \BN} , 
	\end{equation*}
	of $G_0$-modules.
	Moreover, we also have an isomorphism $\CO^\times(\Omega) \cong \varprojlim_{n\in \BN} \CO^\times(\Omega'_n)$ of topological $I$-modules.
\end{remark}

\begin{lemma}\label{Lemma - Invariant function is constant}
	If a function $f$ in $\CO(\Omega)$, $\CO(\Omega_n)$ or $\CO(\Omega'_n)$ is invariant under a family of elements $\big\{ \big ( \begin{smallmatrix} 1 & x_i \\ 0 & 1 \end{smallmatrix} \big) \big\}_{i\in \BN} $ for $x_i \in \CO_F$ such that $x_i \ra 0$, then $f$ is constant.
	For example, this condition is satisfied if $f$ is invariant under $I$.
\end{lemma}
\begin{proof}
	Since $\CO(\Omega) = \varprojlim_{n\in \BN} \CO(\Omega_n) $, it suffices to consider $f \in \CO(\Omega_n)$ or $f\in \CO(\Omega'_n)$. 
	By \cite[Prop.\ 4.2.2]{ArdakovWadsley23EquivLineBun}, both $\Omega_n$ and $\Omega'_n$ are complements of a (non-empty) finite union of open discs in $\BP_K^1$. 
	Hence by \cite[Thm.\ 2.2.9 (1)]{FresnelvanderPut04RigidAnGeom}, $f$ is either the zero function or has only finitely many zeros on this complement .

	Let $[z:1] \in \Omega_n(C) \subset \Omega'_n(C)$ so that we may assume $f([z:1])= 0$. 
	Since $f$ is invariant under the $\big ( \begin{smallmatrix} 1 & x_i \\ 0 & 1 \end{smallmatrix} \big)$ we obtain infinitely many points $\big ( \begin{smallmatrix} 1 & x_i \\ 0 & 1 \end{smallmatrix} \big)^{-1} \cdot [z: 1] = [z -x_i : 1]$ where $f$ vanishes.
	Therefore, $f$ must be the zero function.
\end{proof}

\subsection{The Van der Put Transform}

We can now recapitulate a construction by Van der Put \cite[Sect.\ 2]{vanderPut92DiscGrpsMumfordCurv} which is crucial to our strategy.
To an invertible function $f \in \CO^\times (\Omega)$ he associates a cochain $P(f) \in C(\bA, \BZ)$ as follows:
Let $(v,w) \in \bA$ be a directed edge corresponding to the closed point $x$ of $ \ov{\Omega}_v(\BF_q)_\kappa \subset \ov{\Omega}_{v,\kappa}$. 
By scaling\footnote{The supremum norm on $\Omega_{v}$ in fact takes values in $\abs{K}$, see \cite[Prop.\ 2.4.8 (a)]{Lütkebohmert16RigidGeo}.} one may assume that the supremum of $f$ restricted to $\Omega_{v}$ is equal to $1$. 
Then the reduction $\ov{f}$ of $f$ is a regular function on $\ov{\Omega}_{v,\kappa} \setminus \ov{\Omega}_v(\BF_q)_\kappa$ and it uniquely extends to a rational function on $\ov{\Omega}_{v,\kappa}$.
One defines $P(f) ((v,w))$ as the order of the zero of $\ov{f}$ at $x$:
\begin{equation}\label{Eq - Definition of van der Put transform}
	P(f)( (v,w) ) \defeq \mathrm{ord}_{x} (\ov{f}) . 
\end{equation}

\begin{theorem}[{Van der Put, \cite[Thm.\ 2.1]{vanderPut92DiscGrpsMumfordCurv}}]\label{Thm - Van der Put sequence}
	The transform $P$ is a $G$-equivariant group homomorphism and fits into a short exact sequence of $G$-modules
	\begin{equation*}
		1 \lra K^\times \lra \CO^\times(\Omega) \overset{P}{\lra} \Cur(\bA,\BZ) \lra 0 .
	\end{equation*}
\end{theorem}

\begin{example}[{cf.\ \cite[Prop.\ 2.11]{Gekeler20InvFunctNASymmSpa}}]\label{Example - Van der Put transform of quotient of linear forms}
	Let $\ell$ and $\ell'$ be two non-trivial, $F$-rational linear forms $K^2 \ra K$ with $\Ker(\ell) = U_K$ and $\Ker(\ell') = U'_K$, for subspaces $U, U' \subset F^2$. 
	Then $\frac{\ell}{\ell'}$ defines a function in $\CO^\times(\Omega)$, and we want to compute its Van der Put transform.
	Since $\frac{\ell}{\ell'} \in K^\times$ if $U=U'$, we assume that $U \neq U'$.

	Let $v \in \bV$ correspond to a lattice $L$ of $F^2$.
	We may scale $\ell$ in such a way that the reduction $\ov{\ell} \colon L/\unif L \ra \BF_q$ is a non-zero linear form with $\Ker(\ov{\ell})= (L\cap U)/\unif L$.
	Let $x$ and $x'$ denote the points of $\BP^1(L/\unif L)_\kappa$ corresponding to $(L\cap U)/\unif L$ and $(L\cap U')/\unif L$ respectively.
	Then the reduction $\ov{ \big( \frac{\ell}{\ell'} \big) }$ on $\ov{\Omega}_{v,\kappa}$
	\begin{altenumeratelevel2}
		\item 
		either has a simple zero at $x$ and a simple pole at $x'$ if $x \neq x'$,
		
		\item 
		or is a regular function on $\BP^1(L/\unif L)_\kappa$ if $x= x'$.
	\end{altenumeratelevel2}

	Furthermore, $U$ and $U'$ can be regarded as ends of the tree $\CT$, see \cite[Sect.\ II.1.1]{Serre80Trees}.
	Hence there exist unique non-backtracking paths $(v, v_1, v_2,\ldots)$ and $(v, v'_1, v'_2, \ldots)$ of adjacent vertices towards $U$ and $U'$ respectively. 
	Concretely, $v_1$ is represented by the lattice $L_1 = \unif L + (L \cap U)$.
	Therefore, $x$ corresponds to the directed edge $(v,v_1)$, and similarly $x'$ corresponds to $(v,v'_1)$. 
	We conclude that, for $a \in \bA$,
	\begin{equation*}
		P \bigg(\frac{\ell}{\ell'} \bigg) (a) = \begin{cases} 				1& \text{, if $a$ lies on $\CS$ and points towards $U$,}	\\
			-1 & \text{, if $a$ lies on $\CS$ and points towards $U'$,} \\
			0 & \text{, if $a$ does not lie on $\CS$, } 
		\end{cases}
	\end{equation*}
	where $\CS$ is the unique straight path in $\CT$ between the ends $U$ and $U'$.
	\qed
\end{example}

To extract information about the continuous or condensed group cohomology of $\CO^\times(\Omega_K)$ out of the short exact sequence of \Cref{Thm - Van der Put sequence}, we have to regard it on the level of topological $G$-modules.

We note that in the definition of the Van der Put transform, $P(f)((v,w))$ only depends on the restriction of the invertible function $f$ to $\Omega_{v}$. 
Therefore, $P$ gives rise to a well-defined group homomorphism
\begin{equation*}
	P_n \colon \CO^\times (\Omega_n) \lra C(\bA_{\leq n}, \BZ) \,,\quad f \lto \big[ (v,w) \mto P(f)((v,w)) \big] ,
\end{equation*}
for all $n\in \BN$.
Analogously, we obtain a group homomorphism $P'_n \colon  \CO^\times (\Omega'_n) \ra C(\bA'_{\leq n}, \BZ)$.
Together with the canonical restriction maps, we obtain commutative diagrams
\begin{equation*}
	\begin{tikzcd}[row sep = small]
		\CO^\times(\Omega) \ar[r, "P"] \ar[d] & C(\bA, \BZ) \ar[d] \\
		\CO^\times(\Omega_n) \ar[r, "P_n"] & C(\bA_{\leq n}, \BZ) 
	\end{tikzcd}
\end{equation*}
so that $P = \varprojlim_{n\in \BN} P_n$.

\begin{proposition}\label{Prop - Van der Put sequence for finite subtree}
	For $n\in \BN$, there are short stricly exact sequences of topological $G_0$-modules (resp.\ $I$-modules)
	\begin{equation*}
		1 \lra K^\times \CO^{\times \times}(\Omega^{(\prime)}_n) \lra \CO^\times(\Omega^{(\prime)}_n) \overset{P^{(\prime)}_n}{\lra} \Cur \big(\bA^{(\prime)}_{\leq n}, \BZ \big) \lra 0 .
	\end{equation*}
\end{proposition}
\begin{proof}
	We first show that the kernel of $P_n$ equals $K^\times \CO^{\times \times} (\Omega_n)$.
	It is clear that $K^\times$ is contained in $\Ker(P_n)$.
	Moreover, any $f \in \CO^{\times \times} (\Omega_n)$ is of the form $f = (1+h)$, for some $h \in \CO(\Omega_n)$ with supremum norm $\lVert h \rVert_{\Omega_n} < 1$.
	Therefore, the reduction of $f$ restricted to $\Omega_{v}$ is the constant function $1$, for any $(v,w) \in \bA_{\leq n}$, and hence $P_n(f)((v,w)) = 0$.

	Conversely, let $f \in \Ker(P_n)$.
	It suffices to consider the covering \eqref{Eq - Definition of affinoid} and to show that the restriction of $f$ is contained in $K^\times \CO^{\times \times}(\Omega_{e})$, for all $e=\{v,w\} \in \bE_n$.
	There is a concrete description of $\Omega_{e}$ as follows, see \cite[Sect.\ I.2.3]{BoutotCarayol91ThmCerednikDrinfeld}:
	Let
	\begin{align*}
		D(a,r) \defeq \big\{ z \in C \,\big\vert\, \abs{z-a} \leq r \big\} \quad\text{and}\quad D^-(a,r) \defeq \big\{ z \in C \,\big\vert\, \abs{z-a} < r \big\}
	\end{align*}
	denote the closed, respectively open disc of radius $r\in \BR_{>0}$ centred in $a\in C$.
	Then, there is an identification
	\begin{equation}\label{Eq - Identification for affinoid over edge}
		\Omega_{e} \cong D(0,1) - \bigg( \displaystyle\bigcup_{i=1}^{q-1} D^-(a_i,1) \cup D^- \big(0, \textstyle\frac{1}{q} \big) \cup \displaystyle \bigcup_{i=1}^{q-1} D^- \big(b_i,\textstyle\frac{1}{q} \big) \bigg)
	\end{equation}
	where $a_1,\ldots,a_{q-1}$ is a full system of representatives of $\CO_F \!\setminus \! (\unif)$ modulo $(\unif)$, and $b_0 \defeq 0$, $b_1,\ldots,b_{q-1}$ is one of $(\unif)$ modulo $(\unif^2)$.
	By \cite[Prop.\ 2.4.8 (b)]{Lütkebohmert16RigidGeo}, on $\Omega_{e}$ the function $f$ can be uniquely expressed as
	\begin{equation*}
		f= c\, (1+h) \,  (z-a_1)^{m_1} \cdots (z-a_{q-1})^{m_{q-1}} \, z^{n_0} \, (z-b_1)^{n_1} \cdots (z-b_{q-1})^{n_{q-1}} \,,
	\end{equation*}
	for $c\in K^\times$, $h \in \CO(\Omega_{e})$ with $\lVert h \rVert_{\Omega_{e}} <1$, and $m_1,\ldots,m_{q-1}, n_0,\ldots,n_{q-1} \in \BZ$.
	Since under the identification \eqref{Eq - Identification for affinoid over edge}
	\begin{equation*}
		\Omega_{v} \cong D(0,1) - \bigg( \displaystyle\bigcup_{i=1}^{q-1} D^-(a_i,1) \cup D^- (0, 1) \bigg) ,
	\end{equation*}
	we obtain that the reduction of the scaled restriction of $f$ to $\Omega_{v}$ is equal to 
	\begin{equation*}
		(z- \ov{a_1})^{m_1} \cdots (z- \ov{a_{q-1}} )^{m_{q-1}} \, z^{\sum_{i=0}^{q-1} n_i } .
	\end{equation*}
	However, by the assumption that $f\in \Ker(P_n)$, this reduction has neither zeroes nor poles on $\ov{\Omega}_{v,\kappa}$.
	We thus conclude that in particular $m_1 = \ldots = m_{q-1} = 0$.
	Using that under \eqref{Eq - Identification for affinoid over edge} 
	\begin{equation*}
		\Omega_{w} \cong D\big( 0, \textstyle\frac{1}{q} \big)  - \bigg( D^- \big(0, \textstyle\frac{1}{q} \big)  \cup \displaystyle\bigcup_{i=1}^{q-1} D^- \big(b_i, \textstyle\frac{1}{q} \big)  \bigg) 
	\end{equation*}
	a similar argument shows that $n_0 = \ldots = n_{q-1} = 0$.
	This proves that the restriction of $f$ is contained in $K^\times \CO^{\times \times}(\Omega_e)$.

	Reasoning analogously as for the transform $P$, one finds that the image of $P_n$ is contained in $\Cur(\bA_{\leq n}, \BZ)$. 
	That they are equal follows from $\Im(P) = \Cur(\bA,\BZ)$ and the compatibility between $P$ and $P_n$.

	We remark that $\CO^{\times \times} (\Omega_n)$ and hence $\Ker(P_n)$ is an open subgroup of $\CO^\times (\Omega_n)$. 
	It follows that $P_n$ is continuous when $\Cur(\bA_{\leq n}, \BZ)$ carries the discrete topology. 
	Moreover, $P_n$ certainly is an open map then, and therefore strict.

	The proof for the short strictly exact sequence associated with $P'_n$ is analogous.	
\end{proof}

\begin{corollary}\label{Cor - Solid Van der Put sequence}
	The Van der Put sequence from \Cref{Thm - Van der Put sequence} gives rise to a short strictly exact sequence of topological $G$-modules
	\begin{equation*}
		1 \lra K^\times \lra \CO^\times(\Omega) \overset{P}{\lra} \Cur(\bA,\BZ) \lra 0 .
	\end{equation*}
\end{corollary}
\begin{proof}
	Since the $P_n$ are continuous homomorphisms by the preceding \Cref{Prop - Van der Put sequence for finite subtree}, the transform $P = \varprojlim_{n\in \BN} P_n$ is continuous.
	Moreover, it follows from the open mapping theorem for Polish groups that $P \colon \CO^\times(\Omega) \ra \Cur(\bA,\BZ)$ is open (cf.\ Remarks \ref{Rmk - Inverse limit description for currents} and \ref{Rmk - Inverse limit description for invertible functions}).
\end{proof}

\begin{remark}
	When considering this sequence associated with the Van der Put transform just on the level of topological $G^0$-modules (or for closed subgroups of $G^0$) we will tacitly make the ($G^0$-equivariant) identification $\Cur(\bA, \BZ) \cong \Cur(\bE, \BZ)$ from \Cref{Lemma - Isomorphism between currents on directed and undirected edges}.
\end{remark}

\subsection{First Consequences for Group Cohomology}

\begin{notation}
	We let $P_\ast$ (resp.\ $P_{n,\ast}$, $P'_{n,\ast}$) denote the homomorphism induced by $P$ (resp.\ $P_n$, $P'_n$) between the first continuous or condensed group cohomology groups; for example
	\begin{equation*}
		P_\ast \colon \ul{H}^1 \big( {G^0} , \CO^\times(\Omega) \big) \lra \ul{H}^1 \big(  {G^0} , \Cur(\bE, \BZ) \big) .
	\end{equation*}
	Moreover, we can compose $P_\ast$ with the isomorphisms in \Cref{Cor - Cohomology of currents for maximal compact subgroup} which describe the first group cohomology groups of $\Cur(\bE, \BZ)$ (resp.\ $\Cur(\bE_{n+1}, \BZ)$).
	We abbreviate this composition by $\widetilde{P_\ast}$ (resp.\ $\widetilde{P_{\ast,n}}$).
\end{notation}

\begin{theorem}\label{Thm - Left exact sequences for group cohomology}
	There are compatible exact sequences of solid abelian groups
	\begin{equation}\label{Eq - Short exact sequences for group cohomology of G^0 and G_0}
	\begin{tikzcd}[ row sep=scriptsize , column sep=scriptsize ,
	%	, /tikz/column 3/.append style={anchor=base east} 
	%	, /tikz/column 4/.append style={anchor=base west}
		]
		1 \ar[r] & \ul{\Hom} (G^0 , K^\times ) \ar[r]\ar[d] & \ul{H}^1 \big({G^0} , \CO^\times (\Omega) \big) \ar[r, "\widetilde{P_\ast}"] \ar[d] &\frac{1}{q-1} \BZ \oplus \BZ / (q+1)\,\BZ \ar[d, hook] & \\
		1 \ar[r] & \ul{\Hom} (G_0 , K^\times ) \ar[r]		 & \ul{H}^1 \big({G_0} , \CO^\times (\Omega) \big) \ar[r, "\widetilde{P_\ast}"] &  \BZ_p \oplus \BZ / (q+1)\,\BZ \ar[r] & 0
	\end{tikzcd}
	\end{equation}
	where on the underlying abelian groups $\widetilde{P_\ast}[j] = \big( 1, 1 \mod (q+1) \big)$.
	The lower sequence is the inverse limit of the short exact sequences, for $n\in \BN$,
	\begin{equation}\label{Eq - Short exact sequence for group cohomology for affinoids}
		1 \lra \ul{H}^1 \big( {G_0} , K^\times \CO^{\times \times}(\Omega_n) \big) \lra \ul{H}^1 \big( {G_0} , \CO^\times (\Omega_n) \big) \xrightarrow{\widetilde{P_{n,\ast}}} \BZ / q^n(q+1)\,\BZ \lra 0,
	\end{equation}
	with $\widetilde{P_{n,\ast}}[j] = 1 \mod (q^n(q+1)  )$ on the underlying abelian groups.
\end{theorem}
\begin{proof}
	We consider the condensation of the short strictly exact sequence associated with the Van der Put transform in \Cref{Cor - Solid Van der Put sequence}.
	By \Cref{Lemma - Strictly exact sequence gives exact sequence of condensed modules} this is a short exact sequence of solid $\ul{G^0}$-modules, and thus gives rise to a long exact sequence of condensed group cohomology
	\begin{align*}
		\ldots \ra \ul{H}^0 \big( {G^0} , \Cur(\bE,\BZ) \big) \ra \ul{H}^1 ( G^0 , K^\times ) \ra \ul{H}^1 \big( {G^0} , \CO^\times (\Omega) \big) \ra \ul{H}^1 \big( {G^0} , \Cur(\bE, \BZ) \big) \ra \ldots .
	\end{align*}
	Since the $G^0$-action on $K^\times$ is trivial, we have $\ul{H}^1 ( {G^0} , K^\times ) = \ul{\Hom}(G^0 , K^\times )$.
	Together with \Cref{Prop - Cohomology of currents} the first claimed exact sequence follows.

	For the $G_0$-group cohomology we resort to \Cref{Prop - Cohomology of currents on finite subtree} and \Cref{Cor - Cohomology of currents for maximal compact subgroup}.
	The same reasoning as above then implies the remaining two exact sequences, except for exactness at the respectively last (non-trivial) terms.

	Assuming the statement about $\widetilde{P_{n,\ast}}[j]$ for the moment, we see that $[j]$ is mapped to a generator of $\BZ / q^n(q+1)\,\BZ$ so that $\widetilde{P_{n,\ast}}$ is surjective on the underlying abelian groups.
	Since $\BZ / q^n(q+1)\,\BZ$ is discrete, we conclude that $\widetilde{P_{n,\ast}}$ is an epimorphism also on the level of condensed abelian groups.
	%Let $X \ra \ul{Y}$ be a map of condensed sets with $Y$ discrete and such that $X(\ast)\ra \ul{Y}(\ast)=Y$ is surjective. Let $S\in \mathrm{ED}$ and consider $f \in \ul{Y}(S)$, i.e.\ a locally constant function $f\colon S \ra Y$. Because the image of $f$ is finite, we may assume that $f$ is constant after passing to a covering of $S$. But then $f$ is in the image of $\ul{Y}(\ast) \ra \ul{Y}(S)$ and $X(\ast) \ra \ul{Y}(\ast)$ is surjective by assumption. It follows that we find an element of $X(S)$ which is mapped to $f$. This shows that $X(S)\ra \ul{Y}(S)$ is surjective.
	This shows that \eqref{Eq - Short exact sequence for group cohomology for affinoids} is exact.

	Concerning the second sequence, the inverse systems\footnote{We remark that $\varprojlim_{n\in \BN} K^\times \CO^{\times \times}(\Omega_n) = K^\times$.} of $G_0$-modules $\big( K^\times \CO^{\times \times}(\Omega_n) \big)_{n\in \BN}$ and $\big(  \CO^{\times }(\Omega_n) \big)_{n\in \BN}$ are acyclic by Prop.\ 4.5 and Thm.\ 7.1 of \cite{Junger23CohomAnaArrHyper} respectively.
	Like in the proof of \Cref{Cor - Cohomology of currents for maximal compact subgroup} we may apply \Cref{Lemma - Group cohomology and inverse system} to these systems. 
	For both, the inverse systems of their zeroth $\ul{G_0}$-cohomology groups are in fact constant by \Cref{Lemma - Invariant function is constant}.
	In particular, their $R^1 \!\varprojlim$'s vanish.
	Hence we deduce that the canonical maps
	\begin{align*}
		\ul{H}^1 ( {G_0} , K^\times ) &\overset{\sim}{\lra}  \textstyle \varprojlim_{n\in \BN} \ul{H}^1 \big( {G_0} , K^\times \CO^{\times \times}(\Omega_n) \big) \,, \\
		\ul{H}^1 \big( {G_0} , \CO^\times(\Omega) \big) &\overset{\sim}{\lra}  \textstyle \varprojlim_{n\in \BN} \ul{H}^1 \big( {G_0} , \CO^{ \times}(\Omega_n) \big)
	\end{align*}
	are isomorphisms.
	Similarly, it follows that the inverse system $\big( \BZ / q^n(q+1)\,\BZ \big)_{n\in \BN}$ of discrete condensed abelian groups is acyclic. 
	Therefore, the inverse limit of the last short exact sequences is a short exact sequence again.
	As seen above, it is isomorphic to the second sequence of the theorem.
	\\

	It remains to compute the image of $[j]$ under $\widetilde{P_{\ast}}$ and $\widetilde{P_{n, \ast}}$.
	To this end, let $\ell\colon K^2 \ra K$ be a non-trivial, $F$-rational linear form with $U \subset F^2$ such that $\Ker(\ell) = U_K$.
	To describe the $1$-cocycle 
	\begin{equation*}
		P_\ast ( j_\ell ) \colon G^0 \lra \Cur(\bE, \BZ) \,,\quad g \lto (P \circ j_\ell) (g) \eqdef \varphi_g ,
	\end{equation*}
	we let $\CS_g$ denote the straight path in $\CT$ between the ends $U$ and $g(U)$, for $g\in G^0$.
	Since $(P \circ j_\ell)(g) = P \big( \frac{g \cdot \ell}{\ell} \big)$ and $\Ker(g \cdot \ell) =  g(U)_K$, \Cref{Example - Van der Put transform of quotient of linear forms} implies	that, for $\{v_+, v_-\} \in \bE$,
	\begin{align*}
		\varphi_g ( \{ v_+, v_- \} ) =  \begin{cases}
			1 & \text{, if $(v_+,v_-)$ lies on $\CS_g$ and points towards $g(U)$,} \\
			-1&\text{, if $(v_+,v_-)$ lies on $\CS_g$ and points towards $U$,} \\
			0& \text{, else.}
		\end{cases}
	\end{align*}

	On the other hand, we consider $\varphi \in C(\bE, \BZ)$ defined by
	\begin{equation*}
		\varphi(\{v_+,v_-\} ) \defeq
		\begin{cases}
			1 & \text{, if $(v_+,v_-)$ points towards $U$,} \\
			0 & \text{, else.}
		\end{cases}
	\end{equation*}
	A direct computation shows that $\Sigma (\varphi) = \1_{G^0 v_0} + q\, \1_{G^0 v_1}$.
	By the definition of the boundary map $\delta \colon C( \bV, \BZ)^{G^0} \ra H^1 \big( G^0, \Cur(\bE, \BZ) \big) $, the cohomology class $\delta(  \1_{G^0 v_0} + q \, \1_{G^0 v_1} )$ is represented by the $1$-cocycle $g \mto g\cdot \varphi - \varphi$.
	But we find that $g\cdot \varphi - \varphi$ is equal to $\varphi_g$ which proves that $P_\ast[j_\ell] = \delta ( \1_{G^0 v_0} + q\,\1_{G^0 v_1} )$.

	To determine $P_{\ast,n} [j_\ell] \in H^1 \big( G_0 , \Cur(\bE_{n+1}, \BZ) \big)$, we let $\CP$ denote the path from $v_0$ to the end in $\CT$ corresponding to $U$.
	We define $\psi \in C(\bE_{n+1}, \BZ)$ by setting
	\begin{equation*}
		\psi( \{v_+, v_-\}) \defeq 
		\begin{cases}
			1 & \text{, if $(v_+,v_-)$ lies on $\CP$ and points towards $U$,} \\
			-1&\text{, if $(v_+,v_-)$ lies on $\CP$ and points away from $U$,} \\
			0 & \text{, else.}
		\end{cases}
	\end{equation*}
	Then $\Sigma_n(\psi) = \1_{G_0 v_0}$, and hence $\delta_n(\1_{G_0 v_0}) $ is equal to the class of the $1$-cocycle $g \mto g\cdot \psi - \psi$.
	Again, we have $g \cdot \psi - \psi = \varphi_g$, for $g\in G_0$, so that $P_{n,\ast}[j_\ell] = \delta_n ( \1_{G_0 v_0}) $.

	Finding the images of $[j]$ under $\widetilde{P_\ast}$ and $\widetilde{P_{n, \ast}}$ then is a straightforward computation involving the isomorphisms of \Cref{Prop - Cohomology of currents}, \ref{Prop - Cohomology of currents on finite subtree} and \Cref{Cor - Cohomology of currents for maximal compact subgroup}.
\end{proof}

\section{Group Cohomology for $\GL_2(\CO_F)$}

As part of \Cref{Thm - Left exact sequences for group cohomology} we could already describe $\ul{H}^1 \big( {G_0}, \CO^\times(\Omega) \big)$ as an extension of solid abelian groups. 
In this section, we want to determine the precise structure first of $H^1 \big( G_0, \CO^\times(\Omega_n) \big)[p']$ and then of $\ul{H}^1 \big( G_0, \CO^\times(\Omega) \big)$.

\subsection{Principal Units and Characters}

We begin by collecting some generalities about the sections of the sheaf $\CO^{\times \times}$ of principal units.
Let $U$ be a reduced affinoid $K$-space so that $\CO(U)$ is a $K$-Banach algebra with respect to the supremum norm $\lVert \blank \rVert_U$ and $\CO^\times(U) \subset \CO(U)$ is a Polish group, see \Cref{Lemma - Group of invertible functions is Polish}.
We define
\begin{equation}\label{Eq - Zp-module structure}
	(1+f)^\lambda \defeq \sum_{k \geq 0 } \binom{\lambda}{k} f^k \,, \quad\text{for $1+ f \in \CO^{\times \times}(U)$ and $\lambda \in \BZ_p$.}
\end{equation}

\begin{lemma}\label{Lemma - Zp-module structure on principal units}
	Let $U$ be a reduced affinoid $K$-space.
	Then \eqref{Eq - Zp-module structure} extends the topological abelian group structure on $\CO^{\times \times}(U)$ to a topological $\BZ_p$-module structure.
\end{lemma}
\begin{proof}
	Because of $\big\lVert \binom{\lambda}{k} f^k \big\rVert_U \leq \lVert f \rVert_U^k \ra 0$, the definition \eqref{Eq - Zp-module structure} yields a well defined element $(1+f)^\lambda \in \CO^{\times \times} (U)$.
	We consider the map
	\begin{equation}\label{Eq - Zp-module multiplication map}
		\BZ_p \times \CO^{\times \times}(U) \lra \CO^{\times \times}(U) \,,\quad \big(\lambda, 1+f \big) \lto (1+f)^\lambda .
	\end{equation}
	First, for fixed $1+f \in \CO^{\times \times}(U)$, the sequence of partial sums $\big( \sum_{k = 0 }^n \binom{\lambda}{k} f^k \big)_{n\in \BN}$ is uniformly Cauchy. 
	Since the functions $\lambda \mto \binom{\lambda}{k} f^k$, are continuous, these partial sums converge to the continuous function $\BZ_p \ra \CO^{\times \times}(U)$, $\lambda \mto (1+f)^\lambda$.

	Next, we show that \eqref{Eq - Zp-module multiplication map} is $\BZ$-bilinear.
	Indeed, for fixed $1+f \in \CO^{\times \times}$, we consider the two maps $(\lambda, \mu) \mto (1+f)^\lambda (1+f)^\mu$ and $(\lambda, \mu) \mto (1+f)^{\lambda + \mu}$ on $\BZ_p^2$.
	Since they are continuous and agree on the dense subset $\BN^2 \subset \BZ_p^2$, they agree for all $(\lambda, \mu)\in \BZ_p^2$.
	One verifies the other properties analogously.

	To prove continuity of \eqref{Eq - Zp-module multiplication map} for fixed $\lambda \in \BZ_p$, it thus suffices to show that the homomorphism $\CO^{\times \times}(U) \ra \CO^{\times \times}(U)$, $1+f \mto (1+f)^\lambda$, is continuous at $1$.
	But there we have
	\begin{equation*}
		\big\lVert (1+f)^\lambda -1 \big\rVert_U \leq \max_{k \geq 1} \bigg\lvert \binom{\lambda}{k} \bigg\rvert \lVert f \rVert^k_U \xrightarrow{ \lVert f \rVert_U \ra 0} 0 .
	\end{equation*}

	In total, we have seen that \eqref{Eq - Zp-module multiplication map} is a separately continuous $\BZ$-bilinear map between Polish abelian groups.
	Therefore, \cite[Cor.\ 3]{Pombo00SepContMap} implies that this map is (jointly) continuous.		
\end{proof}

\begin{lemma}\label{Lemma - Zp-module structure on group cohomology of principal units}
	For a closed subgroup $H \subset G_0$ and an $H$-stable affinoid subdomain $U \subset \Omega$, the solid abelian groups $\ul{H}^n \big( H , \CO^{\times \times}(U) \big)$ canonically are $\ul{\BZ_p}$-modules.
	Moreover, the maps $\conjelt^\ast \colon  \ul{H}^n \big( H , \CO^{\times \times}(U) \big) \ra \ul{H}^n \big( {}^\conjelt H , \CO^{\times \times}(\conjelt U) \big)$ are isomorphisms of condensed $\ul{\BZ_p}$-modules.
\end{lemma}
\begin{proof}
	By \Cref{Lemma - Zp-module structure on principal units} the multiplication map $\BZ_p \times \CO^{\times \times}(U) \ra \CO^{\times \times}(U) $ is continuous.
	As $C(\blank, \blank)$ is an internal Hom-functor in the category of compactly generated topological spaces, the induced map
	\begin{equation*}
		C( H^n , \BZ_p) \times C \big( H^n , \CO^{\times \times}(U) \big) \cong C\big( H^n , \BZ_p \times \CO^{\times \times}(U) \big) \lra C \big(H^n , \CO^{\times \times}(U) \big)
	\end{equation*}
	is continuous as well. 
	Precomposing with the embedding $\BZ_p \hookrightarrow C( H^n , \BZ_p)$ of constant functions then yields the multiplication for the $\BZ_p$-module structure on $C^n \big( H, \CO^{\times \times}(U) \big)$.
	Thus, the latter is a topological $\BZ_p$-module.

	Because the differentials $d^n \colon C^n \big( H, \CO^{\times \times}(U) \big) \ra C^{n+1} \big( H, \CO^{\times \times}(U) \big)$ are continuous homomorphisms between Hausdorff topological groups and $\BZ \subset \BZ_p$ is dense, we deduce that they are $\BZ_p$-linear.
	Their condensations $\ul{d^n}$ consequently are homomorphisms of condensed $\ul{\BZ_p}$-modules.
	Therefore, the subquotients $\ul{H}^n \big( H , \CO^{\times \times}(U) \big)$ inherit a condensed $\ul{\BZ_p}$-module structure via \Cref{Prop - Comparison condensed and continuous group cohomology} (ii).

	Similarly, we deduce that the isomorphisms $s^\ast\colon C^n \big( H , \CO^{\times \times}(U) \big) \ra C^n \big( {}^\conjelt H , \CO^{\times \times}(\conjelt U) \big)$ are $\BZ_p$-linear.
	Again, it follows that the induced isomorphisms $\conjelt^\ast$ of condensed group cohomology are isomorphisms of condensed $\ul{\BZ_p}$-modules.
\end{proof}

We now turn to $K^\times$-valued continuous characters of $G^0$, $G_0$ and $I$. 
Recall that $x\mto \widehat{x}$ denotes the projection $F^\times \twoheadrightarrow \mu_{q-1}(F)$ to the $(q-1)$-st roots of unity and $x\mto \langle x \rangle$ the projection $F^\times \twoheadrightarrow \CO^{\times \times}_F$ to the principal units of $F$.

\begin{lemma}\label{Lemma - Prime to p torsion characters}
	\begin{altenumerate}
		\item 	
			Every element of $\Hom( G^0 , K^\times) $ factors over $\det \colon G^0 \ra \CO_F$.
		
		\item 
			The elements of $\Hom(G_0, K^\times )[p']$ precisely are the characters
			\begin{equation*}
				\widehat{\det} {}^k \colon G_0 \lra K^\times \,,\quad g \lto \widehat{\det(g)}^k \,,\quad\text{for $k \in \BZ/ (q-1)\,\BZ$.}
			\end{equation*}
			
		\item 
			The elements of $\Hom(I, K^\times)[p'] $ precisely are the characters
			\begin{equation*}
				\chi_{k,l} \colon I \lra K^\times \,,\quad \left( \begin{matrix} a & b \\ \unif c & d \end{matrix} \right) \lto \widehat{a}^k \, \widehat{d}^l \,,\quad\text{for $k, l \in \BZ/ (q-1) \,\BZ$. }
			\end{equation*}
			These characters satisfy $\conjelt^\ast \chi_{k,l} = \chi_{l,k}$ and $\chi_{k,k} = \widehat{\det}{}^k$.
			
	\end{altenumerate}
\end{lemma}
\begin{proof}
	For (i) one uses that $G^0$ contains $SL_2(F)$ which is its own derived subgroup.
	The assertions (ii) and (iii) are \cite[Lemma 2.2.3]{ArdakovWadsley23EquivLineBun}.
\end{proof}

\begin{lemma}\label{Lemma - Embedding of characters}
	For $n\in \BN$, the canonical homomorphisms
	\begin{align*}
		\ul{\Hom} ( G_0 , K^\times ) [p'] &\lra \ul{H}^1 \big({G_0}, \CO^\times(\Omega_n ) \big)\,, \\
		\ul{\Hom} ( I , K^\times ) [p'] 		&\lra \ul{H}^1 \big( {I}, \CO^\times(\Omega'_n ) \big) 
	\end{align*}
	induced by $K^\times \hookrightarrow \CO^\times(\Omega_n ) $ (resp.\ by $K^\times \hookrightarrow \CO^\times(\Omega'_n ) $) are injective.
\end{lemma}
\begin{proof}
	Let $\iota$ denote the first homomorphism and $\iota(\ast)$ the induced homomorphism between the underlying abelian groups.
	Since $\ul{\Hom} ( G_0 , K^\times ) [p']$ is discrete, \cite[Lemma 4.3]{Tang24OpenCondSubgrMackeysFormula} implies that $\Ker(\iota)$ is discrete as well, i.e.\ equal to $\ul{\Ker(\iota)(\ast)}$.
	Because the global sections functor is (left) exact, the latter then is equal to $\ul{\Ker(\iota(\ast))}$.

	Therefore, it suffices to show that $\iota(\ast)$ is injective.
	We assume that $\widehat{\det}{}^k \in  \Hom ( G_0 , K^\times ) [p'] $ lies in $\Ker(\iota(\ast))$ i.e.\ that there exists $f\in \CO^\times(\Omega_n) $ such that $\widehat{\det}{}^k (g) = \frac{g \cdot f}{f}$, for all $g\in G_0$. 
	In particular, for $g = \left( \begin{smallmatrix} 1 & x \\ 0 & 1 \end{smallmatrix} \right)$ with $ x\in \CO_F$, we have $g \cdot f = \widehat{\det}{}^k (g) f = f$.
	It follows from \Cref{Lemma - Invariant function is constant} that $f$ is a constant function, and hence $\widehat{\det}{}^k = 1$.

	The argument for characters of $I$ is analogous.
\end{proof}

\begin{lemma}\label{Lemma - Characters of quotient by principal units}
	The natural map $\Hom (G_0, K^\times )[p'] \ra \Hom \big( G_0 , K^\times / \CO^{\times \times}_K \big) $ induced by the quotient map $K^\times \twoheadrightarrow K^\times / \CO^{\times \times}_K$ is an isomorphism of topological groups.
	It gives rise to an isomorphism
	\begin{equation*}
		\ul{\Hom} (G_0, K^\times ) \cong  \ul{\Hom} \big( G_0 ,\CO^{\times \times}_K \big) \oplus \ul{\Hom}(G_0, K^\times)[p'] 
	\end{equation*}
	of solid abelian groups. 
	The analogous assertions hold for ${}^\conjelt G_0$ and $I$.
\end{lemma}
\begin{proof}
	We first observe that the image of every $\chi \in \Hom \big( G_0 , K^\times / \CO_K^{\times \times} \big)$ is finite because $G_0$ is compact and $K^\times / \CO_K^{\times \times}$ discrete. 
	Since $K^\times / \CO^{\times \times}_K$ is $p$-torsion-free, it follows that $\Hom \big( G_0 , K^\times / \CO_K^{\times \times} \big) = \Hom \big( G_0 , K^\times / \CO_K^{\times \times} \big)[p']$.

	In the proof of \cite[Lemma 2.2.3]{ArdakovWadsley23EquivLineBun}, it is shown that every $\chi \in \Hom (G_0 , A)[p']$, for any abelian group $A$, factors over $\widehat{\det} \colon G_0 \ra \mu_{q-1}(F)$.
	In particular, the horizontal maps in the commutative square
	\begin{equation*}
		\begin{tikzcd}[row sep= scriptsize]
			\Hom (G_0 , K^\times) [p'] \ar[r, "\sim"] \ar[d] & \Hom ( \mu_{q-1}(F), K^\times ) \ar[d] \\
			\Hom \big( G_0 , K^\times / \CO_K^{\times \times} \big) \ar[r, "\sim"] & \Hom\big( \mu_{q-1} (F), K^\times / \CO_K^{\times \times} \big)
		\end{tikzcd}
	\end{equation*}
	are isomorphisms.
	But the vertical map on the right hand side is an isomorphism, as $(\blank)^{q-1}$ is an automorphism of $\CO^{\times \times}_K$.

	We also have the canonical exact sequence of abelian groups
	\begin{equation*}
		1 \lra \Hom \big(G_0 , \CO^{\times \times}_K \big) \lra \Hom ( G_0 , K^\times ) \lra \Hom \big( G_0, K^\times / \CO_K^{\times \times} \big) .
	\end{equation*}
	The isomorphism $\Hom \big( G_0 , K^\times / \CO_K^{\times \times} \big) \overset{\sim}{\ra} \Hom (G_0 , K^\times) [p'] $ then yields a section to the rightmost map.
	The latter therefore is a (necessarily strict) epimorphism and the above sequence splits.
	With respect to the subspace topologies this yields the claimed decomposition.

	For ${}^\conjelt G_0$ and $I$ one argues analogously.
	Here one uses that every $\chi \in \Hom (I , A)[p']$, for any abelian group $A$, factors over $I \ra \mu_{q-1}(F)^2$, $\big( \begin{smallmatrix} a & b \\ \unif c & d  \end{smallmatrix} \big) \mto (\widehat{a}, \widehat{d})$, see \textit{loc.\ cit.}.
\end{proof}

\begin{proposition}\label{Prop - Decomposition of principal units and constants}
	For $n\in\BN$, the natural inclusions induce isomorphisms of solid abelian groups
	\begin{equation*}
		\arraycolsep=1.4pt
		\def\arraystretch{1.2}
	\begin{array}{rcrcl}
		\ul{H}^1 \big( {G_0} , K^\times \CO^{\times \times}(\Omega_n) \big) &\cong &\ul{H}^1 \big( {G_0} , \CO^{\times \times}(\Omega_n) \big) &\oplus& \ul{\Hom} ( G_0 , K^\times )[p'] \\
		\ul{H}^1 \big( {I} , K^\times \CO^{\times \times}(\Omega'_n) \big) &\cong &\ul{H}^1 \big( {I} , \CO^{\times \times}(\Omega'_n) \big) &\oplus& \ul{\Hom} ( I , K^\times )[p'] .
	\end{array}
	\end{equation*}
\end{proposition}
\begin{proof}
	We consider the short strictly exact sequence
	\begin{equation*}
		1 \lra \CO^{\times \times} (\Omega_n) \lra K^\times \CO^{\times \times}(\Omega_n) \overset{\tau}{\lra} K^\times / \CO^{\times \times}_K \lra 1 
	\end{equation*}
	of topological $G_0$-modules.
	It gives rise to a long exact sequence of solid abelian groups
	\begin{align*}
		1 \lra &\, \CO^{\times \times} (\Omega_n)^\ul{G_0} \lra  \big( K^\times \CO^{\times \times}(\Omega_n) \big)^\ul{G_0} \xrightarrow{\tau^\ul{G_0}}  \big( K^\times / \CO^{\times \times}_K \big)^\ul{G_0} \lra \ul{H}^1 \big( {G_0} , \CO^{\times \times} (\Omega_n) \big)  \\
			&\lra \ul{H}^1 \big( {G_0} , K^\times \CO^{\times \times} (\Omega_n) \big) \xrightarrow{\ul{H}^1({G_0}, \tau)} \ul{H}^1 \big( {G_0}, K^\times / \CO^{\times \times}_K \big) \lra \ldots .
	\end{align*}
	By \Cref{Lemma - Invariant function is constant} we have $\CO^{\times \times} (\Omega_n)^{G_0} = \CO^{\times \times}_K$ and $ \big( K^\times \CO^{\times \times}(\Omega_n) \big)^{G_0}  = K^\times$.
	Therefore, the homomorphism $\tau^{G_0} \colon \big( K^\times \CO^{\times \times}(\Omega_n) \big)^{G_0} \ra \big( K^\times / \CO^{\times \times}_K \big)^{G_0} = K^\times / \CO^{\times \times}_K$ is a quotient map.
	We deduce the exact sequence of solid abelian groups
	\begin{equation*}
		1 \lra \ul{H}^1 \big( {G_0} , \CO^{\times \times} (\Omega_n) \big) \lra \ul{H}^1 \big( {G_0} , K^\times \CO^{\times \times} (\Omega_n) \big) \xrightarrow{\ul{H}^1({G_0}, \tau)} \ul{H}^1 \big( {G_0}, K^\times / \CO^{\times \times}_K \big) .
	\end{equation*}

	Using \Cref{Lemma - Characters of quotient by principal units} we have a homomorphism of solid abelian groups
	\begin{equation*}
		\ul{H}^1 \big( {G_0}, K^\times/ \CO^{\times \times}_K \big) \cong \ul{\Hom} (G_0 , K^\times)[p'] \longhookrightarrow \ul{H}^1 ({G_0}, K^\times ) \lra \ul{H}^1 \big( {G_0} , K^\times \CO^{\times \times} (\Omega_n) \big) 
	\end{equation*}
	which is a section to $\ul{H}^1(G_0, \tau)$. 
	This shows that the latter is an epimorphism and we obtain the split exact sequence of solid abelian groups
	\begin{equation*}
		1 \lra \ul{H}^1 \big( {G_0} , \CO^{\times \times} (\Omega_n) \big) \lra \ul{H}^1 \big( {G_0} , K^\times \CO^{\times \times} (\Omega_n) \big) \lra \ul{\Hom} (G_0, K^\times ) [p'] \lra 1.
	\end{equation*}

	The argument for $\ul{H}^1 \big( {I} , K^\times \CO^{\times \times}(\Omega'_n) \big) $ is completely analogous.
\end{proof}

\subsection{Prime-to-$p$-Torsion Classes}

Our method to describe $H^1 \big( G_0, \CO^\times(\Omega_n) \big)[p']$ is essentially the one used by Ardakov and Wadsley in the proof of Thm.\ 4.4.1 in \cite{ArdakovWadsley23EquivLineBun}.
However, to introduce our notation and for the convenience of the reader we present the complete argument.

\begin{setting}\label{Set - Unramified quadratic extension valued point}
	Let $L$ denote the unramified quadratic extension of $F$ in $C$, and let $z\in L$ such that $[z: 1]$ is an $L$-valued point of $\Omega_{F,n}$\footnote{For example, take $z$ to be the  Teichmüller lift of an element of $\BF_{q^2}^\times \setminus \BF_q^\times$.}.
	We consider $[z:1]$ as a $K(z)$-valued point of $\Omega_n$ so that evaluating at $[z:1]$ yields a continuous group homomorphism $\ev_{[z:1]} \colon \CO^\times (\Omega_n) \ra K(z)^\times$. 
\end{setting}

In this situation, we consider the map
\begin{equation*}
	\iota_z \colon \CO_L^\times \lra G_0 \,,\quad a+ c z \lto M_{a,c} \,,\qquad\text{where $M_{a,c} \defeq \left( \begin{matrix} a& c\, \mathrm{N}(z) \\ -c & a + c\, \Tr(z) \end{matrix} \right)$},
\end{equation*}
and $\mathrm{N}$ and $\Tr$ denote the norm and trace maps of $L/F$ respectively.
Then $\iota_z$ is a continuous group homomorphism whose image stabilizes $[z:1]$, see Lemma 2.2.5 of \textit{loc.\ cit.}\footnote{In fact, $\iota_z$ induces an isomorphism between $\CO_L^\times$ and the stabiliser of $[z:1]$ in $G^0$.}.
By functoriality the pair $(\iota_z, \ev_{[z:1]})$ induces a homomorphism of solid abelian groups
\begin{equation*}
	\ul{H}^1 \big(G_0, \CO^\times (\Omega_n) \big) \lra  \ul{\Hom} \big( \CO^\times_L , K(z)^\times \big) \quad\text{with}\quad [\gamma] \lto \big[ x \mto  \gamma\big(\iota_z(x) \big)([z:1]) \big] 
\end{equation*}
on the underlying abelian groups.
Composing it with restriction to the group of $(q^2 -1)$-st roots of unity $\mu_{q^2 - 1} (L)$ of $L$, we obtain a homomorphism of solid abelian groups
\begin{equation}\label{Eq - Evaluation at quadratic point}
	\rho_{z,n} \colon \ul{H}^1 \big(G_0, \CO^\times (\Omega_n) \big)  \lra \ul{\Hom} \big( \mu_{q^2 - 1} (L) , K(z)^\times \big).
\end{equation}

\begin{lemma}[{Cf.\ \cite[Lemma 2.2.8]{ArdakovWadsley23EquivLineBun}}]\label{Lemma - Image of j under evaluation at quadratic point}
	\begin{altenumerate}
		\item
			The solid abelian group $\ul{\Hom} \big( \mu_{q^2 - 1} (L) , K(z)^\times \big)$ is discrete and cyclic of order $q^2 -1$. 
			Every element of its underlying abelian group is of the form
			\begin{equation*}
				\sigma_k \colon \mu_{q^2 - 1}(L) \lra K(z)^\times \,,\quad \zeta \lto \zeta^k \,,\quad\text{for some $k \in \BZ/ (q^2 -1)\, \BZ$.}
			\end{equation*}

		\item 
			We have $\rho_{z,n} ( [j] ) =  \sigma_1$.
		
	\end{altenumerate}
\end{lemma}
\begin{proof}
	It is clear that the topological abelian group ${\Hom} \big( \mu_{q^2 - 1} (L) , K(z)^\times \big)$ is cyclic of the claimed shape and discrete.
	Thus, (i) follows.

	For (ii), let the linear form $\ell\colon K^2 \ra K$, $(x,y) \mto y$, represent $[j]$. 
	One computes that 
	\begin{equation*}
		\rho_{z,n}([j]) (a +cz) = \frac{\ell \big( M_{a,c}^\ad \cdot ( \begin{smallmatrix} z \\ 1\end{smallmatrix} ) \big)}{ \ell (z,1) } = \frac{a + cz}{1} \,,
	\end{equation*}
	for all $a+cz \in \CO_L^\times$.
	Therefore, $\rho_{z,n}([j])$ indeed is equal to $\sigma_1$.
\end{proof}

\begin{proposition}[{Cf.\ \cite[Thm.\ 4.4.1]{ArdakovWadsley23EquivLineBun}}]\label{Prop - Prime to p torsion of group cohomology for affinoid}
	The group $H^1 \big(G_0, \CO^\times(\Omega_n) \big)[p']$ is cyclic of order $q^2-1$. 
	More concretely, it possesses a unique generator $[\alpha_n] $ such that the homomorphism of solid abelian groups
	\begin{equation*}
		\ul{\Hom} \big( \mu_{q^2 - 1} (L) , K(z)^\times \big)  \lra \ul{H}^1 \big(G_0, \CO^\times(\Omega_n) \big) 
	\end{equation*}
	induced by mapping $\sigma_{q^n} \mto [\alpha_n]$ on the underlying abelian groups is a section to $\rho_{z,n}$.
	Furthermore, we have $\widetilde{P_{n,\ast}}[\alpha_n] = q^n \mod (q^n (q+1))$.
\end{proposition}
\begin{proof}
	Again, we consider $\ell\colon K^2 \ra K$, $(x,y) \mto y$, and $U \defeq \{(x,0) \mid x \in F \}$.
	Then
	\begin{equation*}
		B_0 \defeq \left\{ \bigg( \begin{matrix} a & b \\ 0 & d		\end{matrix} \bigg) \middle{|} a,d \in \CO_F^\times, b \in \CO_F \right\} \subset G_0
	\end{equation*}
	stabilises $U$.
	Moreover, for $G_{n+1} \defeq 1 + \unif^{n+1} M_2(\CO_F)$, the subgroup $G_{n+1} B_0$ is of index $q^n (q+1)$ in $G_0$.
	Let $g_1,\ldots, g_{q^n (q+1)}$ be a full system of representatives of $G_0$ modulo $G_{n+1} B_0$.
	We consider the continuous $1$-cocycle
	\begin{equation*}
		\beta_n \colon G_0 \lra \CO^\times (\Omega_n) \,,\quad g \lto \frac{g \cdot f}{f} \,\, j(g)^{-q^n (q+1)} = \prod_{i=1}^{q^n (q+1)} \frac{g_i \cdot \ell}{gg_i \cdot \ell} ,
	\end{equation*}
	where $f \defeq \prod_{i=1}^{q^n (q+1)} \frac{\ell}{g_i \cdot \ell}$.
	Our claim is that $\beta_n$ takes values in $\Ker(P_n) = K^\times \CO^{\times \times}(\Omega_n)$.
	We have $gg_i = g_{\sigma(i)} h_i  $, for some permutation $\sigma$ of $ \{1,\ldots ,q^n(q+1)\}$ and  $h_i \in G_{n+1}  B_0$.
	Since elements of $B_0$ change $\ell$ only by a scalar, it suffices to show that $\frac{\ell}{h \cdot \ell} \in \Ker(P_n)$, for all $h\in G_{n+1}$.
	But this holds because for such $h$ the straight path between $U$ and $h(U)$ does not intersect $\CT_n$, together with \Cref{Example - Van der Put transform of quotient of linear forms}.

	Composing $\beta_n$ with the quotient map to $K^\times \CO^{\times \times} (\Omega_n) /  \CO^{\times \times} (\Omega_n) $, we obtain a continuous group homomorphism $\ov{\beta_n} \colon G_0 \ra K^\times / \CO_K^{\times \times}$. 
	As seen in the proof of \Cref{Lemma - Characters of quotient by principal units}, $\ov{\beta_n}^{q-1}$ then is trivial.
	In other words, the image of $\beta_n^{q-1}$ is contained in $\CO^{\times \times}(\Omega_n)$.

	By \Cref{Lemma - Zp-module structure on principal units} there exists a unique $(q^2 -1)$-st root of $\beta_n^{q-1}(g)$, for every $g\in G_0$.
	In this way, we obtain a continuous $1$-cocycle $\gamma_n \colon G_0 \ra \CO^{\times \times}(\Omega_n)$ such that $\gamma_n^{q^2 - 1} = \beta_n^{q-1}$ and define
	\begin{equation*}
		\alpha_n \defeq \gamma_n \, j^{q^n} .
	\end{equation*}
	As $\beta_n^{q-1} \, j^{q^n (q^2 -1)} $ is a $1$-coboundary, the class $[\alpha_n] \in H^1 \big( G_0, \CO^\times(\Omega_n) \big)$ is $(q^2 - 1)$-torsion.

	Because $\gamma_n$ takes values in $\CO^{\times \times}(\Omega_n)$, the homomorphism $\rho_{z,n}([\gamma_n])$ takes values in $\CO_{K(z)}^{\times \times}$.
	\Cref{Lemma - Image of j under evaluation at quadratic point} (i) then implies that $[\gamma_n] \in \Ker(\rho_{z,n})$ so that $\rho_{z,n}([\alpha_n]) = \rho_{z,n}([j]^{q^n}) = \sigma_{q^n}$. 
	But since $q^n$ and $q^2 - 1$ are coprime, $\sigma_{q^n}$ is a generator of $\Hom \big( \mu_{q^2 - 1}(L), K(z)^\times \big)$.
	Hence, the order of $[\alpha_n]$ is a multiple of $q^2 - 1$.

	On the other hand, we may pass to the prime-to-$p$ torsion of the short exact sequence of underlying abelian groups of \eqref{Eq - Short exact sequence for group cohomology for affinoids}.
	\Cref{Prop - Decomposition of principal units and constants} together with \Cref{Lemma - Prime to p torsion characters} (ii) then show that the order of $H^1 \big(G_0, \CO^\times(\Omega_n) \big)[p']$ divides $q^2-1$. 
	In total, we deduce that $H^1 \big(G_0, \CO^\times(\Omega_n) \big)[p']$ is cyclic of order $q^2-1$ with $[\alpha_n]$ as a generator.

	We obtain a homomorphism of solid abelian groups
	\begin{equation*}
		\ul{\Hom} \big( \mu_{q^2 - 1} (L) , K(z)^\times \big)  \lra \ul{H}^1 \big(G_0, \CO^\times(\Omega_n) \big) 
	\end{equation*}
	which is defined via mapping the generator $\sigma_{q^n}$ to $[\alpha_n]$ on global sections.
	This homomorphism is a section to $\rho_{z,n}$.

	The uniqueness of $[\alpha_n]$ then follows because $\rho_{z,n}$ restricts to an isomorphism between $H^1 \big(G_0, \CO^\times(\Omega_n) \big)[p']$ and $ {\Hom} \big( \mu_{q^2 - 1} (L) , K(z)^\times \big) $.
	Finally, as $\gamma_n$ takes values in the kernel of $P_n$, we also have 
	\begin{equation*}
		\widetilde{P_{n, \ast}}[\alpha_n] = \widetilde{P_{n, \ast}}\big( [j]^{q^n}\big) = q^n \mod ( q^n(q+1) ) .
	\end{equation*}
\end{proof}

\begin{corollary}
	We have $[\alpha_n]^{q+1} = \widehat{\det}$.
\end{corollary}
\begin{proof}
	It follows from \eqref{Eq - Short exact sequence for group cohomology for affinoids} that $[\alpha_n]^{q+1}$ is contained in $H^1 \big(G_0, K^\times \CO^{\times \times}(\Omega_n) \big)[p']$, and therefore equal to $\widehat{\det}^k$, for some $k \in \BZ/ (q-1)\BZ$.

	Since $\det( M_{a,c}) = \mathrm{N}(a+cz)$ and $\mathrm{N}(\zeta) = \zeta^{q+1}$ for any $\zeta \in \mu_{q^2 - 1}(L)$, we find that $\rho_{z,n} \big(\widehat{\det} \big)= \sigma_{q+1}$.
	On the other hand, \Cref{Prop - Prime to p torsion of group cohomology for affinoid} implies that $\rho_{z,n} \big( [\alpha_n]^{q+1} \big) = \sigma_{q^n(q+1)}$.
	It follows that $k=1$ because $\zeta^{q^n (q+1)} = \zeta^{q+1}$, for every $\zeta \in \mu_{q^2 -1}(L)$.
\end{proof}

We now assemble the $(q^2-1)$-torsion classes for $\CO^\times(\Omega_n)$ constructed above to a $(q^2-1)$-torsion class for $\CO^\times (\Omega)$.

\begin{lemma}
	The family $\big( [\alpha_n]^{q^n} \big)_{n\in \BN}$ is compatible under the canonical transition maps 
	\begin{equation*}
		 H^1 \big( G_0, \CO^\times (\Omega_{n+1} ) \big) \lra H^1 \big( G_0, \CO^\times (\Omega_{n}) \big) \,,\quad c \lto c\res{G_0,n} .
	\end{equation*}
\end{lemma}
\begin{proof}
	We consider $[z:1] \in \Omega_0(K(z))$ as in \Cref{Set - Unramified quadratic extension valued point}.
	Then the homomorphisms from \eqref{Eq - Evaluation at quadratic point} satisfy $\rho_{z,n+1} = \rho_{z,n} \circ \big(\blank\res{G_0,n}\big)$, for all $n\in \BN$.
	Since $\zeta^{q^{2n+2}} = \zeta^{q^{2n}}$, for every $\zeta \in \mu_{q^2 - 1}(L)$, and by \Cref{Prop - Prime to p torsion of group cohomology for affinoid} we have 
	\begin{equation*}
		\rho_{z,n+1} \big([\alpha_{n+1}]^{q^{n+1}} \big) = \sigma_{q^{2n+2}} = \sigma_{q^{2n}} = \rho_{z,n} \big([\alpha_{n}]^{q^{n}} \big) .
	\end{equation*}
	The compatibility now follows because $\rho_{z,n}$ is an isomorphism when restricted to the prime-to-$p$-torsion of $H^1 \big( G_0, \CO^\times (\Omega_{n}) \big)$.
\end{proof}

\begin{definition}\label{Def - Certain torsion cocycle}
	We write $[\alpha]$ for the \emph{inverse}\footnote{We choose to work with the inverse here because of notational convenience later on.} of the class defined by the family $\big( [\alpha_n]^{q^n} \big)_{n\in \BN}$ in $H^1 \big( G_0, \CO^\times (\Omega) \big) = \varprojlim_{n\in \BN} H^1 \big( G_0, \CO^\times (\Omega_n) \big)$.
	This class satisfies $[\alpha]^{q+1} = \widehat{\det}{}^{-1}$ and $\widetilde{P_\ast}[\alpha] = \big( 0, -1 \mod (q+1) \big)$.
\end{definition}

\subsection{Non-Torsion Classes}

\begin{notation}
	In the following, we let a subscript indicate the restriction of group cohomology classes with respect to the subgroup of $G$ and the ``level'', for example
	\begin{alignat*}{3}
	 	H^1 \big( G^0 , \CO^\times (\Omega) \big) &\lra H^1 \big( G_0, \CO^\times(\Omega) \big) &&\lra H^1 \big( G_0, \CO^\times(\Omega_n) \big)  , \\
	 	 c &\lto c\res{G_0} &&\lto c\res{G_0,n} .
	\end{alignat*}
	We further abbreviate $[\gamma]_{G_0,n} \defeq \big( [\gamma] \big)\res{G_0,n}$ et cetera.
\end{notation}

\begin{lemma}\label{Lemma - Power of certain element is principal unit group cohomology class}
	For $n \in \BN$, the class $[j\alpha]^{q^n}_{G_0,n}$ is contained in $H^1 \big( G_0, \CO^{\times \times}(\Omega_n) \big)$.
\end{lemma}
\begin{proof}
	By the construction of $[\alpha]$, we have $[j \alpha]_{G_0,n} = \big[ \gamma_n^{-q^n} j^{1-q^{2n}}\big] $ where the $1$-cocycle $\gamma_n$ takes values in $\CO^{\times \times}(\Omega_n)$.
	The short exact sequence
	\begin{equation*}
		1 \lra H^1 \big( G_0, K^\times \CO^{\times \times}(\Omega_n) \big) \lra H^1 \big( G_0 , \CO^\times(\Omega_n) \big) \xrightarrow{ \widetilde{P_{n, \ast}} } \BZ / q^n (q+1) \, \BZ \lra 0 
	\end{equation*}
	shows that $[j]^{q^n(q+1)} \in H^1 \big( G_0, K^\times \CO^{\times \times}(\Omega_n) \big)$. 
	In view of the decomposition \Cref{Prop - Decomposition of principal units and constants} it follows that $[j]^{q^n(q^2-1)} \in H^1 \big( G_0, \CO^{\times \times}(\Omega_n) \big) $. 
	This shows the claim because $q^{2n}-1$ is divisible by $q^2 -1$.
\end{proof}

\begin{proposition}\label{Prop - Zp-family in group cohomology of affinoid}
	There exists a homomorphism of solid abelian groups
	\begin{equation*}
		\BZ_p \lra \ul{H}^1 \big( G_0, \CO^\times(\Omega_n) \big) \qquad\text{with}\qquad 1 \lto [j\alpha]_{G_0,n}
	\end{equation*}
	on the underlying abelian groups.
	Its composition with
	\begin{equation*}
		\ul{H}^1 \big( G_0, \CO^\times(\Omega_n) \big) \xrightarrow{\widetilde{P_{n, \ast}}} \BZ/ q^n(q+1)\, \BZ \longtwoheadrightarrow \BZ/ q^n \, \BZ
	\end{equation*}
	is equal to the canonical projection $\BZ_p \twoheadrightarrow \BZ/ q^n \, \BZ$.
\end{proposition}
\begin{proof}
	First, there is a homomorphism $\BZ \ra \ul{H}^1 \big( G_0, \CO^\times(\Omega_n) \big)$ of solid abelian groups induced by the map $1 \mto [j \alpha]_{G_0,n}$ on global sections, i.e.\ for the underlying abelian groups. 
	Because of \Cref{Lemma - Power of certain element is principal unit group cohomology class} and \Cref{Lemma - Zp-module structure on group cohomology of principal units}, we also have a homomorphism 
	\begin{equation*}
		\BZ_p \lra \ul{H}^1 \big( G_0, \CO^{\times \times}(\Omega_n) \big) \longhookrightarrow \ul{H}^1 \big( G_0, \CO^\times(\Omega_n) \big)
	\end{equation*}
	of solid abelian groups defined by mapping $1 \mto [j\alpha]_{G_0,n}^{q^n}$ on global sections.

	The cokernel of $\BZ \ra \BZ \oplus \BZ_p$, $m \mto (q^n m, -m)$, is given by the homomorphism $\BZ \oplus \BZ_p \ra \BZ_p$, $(k,\lambda) \mto k + q^n \lambda$, (on the level of topological abelian groups and of condensed abelian groups by \Cref{Lemma - Strictly exact sequence gives exact sequence of condensed modules}).
	Therefore, the above maps together induce the sought homomorphism
	\begin{equation*}
		\BZ_p \lra \ul{H}^1 \big( G_0, \CO^\times(\Omega_n) \big) \,, \qquad\text{with}\qquad \lambda = k + q^n \mu \lto [j\alpha]_{G_0,n}^k \, [j\alpha]_{G_0,n}^{q^n \mu} ,
	\end{equation*}
	where $k \in \BZ$ and $\mu \in \BZ_p$.
	Since $\widetilde{P_{n,\ast}}[j\alpha]_{G_0,n} = \widetilde{P_{n, \ast}}[j \alpha_n^{-q^n}]  = 1 - q^{2n} \mod ( q^n(q+1) )$, the last claim is obvious.	
\end{proof}

\begin{corollary}\label{Cor - Zp-family in group cohomology}
	There exists a homomorphism of solid abelian groups
	\begin{equation*}
		\BZ_p \lra \ul{H}^1 \big( G_0 , \CO^\times(\Omega) \big) \qquad\text{with}\qquad 1 \lto [ j \alpha] 
	\end{equation*}
	on the underlying abelian groups, which is a section to 
	\begin{equation*} 
		\ul{H}^1 \big( G_0, \CO^\times(\Omega) \big) \overset{\widetilde{P_{\ast}}}{\lra} \BZ_p \times \BZ/ (q+1)\,\BZ \longtwoheadrightarrow \BZ_p .
	\end{equation*}
\end{corollary}
\begin{proof}
	We take the inverse limit of the homomorphisms constructed in \Cref{Prop - Zp-family in group cohomology of affinoid}.
	The second assertion follows since $\widetilde{P_\ast} = \varprojlim_{n\in \BN} \widetilde{P_{n,\ast}}$, see \Cref{Thm - Left exact sequences for group cohomology}.
\end{proof}

\begin{theorem}\label{Thm - Group cohomology for G_0}
	There is an isomorphism of solid abelian groups
	\begin{align*}
		\BZ_p \,\oplus\, \BZ/ (q^2- 1)\, \BZ \,\oplus\, \ul{\Hom} \big( G_0, \CO_K^{\times \times} \big) &\overset{\sim}{\lra} \ul{H}^1 \big( G_0 , \CO^\times(\Omega) \big)	  \qquad\text{with} \\
		 (\lambda , k , \psi ) &\lto		[ j \alpha]^\lambda \, [\alpha]^k \, \psi 
	\end{align*}
	on the underlying abelian groups.
\end{theorem}
\begin{proof}
	Let $[z:1] $ be a $K(z)$-valued point of $ \Omega_0$ as in \Cref{Set - Unramified quadratic extension valued point}.
	By \Cref{Prop - Prime to p torsion of group cohomology for affinoid} the mapping $\sigma_{-1} \mto [\alpha]$ induces a homomorphism of solid abelian groups
	\begin{equation*} 
		\BZ / (q^2 -1) \, \BZ \cong \ul{\Hom} \big( \mu_{q^2 - 1}(L), K(z)^\times \big) \lra \ul{H}^1 \big( G_0, \CO^\times (\Omega) \big) .
	\end{equation*}
	We write $[\alpha]^{\BZ / (q^2 -1) \, \BZ}$ for its image.
	Moreover, the homomorphisms $\rho_{z,n}$ from \eqref{Eq - Evaluation at quadratic point} for varying $n\in \BN$ are compatible and thus yield $\rho_z \colon \ul{H}^1 \big( G_0, \CO^\times (\Omega) \big) \ra \ul{\Hom} \big( \mu_{q^2 - 1}(L), K(z)^\times \big)$.
	The above homomorphism is a section to $\rho_z$ and therefore $[\alpha]^{\BZ / (q^2 -1) \, \BZ}$ is a direct summand of $\ul{H}^1 \big( G_0, \CO^\times (\Omega) \big)$.
	Likewise, the image $[j\alpha]^{\BZ_p}$ of the homomorphism from \Cref{Cor - Zp-family in group cohomology} is a direct summand as well.
	We abbreviate the sum of these two direct summands by $A$.

	Now, recall the second short exact sequence of solid abelian groups from \eqref{Eq - Short exact sequences for group cohomology of G^0 and G_0}.
	We obtain from it the short exact sequence
	\begin{align*}
		1 \lra \ul{\Hom} ( G_0, K^\times ) \Big/ \ul{\Hom} ( G_0, K^\times ) \cap A
			 \lra &\, \, \ul{H}^1 \big( G_0, \CO^\times (\Omega) \big) \Big/ A \\
			 &\overset{\widetilde{P_\ast}}{\lra} \big( \BZ_p \oplus \BZ/(q+1)\, \BZ \big) \Big/ \widetilde{P_\ast}(A) 
			 \lra 0 .
	\end{align*}
	Since $\widetilde{P_\ast}[\alpha] = \big( 0, -1 \mod (q+1) \big)$ and $\widetilde{P_\ast}[j\alpha] = \big( 1, 0 \mod (q+1) \big)$, we see that the last term vanishes.
	Moreover, we have $\ul{\Hom} ( G_0, K^\times ) \cap A = \ul{\Hom} (G_0, K^\times )[p']$ because of \Cref{Lemma - Characters of quotient by principal units} and $[\alpha]^{q+1} = \widehat{\det}{}^{-1}$.
	Hence the above short exact sequence simplifies to an isomorphism 
	\begin{equation*}
		\ul{\Hom} \big( G_0, \CO_K^{\times\times} \big) \overset{\sim}{\lra} \ul{H}^1 \big( G_0, \CO^\times (\Omega) \big) \Big/ A  
	\end{equation*}
	of solid abelian groups.
	In total, this proves the claimed decomposition.
\end{proof}

\begin{remark}
	For $q>2$, every $\chi \in \Hom( G_0, K^\times) $ factors over $\det\colon G_0 \ra F^\times$ so that $\Hom \big( G_0, \CO_K^{\times \times} \big) \cong \Hom \big( \CO_F^{\times \times}, \CO_K^{\times \times} \big)$ in this case, see \cite[Ch.\ 16, Thm.\ 1.7 (iii)]{Karpilovsky93GrpRepVol2}. 
\end{remark}

\section{Lifting of Cohomology Classes}

\subsection{Prime-to-$p$-Torsion Classes for $G^0$}

Recall that $G^0$ is equal to the free amalgamated sum $G_0 \ast_I {}^\conjelt G_0$ as an abstract group \cite[II.1.4, Thm.\ 3]{Serre80Trees}.
Since $G_0 \subset G^0$ is an open subgroup, it follows that $G^0$ is the pushout of $G_0$ and ${}^\conjelt G_0$ along $I$ in the category of topological groups as well. 
\Cref{Prop - Condensation of pushout} therefore implies the equality $\ul{G^0} = \ul{G_0} \ast_\ul{I} \ul{ {}^\conjelt G_0}$ for the associated condensed groups.

We may thus apply \Cref{Cor - Mayer-Vietoris sequence} to the short strictly exact sequence of $G^0$-modules induced by the Van der Put transform $P$.
We obtain the commutative diagram
\begin{equation}\label{Eq - Diagram induced by Mayer-Vietoris}
	\begin{tikzcd}[row sep= scriptsize]
				& 1 \ar[d]	& 1 \ar[d]	& 1 \ar[d]	& \\
			1 \ar[r]	& \ul{\Hom} ( {G^0} , K^\times ) \ar[r]\ar[d] 	& \ul{H}^1 \big( {G^0} , \CO^\times (\Omega) \big) \ar[r, "P_\ast"] \ar[d] 	& \ul{H}^1 \big({G^0}, \Cur(\bE, \BZ) \big) \ar[d] & \\
				&\ul{\Hom} ( {G_0} , K^\times )& \ul{H}^1 \big({G_0} , \CO^\times (\Omega) \big)& \ul{H}^1 \big( {G_0}, \Cur(\bE, \BZ) \big)& \\[-19pt]
			1 \ar[r, shorten >= 28pt] &\oplus \ar[r, shorten <= 28pt, shorten >= 32pt] & \oplus \ar[r, shorten <= 32pt, shorten >= 34pt] & \oplus \ar[r,shorten <= 34] & 0 \\[-19pt]
			&\ul{\Hom} (  {{}^\conjelt G_0} , K^\times ) \ar[d, "\Xi"]& \ul{H}^1 \big( {{}^\conjelt G_0} , \CO^\times (\Omega) \big) \ar[d]& \ul{H}^1 \big( {{}^\conjelt G_0}, \Cur(\bE, \BZ) \big) \ar[d]& \\
			1 \ar[r]	& \ul{\Hom} (  {I} , K^\times ) \ar[r]	& \ul{H}^1 \big( {I} , \CO^\times (\Omega) \big) \ar[r] 	& \ul{H}^1 \big( {I}, \Cur(\bE, \BZ) \big)  & .
	\end{tikzcd}
\end{equation}

\begin{lemma}
	All rows and columns of \eqref{Eq - Diagram induced by Mayer-Vietoris} are exact.
\end{lemma}
\begin{proof}
	The exactness of the first two rows is part of \Cref{Thm - Left exact sequences for group cohomology} (resp.\ is so after applying $\conjelt^\ast$ to pass from $G_0$ to ${}^\conjelt G_0$).
	The last row is extracted from the long exact sequence associated to the Van der Put transform using $\ul{H}^0 \big( I, \Cur ( \bE, \BZ )\big) = \{0\}$ from \Cref{Cor - Cohomology of currents for Iwahori} analogously to before.

	Concerning the columns, the exact sequence from \Cref{Cor - Mayer-Vietoris sequence} for $\CO^\times(\Omega)$ reads as follows
	\begin{align*}
		1 \lra \ul{H}^0 \big( G^0 , \CO^\times(\Omega) \big) \lra \ul{H}^0 \big( G_0 , &\,\CO^\times(\Omega) \big)  \oplus \ul{H}^0 \big( {}^\conjelt G_0 , \CO^\times(\Omega) \big) \\
			&\overset{\Lambda}{\lra} \ul{H}^0 \big( I, \CO^\times(\Omega) \big) \lra \ul{H}^1 \big( G^0 , \CO^\times(\Omega) \big) \lra \ldots .
	\end{align*}
	Then \Cref{Lemma - Invariant function is constant} implies that $\Lambda$ is equal to (the condensation of) the map $K^\times \oplus K^\times \ra K^\times$, $(x,y) \mto xy^{-1}$, and therefore an epimorphism.
	This shows the exactness of the middle and the first column.
	For the last column we again use that $\ul{H}^0 \big( I, \Cur (\bE, \BZ )\big) = \{0\}$.
\end{proof}

We now consider the short strictly exact sequence of topological $I$-modules induced from $P'_n$.
Via \Cref{Prop - Cohomology of currents on Iwahori finite subtree} we obtain from the associated long exact sequence of condensed group cohomology
\begin{equation}\label{Eq - Exact sequence for Iwahori on affinoid}
	\BZ\, \psi_n \overset{ \partial }{\lra} \ul{H}^1 \big( {I}, K^\times \CO^{\times \times}(\Omega'_n ) \big) \overset{\iota_\ast}{\lra} \ul{H}^1 \big({I}, \CO^{\times }(\Omega'_n ) \big) \overset{P'_{n,\ast}}{\lra} \ul{H}^1 \big( {I}, \Cur(\bE'_{n+1}, \BZ)  \big) .
\end{equation}

\begin{lemma}\label{Lemma - Image of edge map}
	The image of $\partial$ is contained in the condensed subgroup $\ul{H}^1 \big( {I}, \CO^{\times \times}(\Omega'_n ) \big)$ of $\ul{H}^1 \big( {I},K^\times \CO^{\times \times}(\Omega'_n ) \big)$.
\end{lemma}
\begin{proof}
	We have $\Hom_{\mathrm{Cond(Ab)}} ( \BZ \, \psi_n , A ) \cong A\big(\{\psi_n\} \big) \cong A(\ast)$ for any condensed abelian group $A$.
	Therefore, the homomorphism $\partial$ is uniquely determined by the image of $\psi_n$ under its underlying homomorphism of abelian groups.

	We now consider the commutative diagram
	\begin{equation*}
	\begin{tikzcd}[row sep= scriptsize]
		H^1 \big( I , K^\times \CO^{\times \times}(\Omega'_n ) \big) \ar[r, "\iota_\ast"] \ar[d, two heads, "\pr"'] & H^1 \big( I , \CO^{\times }(\Omega'_n ) \big) \\
		H^1 \big( I , K / \CO^{\times \times}_K \big) \ar[r, "\sim"] & \Hom \big( I , K^{\times } \big)[p'] \ar[u, hook] \ar[ul, hook'] .
	\end{tikzcd}
	\end{equation*}
	Here, the vertical maps are the canonical projections respectively inclusions (see \Cref{Prop - Decomposition of principal units and constants} and \Cref{Lemma - Embedding of characters}) and the bottom isomorphism is the one from \Cref{Lemma - Characters of quotient by principal units}.
	Because of $\iota_\ast (\partial( \psi_n) ) = 1$, this diagram implies that $\pr(\partial(\psi_n))=1$.
	We conclude via the decomposition from \Cref{Prop - Decomposition of principal units and constants}.
\end{proof}

\begin{proposition}\label{Prop - Shape of prime to p torsion class restricted to affinoid over edge}
	There exists $c \in H^1 \big( I , \CO^{\times \times}(\Omega'_0 ) \big)$ such that in $H^1 \big( I, \CO^\times(\Omega'_0) \big)$ we have $[\alpha]_{I, 0} =\chi_{-1,0} \, \iota_\ast ( c )$.
\end{proposition}
\begin{proof}
	Recall from the construction of $[\alpha]$ in \Cref{Prop - Prime to p torsion of group cohomology for affinoid} and \Cref{Def - Certain torsion cocycle} that its image $[\alpha]_{G_0,1} \in H^1 \big( G_0 , \CO^\times(\Omega_1) \big)$ is of the form $[\alpha]_{G_0,1} = \big[ \gamma_1^{-q} j^{-q^2}\big] $, where $\gamma_1 \colon G_0 \ra \CO^{\times \times} (\Omega_1)$ is a certain $1$-cocycle.

	Then $[\alpha]_{G_0,1}$ is mapped to $[\alpha]_{I, 0}$ under $H^1 \big( G_0 , \CO^\times(\Omega_1) \big) \ra H^1 \big( I , \CO^\times(\Omega'_0) \big)$.
	Since $[\gamma_1]$ is mapped into $\iota_\ast \big( H^1 \big( I, \CO^\times(\Omega'_0) \big) \big)$ and we have $\chi_{-1,0}^q = \chi_{-1,0}$, it suffices to find a continuous $1$-cocycle $\theta \colon G_0 \ra \CO^\times(\Omega_1)$ such that $[\theta] = [j]^{q}$ in $H^1 \big( G_0 , \CO^\times(\Omega_1) \big)$ and $\big( \chi_{-1,0} \, \theta  \big)(g)\res{\Omega'_0}$ is contained in $\CO^{\times \times}(\Omega'_0)$, for all $g\in I$.

	To this end, we consider the non-trivial, $F$-rational linear forms
	\begin{align*}
		\ell_\zeta \colon K^2 \ra K \,,\quad (x,y) \mto \unif \zeta x + y \,,\qquad \text{for $\zeta \in \mu_{q-1}(F)\cup \{ 0 \}$,}
	\end{align*}
	and define the $1$-cocycle
	\begin{equation*}
		\theta \colon G_0 \lra  \CO^\times (\Omega_1) \,,\quad g \lto \frac{g \cdot \ell_0}{\ell_0} \prod_{\zeta \in \mu_{q-1}(F)} \frac{g \cdot \ell_\zeta}{\ell_\zeta} .
	\end{equation*}
	Then $\theta$ certainly satisfies $[\theta] = [j_{\ell_0}] \prod_{\zeta \in \mu_{q-1}(F)} [j_{\ell_\zeta}] = [j]^q$.
	Furthermore, recall that under $[z:1] \mto z$ we can identify the affinoid subdomain $\Omega'_0 = \Omega_{e_0}$ with
	\begin{equation*}
		\left\{ z \in C \,\middle{|} \begin{array}{l}	\abs{z} \leq 1 \\
																					\forall \eta \in \big( \CO_F \! \setminus\! (\unif) \big) / (\unif) : \abs{z - \eta} \geq 1 \\
																					\forall \rho \in (\unif) / (\unif^2 ) : \abs{z-\rho} \geq \frac{1}{q}		\end{array}
		\right\} ,
	\end{equation*}
	see \cite[Sect.\ I.2.3]{BoutotCarayol91ThmCerednikDrinfeld} and \eqref{Eq - Identification for affinoid over edge}.
	In the following $\lVert \blank \rVert$ denotes the supremum norm of $ \Omega'_0$.

	For $[z:1] \in \Omega'_0$ and $g\in I$ with $g^{\ad} = \left( \begin{smallmatrix} a & b \\ \unif c & d \end{smallmatrix} \right)$, we have $g^{\ad} \cdot (z,1) = \big( az+b, \unif cz + d\big)$.
	Therefore, we compute that
	\begin{align*}
		\theta(g)([z:1]) 
			&= \frac{\unif c z +d }{1} \prod_{\zeta \in \mu_{q-1}(F)} \frac{\unif (\zeta a +c) z + \unif \zeta b + d}{ \unif \zeta z + 1} \\
			&= (\unif c z +d ) \prod_{\zeta \in \mu_{q-1}(F)} d \,\, \frac{ \unif \zeta \frac{a}{d} z + 1 + \unif \big( \frac{cz + \zeta b}{d} \big) }{\unif \zeta z + 1} .
	\end{align*}
	Because of $\lVert \unif c z \rVert \leq \frac{1}{q} \lVert z \rVert = \frac{1}{q}$, we have $d + \unif c z \equiv \widehat{d} \mod \CO^{\times \times}(\Omega'_0)$.
	Furthermore, for any $\zeta \in \mu_{q-1}(F)$ there exist unique $\zeta' \in \mu_{q-1}(F)$ and $x \in (\unif)$ such that $\zeta \frac{a}{d} = \zeta' (1 +x)$.
	Then
	\begin{equation}\label{Eq - Principal units term}
		\frac{ \unif \zeta \frac{a}{d} z + 1 + \unif \big( \frac{cz + \zeta b}{d} \big) }{\unif \zeta' z + 1}
			= \frac{ \unif \zeta'(1+x) z + 1 + \unif \big( \frac{cz + \zeta b}{d} \big) }{\unif \zeta' z + 1}
			= 1 + \frac{\unif \big( \zeta' x z + \frac{cz + \zeta b}{d} \big)}{\unif \zeta' z + 1} .
	\end{equation}
	Since $\lVert \unif \zeta' z \rVert = \frac{1}{q}$, it follows that $\lVert \unif \zeta' z +1 \rVert = \max \big( \lVert \unif \zeta' z \rVert , \lVert 1 \rVert \big) = 1$.
	We also have 
	\begin{equation*}
		\Big\lVert \, \unif \, \Big( \zeta' x z + \frac{cz + \zeta b}{d} \Big) \Big\rVert \leq \frac{1}{q} \max \bigg( \lVert \zeta' x z \rVert , \Big\lVert \frac{cz}{d} \Big\rVert , \Big\lVert \frac{\zeta b}{d} \Big\rVert \bigg) \leq \frac{1}{q}.
	\end{equation*}
	Therefore, \eqref{Eq - Principal units term} lies in $\CO^{\times \times}(\Omega'_0)$.
	Finally, $d^{q-1} \equiv 1 \mod (\unif)$ holds.
	In total, we conclude that $\theta(g)([z:1])  \equiv \widehat{d} = \chi_{-1,0}(g)^{-1} \mod \CO^{\times \times}(\Omega'_0)$ as claimed.
\end{proof}

\begin{proposition}\label{Prop - Lifting of prime-to-p torsion}
	There exists a unique torsion class in $H^1 \big( G^0 , \CO^\times (\Omega) \big)$ of order $(q^2 -1)$ that is mapped to $[\alpha] \in H^1 \big( G_0 , \CO^\times (\Omega) \big)$.
	We continue to denote this class by $[\alpha]$. 
	It moreover satisfies $\conjelt^\ast [\alpha] = [\alpha]^q $ with respect to the conjugation action of $\conjelt$ on $H^1 \big( G^0 , \CO^\times (\Omega) \big)$.
\end{proposition}
\begin{proof}
	Our argument for showing the existence of a lift of $[\alpha]$ is inspired by the proof of \cite[Cor.\ 4.4.5]{ArdakovWadsley23EquivLineBun}.
	In view of the exactness of the middle column of \eqref{Eq - Diagram induced by Mayer-Vietoris} it suffices to show that $\big( [\alpha] , \widehat{\det}{}^{-1} \, \conjelt^{\ast}[\alpha]^{-1} \big)$ lies in the kernel of 
	\begin{equation}\label{Eq - Sum of restriction to Iwahori map}
		H^1 \big( G_0 , \CO^\times (\Omega) \big) \oplus H^1 \big( {}^\conjelt G_0 , \CO^\times (\Omega) \big) \lra H^1 \big( I , \CO^\times (\Omega) \big) \,,\quad (c,c') \lto c\res{I} \, (c'\res{I})^{-1} .
	\end{equation}
	To do so, we first consider
	\begin{equation*}
		P_\ast \Big( [\alpha]_I \, \,\widehat{\det} \, (\conjelt^\ast[\alpha] )\res{I} \Big)
		%	= \big( P_\ast [\alpha] \big)\res{I} \, \big( P_\ast ( \conjelt^\ast[\alpha] ) \big)\res{I}
			= \Big( P_\ast[\alpha] \, \,\conjelt^\ast \big( P_\ast[\alpha] \big) \Big)\res{I} .
	\end{equation*}
	Since $\widetilde{P_{\ast}}[\alpha] = \big( 0, -1 \mod (q+1) \big) $ and $\conjelt^\ast \big( \widetilde{P_{\ast}}[\alpha] \big) = \big( 0, 1 \mod (q+1) \big) $ by \Cref{Prop - Cohomology of currents}, we conclude that $[\alpha]_I \, \,\widehat{\det}  \, (\conjelt^\ast[\alpha] )\res{I} $ lies in the kernel of $P_\ast$, i.e.\ in $\Hom(I, K^\times )$.

	In \Cref{Prop - Shape of prime to p torsion class restricted to affinoid over edge} we have seen that $[\alpha]_{I, 0} = \chi_{-1,0} \, \iota_\ast (c)$, for some $c \in H^1 \big( I, \CO^{\times \times}(\Omega'_0) \big)$. 
	As $[\alpha]_{I,0}$ is $(q^2-1)$-torsion, $c^{q^2 -1} $ is contained in $\Ker(\iota_\ast)$.
	Thus, via the exact sequence \eqref{Eq - Exact sequence for Iwahori on affinoid} there exists some $k\in \BZ$ such that $c^{q^2-1} = \partial( k \,\psi_0) \in H^1 \big( I, \CO^{\times \times}(\Omega'_0) \big)$, cf.\ \Cref{Lemma - Image of edge map}.
	Since $H^1 \big( I, \CO^{\times \times}(\Omega'_0) \big)$ is a $\BZ_p$-module and $\conjelt^\ast$ is $\BZ_p$-linear (see \Cref{Lemma - Zp-module structure on group cohomology of principal units}), it follows that
	\begin{equation*}
		\conjelt^\ast c= \conjelt^\ast \Big(  \partial( k \,\psi_0)^\frac{1}{q^2-1} \Big) = \partial \big( k \, \conjelt^\ast \psi_0  \big)^\frac{1}{q^2-1}  =  \partial ( -k \, \psi_0 )^\frac{1}{q^2-1}= c^{-1} .
	\end{equation*}
	Here we have used that $\conjelt^\ast \psi_0 = - \psi_0$, see \Cref{Prop - Cohomology of currents on Iwahori finite subtree}.
	We deduce that
	\begin{align*}
		\conjelt^\ast [ \alpha]_{I,0}
			= \conjelt^\ast (\chi_{-1,0}) \, \iota_\ast \big(  \conjelt^\ast (c )\big) 
			= \chi_{0,-1} \, \iota_\ast \big( c^{-1} \big)
			= \widehat{\det}{}^{-1} \, \iota_\ast \big( \chi_{1,0}\, c^{-1} \big)
			= \widehat{\det}{}^{-1}  \, [\alpha]_{I,0}^{-1}
	\end{align*}
	and hence indeed $[\alpha]_{I} \, \conjelt^\ast [ \alpha]_{I} = \widehat{\det}{}^{-1} $ by \Cref{Lemma - Embedding of characters}.

	Concerning the last assertion, we consider the commutative square
	\begin{equation*}
		\begin{tikzcd}[row sep= scriptsize]
			H^1 \big( G^0 , \CO^\times(\Omega) \big) \ar[r, "\conjelt^\ast"] \ar[d, "(\blank)\res{{}^\conjelt G_0}"'] & H^1 \big( G^0 , \CO^\times(\Omega) \big) \ar[d,"(\blank)\res{G_0}"] \\
			H^1 \big( {}^\conjelt G_0, \CO^\times(\Omega) \big) \ar[r, "\conjelt^\ast"]	& H^1 \big( G_0 , \CO^\times(\Omega) \big) .
		\end{tikzcd}
	\end{equation*}
	Because we have constructed the lift $[\alpha]$ as the pullback of $[\alpha]_{G_0}$ and $\widehat{\det}{}^{-1} \, \conjelt^\ast ( [\alpha]_{G_0} )^{-1}$ we have $[\alpha]_{{}^\conjelt G_0} =  \widehat{\det}{}^{-1} \, \conjelt^\ast([\alpha]_{G_0} )^{-1}$.
	Since $(\conjelt^\ast)^2 = \id$, the latter class is mapped to $\widehat{\det}{}^{-1} \, [\alpha]_{G_0}^{-1}$ under $\conjelt^\ast$.
	Because of $\widehat{\det}{}^{-1} = [\alpha]^{q+1}_{G_0}$, the commutativity of the above diagram implies that $(\conjelt^\ast[\alpha])\res{G_0} = [\alpha]^q_{G_0}$, and hence $\conjelt^\ast[\alpha] = [\alpha]^q$.	
\end{proof}

\begin{remark}
	Alternatively, to show that $[\alpha]$ lifts to $H^1 \big( G^0, \CO^\times(\Omega) \big)$ one can proceed as follows:
	Once it is established that $[\alpha]_I ( \conjelt^\ast [\alpha])\res{I}$ lies in $\Hom(I, K^\times)$, the fact that this element is $\conjelt^\ast$-invariant implies that $[\alpha]_I = \widehat{\det}{}^k ( \conjelt^\ast [\alpha]^{-1})\res{I}$, for some $k\in \BZ/(q-1)\,\BZ$, see \Cref{Lemma - Prime to p torsion characters} (iii).
	Then, one concludes via the exactness of the middle column of \eqref{Eq - Diagram induced by Mayer-Vietoris} again.

	Furthermore, it follows from a result of Taylor \cite[Cor.\ 7.4]{Taylor25CatLubinTateDrinfeldBun} that the canonical forgetful map
	\begin{equation*}
		\mathrm{PicCon}^{G^0}(\Omega)[p'] \overset{}{\lra} \Pic^{G^0}(\Omega)[p']
	\end{equation*}
	is an isomorphism, cf.\ \Cref{Rmk - Connection to Ardakov-Wadsley}.
	Hence, the assertion that $\conjelt^\ast [\alpha] = [\alpha]^q$ is equivalent to \cite[Thm.\ 4.4.11 (a)]{ArdakovWadsley23EquivLineBun}.
\end{remark}

\subsection{$G^0$-, $G^{(2)}$- and $\GL_2(F)$-equivariant Line Bundles}

We may apply the snake lemma to the middle and bottom row of \eqref{Eq - Diagram induced by Mayer-Vietoris}. 
It yields in particular a connecting homomorphism
\begin{equation}\label{Eq - Connecting homomorphism}
	\Delta \colon \ul{H}^1 \big( {G^0} , \Cur(\bE, \BZ) \big) \lra \Coker (\Xi)
\end{equation}
such that $\Ker(\Delta) = \Im(P_\ast)$.
To determine the precise image of $P_\ast$ in that diagram, we will therefore study the image of $\Delta$ more closely.

\begin{lemma}\label{Lemma - Image under connecting homomorphism}
	Under the connecting homomorphism $\Delta$ in \eqref{Eq - Connecting homomorphism}, the element $\delta\big( \1_{G^0 v_1} \big) \in  H^1 \big( G^0 , \Cur(\bE, \BZ) \big) $ is mapped to $\chi_{1,0} \mod \Im(\Xi)$.
\end{lemma}
\begin{proof}
	We trace the construction of the connecting homomorphism $\Delta$ of the snake lemma.
	First, we note that under
	\begin{equation*}
		H^1 \big( G^0 , \Cur(\bE, \BZ) \big) \lra H^1 \big( G_0 , \Cur(\bE, \BZ) \big) \oplus H^1 \big( {}^\conjelt G_0 , \Cur(\bE, \BZ) \big)
	\end{equation*}
	the class $\delta\big( \1_{G^0 v_1} \big)$ is mapped to $\big( \delta\big( \1_{G^0 v_1} \big)\res{G_0} \,,\, \delta\big( \1_{G^0 v_1} \big)\res{{}^\conjelt G_0} \big)$.
	We need to find a preimage of this latter element under 
	\begin{equation}\label{Eq - Van der Put map on direct sum of group cohomology for maximal compact and conjugate}
		H^1 \big( G_0 , \CO^\times(\Omega) \big) \oplus H^1 \big( {}^\conjelt G_0 ,  \CO^\times(\Omega) \big) \xrightarrow{P_\ast \oplus P_\ast} H^1 \big( G_0 , \Cur(\bE, \BZ) \big) \oplus H^1 \big( {}^\conjelt G_0 , \Cur(\bE, \BZ) \big) .
	\end{equation}
	Because the $\conjelt$-action permutes $\1_{G^0 v_1}$ and $- \1_{G^0 v_0}$, we have $\delta( \1_{G^0 v_1} )\res{{}^\conjelt G_0} = \conjelt^\ast \big( -\delta( \1_{G^0 v_0} )\res{G_0}  \big)$.
	It follows from \Cref{Prop - Cohomology of currents} and \Cref{Cor - Zp-family in group cohomology} that $P_\ast \big( [j\alpha]^{\frac{1}{q-1}}_{G_0} \big) = \delta\big( \1_{G^0 v_1} \big)\res{G_0}$ and $P_\ast \big( [j\alpha]^{\frac{1}{q-1}}_{G_0} [\alpha]_{G_0} \big) = - \delta\big( \1_{G^0 v_0} \big)\res{G_0}$.
	We thus deduce that under \eqref{Eq - Van der Put map on direct sum of group cohomology for maximal compact and conjugate}
	\begin{equation*}
		\Big( [j\alpha]^{\frac{1}{q-1}}_{G_0} \, , \, \conjelt^\ast \big( [j\alpha]^{\frac{1}{q-1}}_{G_0} [\alpha]_{G_0} \big)  \Big) \lto \big( \delta\big( \1_{G^0 v_1} \big)\res{G_0} \, , \, \delta\big( \1_{G^0 v_1} \big)\res{{}^\conjelt G_0} \big) .
	\end{equation*}

	The exactness of the bottom row of \eqref{Eq - Diagram induced by Mayer-Vietoris} implies that the image
	\begin{equation}\label{Eq - Definition of chi}
		\chi \defeq  [j\alpha]^{\frac{1}{q-1}}_I  \, \Big( \conjelt^\ast \big( [j\alpha]^{\frac{1}{q-1}}_{G_0}  \, [\alpha]_{G_0} \big) \Big)^{-1}\res{I}
	\end{equation}
	of this former element under \eqref{Eq - Sum of restriction to Iwahori map} is contained in $\Hom ( I, K^\times )$.
	Moreover, a direct computation shows that $\chi$ is $(q-1)$-torsion.
	By the definition of the connecting homomorphism, we now have
	\begin{equation*}
	 	\Delta \big( \delta\big( \1_{G^0 v_1} \big) \big) = \chi \mod \Im(\Xi) .
	\end{equation*}

	To show that $\chi = \chi_{1,0}$, we first recall that in $H^1 \big( G_0 , \CO^\times(\Omega_1) \big)$
	\begin{equation*}
		[j\alpha]_{G_0,1}^\frac{1}{q-1} = \prod_{i\geq 0} [j\alpha]_{G_0,1}^{-q^i} = [j\alpha]_{G_0,1}^{-1} \, \big( [j\alpha]^q_{G_0,1} \big)^\frac{1}{q-1} .
	\end{equation*}
	By \Cref{Lemma - Power of certain element is principal unit group cohomology class} the element $\big( [j\alpha]^q_{G_0,1} \big)^\frac{1}{q-1}$ is contained in the subgroup $H^1 \big( G_0 , \CO^{\times\times}(\Omega_1) \big)$.
	Via functoriality we therefore have $[j\alpha]_{I,0}^\frac{1}{q-1} = [j\alpha]_{I,0}^{-1} \, \, \iota_\ast (c)$, for some $c\in H^1 \big( I, \CO^{\times \times}(\Omega'_0) \big)$, and similarly for $\big( \conjelt^\ast [j \alpha] \big)\res{I,0}$.
	Using \Cref{Prop - Shape of prime to p torsion class restricted to affinoid over edge} we conclude that
	\begin{align*}
		\chi\res{I,0} &\equiv [j\alpha]^{-1}_{I,0} \,\, \Big(  \conjelt^\ast [j\alpha]_{I,0}^{-1} \,\, \conjelt^\ast[\alpha]_{I,0} \Big)^{-1}
		%	&=[j]^{-1} \, [\alpha]_{I,0}^{-1} \, \conjelt^\ast[j] \, \conjelt^\ast [\alpha]_{I,0} \,\conjelt^\ast[\alpha]_{I,0}^{-1} 
			= [\alpha]_{I,0}^{-1} \equiv \chi_{1,0} \quad\mod H^1 \big( I, \CO^{\times \times}(\Omega'_0) \big) .
	\end{align*}
	In view of the decomposition of \Cref{Prop - Decomposition of principal units and constants} this proves that $\chi = \chi_{1,0}$.
\end{proof}

\begin{theorem}\label{Thm - Group cohomology of G^0}
	There is a short exact sequence of solid abelian groups
	\begin{equation*}
		1 \lra \ul{\Hom} (  {G^0}, K^\times ) \lra \ul{H}^1 \big(  {G^0}, \CO^\times(\Omega) \big) \xrightarrow{\widetilde{P_\ast}} \BZ \oplus \BZ/ (q+1) \, \BZ \lra 0 
	\end{equation*}
	with $\widetilde{P_\ast}[j] = (1,1)$ and $\widetilde{P_\ast}[\alpha] = (0, -1)$.
	Moreover, this yields an isomorphism 
	\begin{align*}
		\BZ \,\oplus\, \BZ/(q^2 -1 )\, \BZ \,\oplus\, \ul{\Hom} \big(\CO^{\times \times}_F , \CO^{\times \times}_K \big) &\overset{\sim}{\lra} \ul{H}^1 \big( {G^0}, \CO^\times(\Omega) \big) \qquad\text{with} \\
		\big( n, k , \chi\big) &\lto [j]^n \, [\alpha]^k \, \big(\chi \circ \langle \det \rangle \big)
	\end{align*}
	on the underlying abelian groups.
	Under this isomorphism, the automorphism $\conjelt^\ast$ corresponds on the left hand side to $\big( n,k ,\chi \big) \mto \big( n, qk, \chi \big)$.
\end{theorem}
\begin{proof}
	Concerning the short exact sequence, due to \Cref{Thm - Left exact sequences for group cohomology} it remains to show that the image of $\widetilde{P_{\ast}}$ is equal to the (discrete) condensed subgroup $\BZ \oplus \BZ/ (q+1) \, \BZ$ of $\frac{1}{q-1} \BZ \oplus \BZ/ (q+1) \, \BZ $. 
	Since $\widetilde{P_\ast}[j] = (1,1)$ and $\widetilde{P_\ast}[\alpha]=(0,-1)$, this image certainly contains the former subgroup.
	Using that $\Im({P_\ast}) = \Ker (\Delta)$ for the connecting homomorphism \eqref{Eq - Connecting homomorphism}, it thus suffices to show that $\Im(\Delta) \subset \Coker(\Xi)$ contains at least $q-1$ elements.

	In \Cref{Lemma - Image under connecting homomorphism} we have seen that $\chi_{1,0} \mod \Im(\Xi)$ is an element of $\Im(\Delta)$.
	Moreover, recall the decompositions from \Cref{Lemma - Characters of quotient by principal units}.
	The homomorphism $\Xi$ preserves these decompositions so that the order of $\chi_{1,0} \mod \Im(\Xi)$ is equal to the order of $\chi_{1,0}$ in $H^1 \big(I, K^\times)[p'] $ modulo the image of the restriction
	\begin{equation*}
		H^1 \big(G_0, K^\times)[p'] \oplus H^1 \big( {}^\conjelt G_0, K^\times)[p'] \overset{\Xi}{\lra } H^1 \big(I, K^\times)[p'] .
	\end{equation*}
	However, the image of this restriction map only contains of $\widehat{\det}{}^k$, for $k\in \BZ/ (q-1)\, \BZ$.
	We deduce that the element $\chi_{1,0} \mod \Im(\Xi)$ is of order $q-1$ in $\Coker(\Xi)$ as desired.

	For the decomposition of $\ul{H}^1 \big( {G^0}, \CO^\times(\Omega) \big)$ we remark that ${\Hom}(G^0 , K^\times) = {\Hom} (F^\times , K^\times)$ by \Cref{Lemma - Prime to p torsion characters} (i).
	Analogously to \Cref{Lemma - Characters of quotient by principal units} and via $\CO_F^\times \cong \mu_{q-1}(F) \times \CO_F^{\times \times}$ one then deduces the isomorphism
	\begin{equation*}
		\ul{\Hom} \big( \mu_{q-1}(F) , K^\times \big) \oplus \ul{\Hom} \big( \CO^{\times \times}_F , \CO_K^{\times \times} \big) \overset{\sim}{\lra}	\ul{\Hom} (  {G^0}, K^\times ) .
	\end{equation*}
	The proof of the asserted decomposition of $\ul{H}^1 \big( {G^0}, \CO^\times(\Omega) \big)$ now goes along the same lines as the one of \Cref{Thm - Group cohomology for G_0}.	
\end{proof}

\begin{remark}\label{Rmk - Connection to Ardakov-Wadsley}
	Let $F$ be of characteristic $0$, assume $K$ contain the unramified quadratic extension $L$ of $F$, and recall that $\mathrm{PicCon}^{G^0}(\Omega)$ denotes the group of isomorphism classes of $G^0$-equivariant line bundles with integrable connection on $\Omega$.
	The main part of the argument of Ardakov and Wadsley \cite{ArdakovWadsley23EquivLineBun} for describing $\mathrm{PicCon}^{G^0}(\Omega)$ is the construction of an isomorphism 
	\begin{equation*}
		\mathrm{PicCon}^{G^0}(\Omega)_\mathrm{tors} \overset{\sim}{\lra} \Hom \big( \CO_L^\times / P_L^1 , K^\times \big)_\mathrm{tors}
	\end{equation*}
	which involves the choice of some $[z:1] \in \Omega_F(L)$ and where $P_L^1 \defeq \Ker( \mathrm{N} ) \cap \CO_L^{\times \times}$.
	Together with $\CO_L^\times / P_L^1 \cong \big( \CO^\times_D\big)^\mathrm{ab}$ induced by an $F$-algebra homomorphism $L \hookrightarrow D$ (\cite[Prop.\ 2.3.6]{ArdakovWadsley23EquivLineBun}), they obtain an isomorphism
	\begin{equation*}
		\mathrm{PicCon}^{G^0}(\Omega)_\mathrm{tors} \overset{\sim}{\lra} \Hom \big( \CO^\times_D , K^\times \big)_\mathrm{tors} .
	\end{equation*}
	Here, $D$ is the quaternion division algebra over $F$ and $\CO_D$ its maximal order.

	Since the extension $L/F$ is unramified, its norm map $\mathrm{N} $ gives rise to an isomorphism $\CO_L^{\times \times} / P_L^1 \cong \CO_F^{\times \times}$ and thus $\CO_L^\times / P_L^1 \cong \mu_{q^2 - 1}(L) \times \CO_F^{\times \times}$.
	It follows that 
	\begin{equation*}
		\Hom \big( \CO_L^\times / P_L^1 , K^\times \big)_\mathrm{tors} \cong \BZ/(q^2 -1)\, \BZ \,\oplus\, \Hom \big( \CO^{\times \times}_F , \CO_K^{\times \times} \big)_\mathrm{tors} .
	\end{equation*}

	Furthermore, there is a forgetful group homomorphism $\mathrm{PicCon}^{G^0}(\Omega) \ra \Pic^{G^0}(\Omega)$.
	It is a special case of a result of Taylor \cite[Cor.\ 7.4]{Taylor25CatLubinTateDrinfeldBun} that this map induces an isomorphism 
	\begin{equation*}
		\mathrm{PicCon}^{G^0}(\Omega)_\mathrm{tors} \overset{\sim}{\lra} \Pic^{G^0}(\Omega)_\mathrm{tors} .
	\end{equation*}
	Using \cite[Lemma 3.3.4]{ArdakovWadsley23EquivLineBun}, one can concretely determine the equivariant integrable connection with which each element of $\mathrm{Pic}^{G^0}(\Omega)_\mathrm{tors}$ can be endowed.
	\sloppy
	In this way, \Cref{Thm - Group cohomology of G^0} recovers the assertion that $\mathrm{PicCon}^{G^0}(\Omega)_\mathrm{tors}$ is isomorphic to $\Hom \big( \CO_L^\times / P_L^1 , K^\times \big)_\mathrm{tors}$ and thus to $\Hom \big( \CO^\times_D , K^\times \big)_\mathrm{tors}$.
	\qed
\end{remark}

\begin{theorem}\label{Thm - Group cohomology for G}
	\begin{altitemize}
		\item[(i)]
			There is an isomorphism of solid abelian groups
			\begin{align*}
				\BZ \oplus  \BZ/ (q^2 - 1) \, \BZ \oplus \ul{\Hom}\big( \unif^{2 \BZ} \times \CO^{\times \times}_F  , K^\times \big) &\overset{\sim}{\lra}  \ul{H}^1 \big( G^{(2)}, \CO^\times(\Omega) \big) 	\qquad\text{with} \\
				{(n,k, \chi)} &\lto [j]^n \, [\alpha]^k \, \big(\chi \circ \widetilde{\det} \big) 
			\end{align*}
			on the underlying abelian groups.
			Here, $[\alpha]$ is the unique $(q^2-1)$-torsion class which is mapped to $[\alpha] \in H^1 \big( G^0, \CO^\times(\Omega) \big)$, and $x \mto \widetilde{x} \defeq \unif^{v_\unif(x)} \langle x \rangle $ denotes the projection $F^\times \twoheadrightarrow \unif^\BZ \times \CO^{\times \times}_F $.
			
		\item[(ii)]
			There is an isomorphism of solid abelian groups
				\begin{equation*}
				\BZ \oplus \ul{\Hom} ( F^\times , K^\times ) \lra \ul{H}^1 \big( {G} , \CO^\times(\Omega) \big) \qquad\text{with}\qquad (n, \chi) \lto [j]^n \, (\chi \circ \det)
			\end{equation*}
			on the underlying abelian groups.
			
	\end{altitemize}
\end{theorem}
\begin{proof}
	We first consider $G$.
	Recall that $G^0$ is the kernel of the homomorphism $G \ra \BZ$, $g \mto v_\unif (\det(g))$, and that $n \mto s^n$ is a section to this quotient map. 
	Then $\ul{G^0}$ is a normal condensed subgroup of $\ul{G}$ so that we obtain an associated Hochschild--Serre spectral sequence by \Cref{Prop - Hochschild-Serre spectral sequence}. 
	Under the identifications $\ul{G}/ \ul{G^0} \cong {\BZ}$ and $ \CO^\times(\Omega)^\ul{G^0} = K^\times$ (see \Cref{Lemma - Invariant function is constant}) we extract from it the exact sequence
	\begin{align}\label{Eq - Concrete 5-term exact sequence}
		1 \lra \ul{H}^1 \big( \BZ , K^\times \big) \lra \ul{H}^1 \big( {G} , \CO^\times(\Omega) \big) \lra \ul{H}^1 \big( {G^0} , \CO^\times(\Omega) \big)^{\BZ} \lra \ul{H}^2 \big( \BZ , K^\times \big) .	
	\end{align}
	Here, the action of $\BZ$ on $ \ul{H}^1 \big( {G^0} , \CO^\times(\Omega) \big)$ is through the conjugation action $\conjelt^\ast$.

	By \Cref{Thm - Group cohomology of G^0} an element $[j]^n \, [\alpha]^{k} \, (\chi \circ \det) $ of $H^1 \big( {G^0} , \CO^\times(\Omega) \big)$, for $n\in \BZ$, $k\in \BZ/(q^2-1)\,\BZ$ and $\chi \in \Hom \big( \CO^{\times \times}_F , \CO^{\times \times}_K \big)$ is $\conjelt^\ast$-invariant if and only if $q k \equiv k \mod (q^2 -1)$ or equivalently if and only if $k$ is divisible by $(q+1)$.
	Because of $[\alpha]^{q+1} = \widehat{\det}{}^{-1}$ this shows that
	\begin{equation*}
		\BZ  \, \oplus \, \ul{\Hom} \big( \CO^\times_F , K^\times \big) \overset{\sim}{\lra} \ul{H}^1 \big( G^0 , \CO^\times(\Omega) \big)^{\BZ} \quad\text{with}\quad (n, \chi) \lto [j]^n \, (\chi \circ \det).
	\end{equation*}

	On the other hand, since $\BZ$ is discrete, we have $\ul{H}^2 ( \BZ, K^\times ) = \ul{H^2 (\BZ, K^\times)}$ which vanishes.
	Moreover, the first homomorphism of \eqref{Eq - Concrete 5-term exact sequence} is the embedding $\ul{\Hom} \big(\BZ ,K^\times \big) \hookrightarrow \ul{H}^1 \big( {G} , \CO^\times(\Omega) \big)$, $\mu \mto \mu \circ v_\unif \circ \det$. 
	Therefore, \eqref{Eq - Concrete 5-term exact sequence} becomes the split exact sequence
	\begin{equation*}
		1 \lra \ul{\Hom} \big( \BZ , K^\times \big) \lra \ul{H}^1 \big( G , \CO^\times(\Omega) \big) \lra \BZ \,  \oplus \, \ul{\Hom} \big( \CO^\times_F , K^\times \big)\lra 1 . 
	\end{equation*}
	The statement for $G$ now follows from $F^\times \cong \unif^{\BZ} \times \CO^\times_F$.

	For $G^{(2)}$, we use the identification $\ul{G^{(2)}}/ \ul{G^0} \cong {2 \BZ}$.
	Analogously, we obtain the short exact sequence	
	\begin{align*}
		1 \lra  \ul{\Hom} \big( 2 \BZ , K^\times \big) \lra \ul{H}^1 \big( G^{(2)} , \CO^\times(\Omega) \big) \lra \ul{H}^1 \big( {G^0} , \CO^\times(\Omega) \big)^{2\BZ} \lra 1 .
	\end{align*}
	But it follows from the above that every element of $ H^1 \big( {G^0} , \CO^\times(\Omega) \big)$ is invariant under $2\BZ$.
	We can therefore deduce (i) via \Cref{Thm - Group cohomology of G^0}.
\end{proof}

\begin{remark}\label{Rmk - Connection to Junger}
	In \cite{Junger23CohModpFibresEquivDrinfeld}, Junger classifies the $G^{(2)}$-equivariant line bundle on the semi\-stable formal model $\widehat{\Omega}$ of $\Omega$, for $K= \breve{F}$ the completion of the maximally unramified extension of $F$.
	He identifies certain equivariant line bundles $\omega_0$ and $\CL_0$ on $\widehat{\Omega}$, and proves the isomorphism
	%(coming from the Lie algebra of the universal formal module $\FX$ over $\widehat{\Omega}$ respectively related to $\FX[\Pi_D]$ being a so called Raynaud scheme).
	\begin{align*}
		\BZ \,\oplus\, \BZ \,\oplus\, \Hom\big( G^{(2)}_{\mathrm{disc}}, \CO^\times_{K} \big) \overset{\sim}{\lra} \Pic^{G^{(2)}}_{\mathrm{disc}}  \big(\widehat{\Omega} \big) \,,\quad 
		(n,m,\chi) \lto \big[ \omega_0^{\otimes n} \otimes \CL_0^{\otimes m} \otimes \CO_{\chi} \big] ,
	\end{align*}
	see Thm.\ 2.19 and Cor.\ 2.22 of {\it loc.\ cit.}.
	Here, $\Pic^{G^{(2)}}_{\mathrm{disc}} \big(\widehat{\Omega} \big) $ denotes the group of isomorphism classes of $G^{(2)}$-equivariant line bundles on $\widehat{\Omega}$ without any continuity condition, and also the occurring characters of $G^{(2)}$ are homomorphisms of abstract groups.

	Furthermore, under the canonical homomorphism $\Pic^{G^{(2)}}_{\mathrm{disc}} \big(\widehat{\Omega} \big) \ra \Pic^{G^{(2)}}_{\mathrm{disc}} (\Omega)$ induced by restriction to the generic fibre, $[\omega_0]$ is mapped to $[\CO(-1)]$ while $[\CL_0]$ is mapped to a class $[\CL_0\res{\Omega}]$ of order 
	%\footnote{There is a typo in \cite[Sect.\ 2.4]{Junger23CohModpFibresEquivDrinfeld} stating the order to be $(q+1)$ instead of $(q^2-1)$.} 
	$(q^2-1)$ which corresponds to an isotypic component of the pushforward of the structure sheaf of the first Drinfeld covering of $\Omega$, see Sect.\ 2.4 of {\it loc.\ cit.}.
	By \cite[Lemma A.8]{Taylor25EquivVecBun}, the $G^{(2)}$-equivariant structure of $\CL_0\res{\Omega}$ is continuous.
	Hence, $[\CL_0\res{\Omega}]$ is a generator of $\Pic^{G^{(2)}}(\Omega)[p']$.
\end{remark}

\appendix

\section{Condensed Group Cohomology}\label{Sect - Condensed group cohomology}

In this appendix, we recapitulate the definition of condensed group cohomology as treated for example in \cite{AnschuetzLeBras20SolidGrpCohom,BarthelSchlankStapletonWeinstein25RatK(n)LocSphere,Bosco23pAdicProEtCohomDrinfeldSymSp,Zou24CatFormFarguesConjTori} and consider some of its properties.
We begin by recalling relevant notions from condensed mathematics following Clausen and Scholze \cite{ClausenScholze19CondMath}.

\subsection{Condensed Mathematics and Definitions}

A \emph{condensed set} (\emph{group}, \emph{ring}, ...) is a sheaf of sets (groups, rings, ...) on the pro-\'etale site $\ast_\text{pro\'et}$ of a point\footnote{Or rather, the category of condensed sets (groups, rings, ...) is defined as in \cite[p.\ 15]{ClausenScholze19CondMath} to avoid set-theoretic problems.}, i.e.\ on the category $\Prof$ of profinite sets with coverings given by finite, jointly surjective families of maps.
For a condensed set $X$ we refer to $X(\ast)$ as its underlying set.

We let $\mathrm{ED}$ denote the full subcategory of $\mathrm{Prof}$ consisting of extremally disconnected compact Hausdorff spaces.
Then the topoi of $\mathrm{Prof}$ and $\mathrm{ED}$ (with the same kind of coverings) are equivalent.

There is a \emph{condensation} functor $X \mto \ul{X}$ from T1 topological spaces (groups, rings, ...) to condensed sets (groups, rings, ...) defined by $\ul{X}(S) \defeq C(S,X)$, for $S \in \Prof$. 
It is fully faithful when restricted to the full subcategory of compactly generated T1 spaces.
Therefore, we sometimes implicitly consider such spaces as condensed sets, i.e.\ omit the underline from the notation.
This functor admits a left adjoint $X \mto X(\ast)_\mathrm{top}$ by endowing the underlying set of $X$ with a certain topology.

Let $\mathrm{Cond(Ab)}$ denote the category of condensed abelian groups.
It is an abelian category containing all limits and colimits, and has a symmetric monoidal tensor product $\otimes$ and an internal $\Hom$-functor $\ul{\Hom}$.

The category $\mathrm{Cond(Ab)}$ has enough projective objects. 
We let $D(\mathrm{Cond(Ab)})$ denote its derived category, and let $\otimes^L$ (resp.\ $R\ul{\Hom}$) denote the derived tensor product (resp.\ derived internal $\Hom$-functor).

For a profinite set $S = \varprojlim_{i\in I} S_i$ with $S_i$ finite, consider the condensed abelian group $\BZ[S]^\blacksquare \defeq  \varprojlim_{i\in I} \ul{\BZ[S_i]}$ where $\BZ[S_i]$ carries the discrete topology.  
A \emph{solid abelian group} is a condensed abelian group $A$ such that, for all $S \in \Prof$, every morphism $\ul{S} \ra A$ uniquely extends to a morphism $\BZ[S]^\blacksquare \ra A$.
The full subcategory $\mathrm{Solid}$ consisting of solid abelian groups is abelian and stable under all limits, colimits and extensions.
Moreover, the functor $D(\mathrm{Solid}) \ra D(\mathrm{Cond(Ab)})$ is fully faithful.
\\

We also record some straightforward notions concerning condensed torsion subgroups which we could not find in the literature.

\begin{definition}
	For a condensed abelian group $A$ and $d\in \BN_{\geq 1}$, we let $A[d]$ denote the \emph{condensed $d$-torsion subgroup} of $A$ (see \Cref{Lemma - Condensed torsion subgroups} (i) below) defined by
	\begin{equation*}
		A[d](S) \defeq A(S)[d] = \big\{ a \in A(S) \,\big\vert\, da = 0 \big\} \,,\qquad\text{for $S \in \mathrm{ED}$.}
	\end{equation*}
	Analogously, we define the \emph{condensed torsion} (resp.\ \emph{$p$-power torsion}, \emph{prime-to-$p$-torsion}) \emph{subgroup} $A_\mathrm{tors}$ (resp.\ $A[p^\infty]$, $A[p']$) of $A$.
\end{definition}

\begin{lemma}\label{Lemma - Condensed torsion subgroups}
	\begin{altenumerate}
		\item 
			The presheaves $A[d]$, $A_\mathrm{tors}$, $A[p^\infty]$ and $A[p']$ are condensed abelian subgroups.
		
		\item 
			For a topological abelian group $A$ we have $\ul{A}[d] = \ul{A[d]}$.
			If $A_\mathrm{tors}$ (resp.\ $A[p^\infty]$, $A[p']$) is discrete, then $\ul{A}_\mathrm{tors} = \ul{A_\mathrm{tors}}$ (resp.\ $\ul{A}[p^\infty] = \ul{A[p^\infty]}$, $\ul{A}[p'] = \ul{A[p']}$).
	\end{altenumerate}
\end{lemma}
\begin{proof}
	To show that $A[d]$ is a sheaf, by the discussion after \cite[Prop.\ 2.7]{ClausenScholze19CondMath}, it suffices to verify that $A[d](\emptyset) = \{0\}$ and that for $S, S' \in \mathrm{ED}$ the natural map
	\begin{equation*}
		A[d](S\sqcup S' ) \lra A[d](S) \times A[d](S')
	\end{equation*}
	is a bijection.
	But this holds because $A$ itself is a condensed abelian group and taking $d$-torsion commutes with products.
	For the other presheaves in (i) one argues analogously.

	For a topological abelian group $A$ in (ii), we have $C(S, A)[d] = C \big(S,A[d] \big)$, for all $S \in \mathrm{ED}$.
	The inclusion $C(S, A)_\mathrm{tors} \subset C \big(S,A_\mathrm{tors} \big)$ is also clear.
	If $A_\mathrm{tors}$ is assumed to be discrete, this is an equality since then the image of every $\varphi \in C \big(S,A_\mathrm{tors} \big)$ is finite.
	The reasoning for $\ul{A}[p^\infty]$ and $\ul{A}[p']$ is similar.
\end{proof}

We now turn towards condensed group cohomology.
Let $G$ be a condensed group.
A \emph{condensed $G$-module} is a condensed abelian group $M$ endowed with a linear (left) $G$-action $G \times M \ra M$.
Equivalently, $M$ is a condensed (left) $\BZ[G]$-module where $\BZ[G]$ denotes the \emph{condensed group ring}, i.e.\ the sheafification of the presheaf $S \mto \BZ[G(S)]$.
We let $\mathrm{Cond(Ab)}_G$ denote the abelian category of condensed $G$-modules.
Furthermore, we define a \emph{solid $G$-module} to be a condensed $G$-module whose condensed abelian group is solid.

\begin{lemma}\label{Lemma - Solid G-module}
	Let $G$ be a topological group and $M$ a topological $G$-module.
	Then $\ul{M}$ is a condensed $\ul{G}$-module.

	Moreover, if the topology of $M$ is linear (i.e.\ its open subgroups form a neighbourhood basis of the identity element), Hausdorff and complete, then $\ul{M}$ is a solid $\ul{G}$-module.
\end{lemma}
\begin{proof}
	As condensation preserves products, $\ul{M}$ is a condensed $\ul{G}$-module by functoriality.
	For $M$ satisfying the additional properties, it is the statement of \cite[Lemma 3.2.1]{BarthelSchlankStapletonWeinstein25RatK(n)LocSphere} that then $\ul{M}$ is solid.	
\end{proof}

The topological groups relevant to our setting all will be \emph{Polish groups}, i.e.\ topological groups that are separable and completely metrisable.
For this class of groups, an \emph{open mapping theorem} holds: Any surjective continuous homomorphism between Polish groups is open, see \cite[V.32 Cor.\ 6]{Husain66IntroTopGrps}.

Moreover, if $X$ is a locally compact, second countable space and $M$ a Polish group, then $C(X,M)$ is a Polish group again, see \cite[X.3.3 Cor.\ b), X.1.6 Cor.\ 3]{Bourbaki66GenTop2}.
We remark in passing that in \Cref{Lemma - Solid G-module} the additional conditions on $M$ are fulfilled if $M$ is an ultrametrisable Polish group.

\begin{lemma}\label{Lemma - Strictly exact sequence gives exact sequence of condensed modules}
	Let $0\ra L \xrightarrow{f} M \xrightarrow{g} N \ra 0$ be a short strictly exact sequence of Polish abelian groups\footnote{By the open mapping theorem it suffices to demand that $f$ and $g$ are continuous homomorphisms, $f$ is a topological embedding, and that the sequence is algebraically exact.}.
	Then the induced sequence $0 \ra \ul{L} \ra \ul{M} \ra \ul{N} \ra 0$ of condensed abelian groups is exact. 
\end{lemma}
\begin{proof}
	Since $\Ker(f)$ remains a limit in the category of topological spaces and the functor $X \mto \ul{X}$ admits a left adjoint, the exactness of $0 \ra \ul{L} \ra \ul{M} \ra \ul{N}$ follows.

	It remains to show that $\ul{g}\colon \ul{M} \ra \ul{N} $ is an epimorphism.
	To this end, we fix $S \in \mathrm{ED}$ and $\varphi \in \ul{N}(S) = C(S, N)$.
	We note that by the open mapping theorem $g$ is a quotient map.
	Then $\varphi(S)$ is compact, and by \cite[IX.2 Prop. 18]{Bourbaki66GenTop2} there exists a compact subset $K$ of $M$ such that $g(K) = \varphi(S)$. 
	Since $S$ is extremally disconnected, $\varphi$ lifts to a map $\varphi'\colon S \ra K$ such that $\varphi = g \circ \varphi'$. 
	This shows that $\ul{g}(S)\colon C(S,M) \ra C(S,N)$ is surjective.	
\end{proof}

Following Bosco \cite[Def.\ B.1]{Bosco23pAdicProEtCohomDrinfeldSymSp} we now define:

\begin{definition}\label{Def - Condensed Group Cohomology}
	Let $G$ be a condensed group and $M$ a condensed $G$-module.
	We define the \emph{condensed group cohomology} of $G$ with coefficients in $M$ to be
	\begin{align*}
		R \ul{\Gamma} (G,M) \defeq R\ul{\Hom}_{\BZ[G]} (\BZ, M) 
	\end{align*}
	in $D(\mathrm{Cond(Ab)})$ where $\BZ$ carries the trivial $G$-action.
	We set $\ul{H}^n (G,M) \defeq R^n \ul{\Gamma} (G,M)$ and sometimes abbreviate $M^G \defeq R^0 \ul{\Gamma} (G,M)$.
\end{definition}

For a topological group $G$ and a topological $G$-module $M$, the continuous group cohomology of Tate \cite[\S 2]{Tate76RelK2GalCohom} is defined using the complex $\big( C^\bullet (G,M), d^\bullet \big)$ of continuous cochains $C^n (G, M) \defeq C(G^n, M)$.
To compare it to the above condensed group cohomology in the cases relevant to us we will need a slightly strengthened version of \cite[Prop.\ B.2]{Bosco23pAdicProEtCohomDrinfeldSymSp}.

\begin{proposition}\label{Prop - Comparison condensed and continuous group cohomology}
	Let $G$ be a locally profinite group.
	\begin{altenumerate}
		\item 
		For any solid $\ul{G}$-module $M$, the complex\footnote{With differentials induced by the bar resolution, see the proof of \cite[Prop.\ B.2]{Bosco23pAdicProEtCohomDrinfeldSymSp}.}
		\begin{equation*}
			M \lra \ul{\Hom}( \BZ[\ul{G}], M ) \lra  \ul{\Hom}( \BZ[\ul{G}^2], M ) \lra \ldots
		\end{equation*}
		of solid abelian groups is quasi-isomorphic to $R \ul{\Gamma} (G,M) $.
		
		\item 
		Suppose $M$ is an ultrametrisable Polish $G$-module so that $\ul{M}$ is a solid $\ul{G}$-module.
		Then $R \ul{\Gamma} (\ul{G}, \ul{M})$ is quasi-isomorphic to the condensation of $\big( C^\bullet (G, M), d^\bullet \big)$.
		In particular, for all $n \in \BN$, there are natural isomorphisms of abelian groups
		\begin{equation*}
			\ul{H}^n ( \ul{G}, \ul{M} )   (\ast) \cong  H^n(G, M)  .
		\end{equation*}
	\end{altenumerate}
\end{proposition}
\begin{proof}
	For (i), the same reasoning as in the proof of \cite[Prop.\ B.2]{Bosco23pAdicProEtCohomDrinfeldSymSp} works once we show that $\ul{\Ext}^j ( \BZ[\ul{G}^{i-1}], M ) = 0$, for all $j>0$ and $i>0$.
	But for $S\in \mathrm{ED}$ we have
	\begin{align*}
		\ul{\Ext}^j ( \BZ[\ul{G}^{i-1}], M )(S) 
		&= \Ext^j ( \BZ[S] \otimes \BZ[\ul{G}^{i-1}], M )
		= \Ext^j ( \BZ[\ul{S \times G^{i-1}}], M ) .
	\end{align*}
	The last $\Ext$-group vanishes by \cite[Lemma 2.2]{AnschuetzLeBras20SolidGrpCohom} since the set $S \times G^{i-1}$ is locally profinite.

	Also for (ii), we can argue along the line of the proof of Prop.\ B.2 in \cite{Bosco23pAdicProEtCohomDrinfeldSymSp}.
	For $S \in \mathrm{ED}$ we have by \cite[X.3.4 Cor.\ 2]{Bourbaki66GenTop2}
	\begin{align*}
		\ul{C^n(G, M)}(S) 
		&= C( S, C(G^n, M) )
		\cong C( S \times G^n , M) .
	\end{align*}
	Using the fully faithfulness of condensation on compactly generated T1 spaces and
	\begin{align*}
		\Hom_\mathrm{Cond(Set)} (\ul{S \times G^n} , \ul{M} )
		&\cong \Hom ( \BZ[ \ul{S \times G^n} ] , \ul{M} ) \\
		&\cong \Hom ( \BZ[S] \otimes \BZ[\ul{G}^n], \ul{M} ) \\
		&= \ul{\Hom} ( \BZ[\ul{G}^n], \ul{M} ) (S)
	\end{align*}
	we obtain natural isomorphisms $\ul{\Hom} ( \BZ[\ul{G}^n], \ul{M} ) \cong \ul{C^n(G, M)}$, for all $n\in \BN$.	
	These are compatible with the differentials of the two complexes. 
	The last claimed isomorphism follows from the exactness of the global sections functor $X \mto X(\ast)$.
\end{proof}

\begin{remark}\label{Rmk - Notation for condensed group cohomology}
	In the situation of (ii), i.e.\ for a locally profinite group $G$ and an ultrametrisable Polish $G$-module $M$, it follows that $\ul{H}^0 \big( \ul{G}, \ul{M} \big) = \ul{M^G}$ since condensation preserves limits.
	If the action of $G$ on $M$ is trivial, we also have $\ul{H}^1 \big(\ul{G},\ul{M} \big) \cong \ul{Z^1(G,M)} = \ul{\Hom (G,M)}$.

	This justifies our use of the notation 
	\begin{equation*}
		\ul{H}^n (G, M) \defeq \ul{H}^n \big(\ul{G}, \ul{M} \big)  \qquad\text{so that}\qquad \ul{H}^n (G,M) (\ast) = H^n (G,M) 
	\end{equation*}
	for the underlying abelian groups, as well as
	\begin{equation*}
		 \ul{\Hom}(G,A) \defeq \ul{H}^1 ( \ul{G}, \ul{A}) \qquad\text{so that}\qquad \ul{\Hom}(G,A) (\ast) = \Hom (G, A) .
	\end{equation*}
	when $A$ is an ultrametrisable Polish abelian group endowed with the trivial $G$-action.
\end{remark}

\subsection{Some Theorems for Condensed Group Cohomology}

Here, we collect a useful lemma for the condensed group cohomology of an inverse system as well as versions of Shapiro's lemma, the Hochschild--Serre spectral sequence and the Mayer--Vietoris sequence for condensed group cohomology.

\begin{lemma}\label{Lemma - Group cohomology and inverse system}
	Let $G$ be a condensed group and $( M_n)_{n\in \BN}$ an acyclic inverse system of condensed $G$-modules. 
	Then
	\begin{equation*}
		 \ul{H}^{0} \big( G,  \textstyle\varprojlim_{n\in \BN} M_n  \big) \overset{\sim}{\lra} \textstyle \varprojlim_{n\in\BN} \ul{H}^{0} \big( G, M_n \big) 
	\end{equation*}
	and, for $k \in \BN$, there are natural short exact sequences of condensed abelian groups
	\begin{equation*}
		0 \lra \textstyle R^1\!\varprojlim_{n\in \BN} \ul{H}^{k} \big( G, M_n \big) \lra \ul{H}^{k+1} \big( G, \varprojlim_{n\in \BN} M_n  \big) \lra \varprojlim_{n\in\BN} \ul{H}^{k+1} \big( G, M_n \big) \lra 0 .
	\end{equation*}
\end{lemma}
\begin{proof}
	We denote by $f_n \colon M_{n+1} \ra M_n$ the transition maps of $( M_n)_{n\in \BN}$.
	Since we assume the inverse system to be acyclic, the sequence
	\begin{equation}\label{Eq - Inverse limit short exact sequence}
		0 \lra \textstyle \varprojlim_{n\in \BN} M_n \lra \prod_{n\in \BN} M_n \overset{\tau}{\lra} \prod_{n\in \BN} M_n \lra 0 
	\end{equation}
	is exact.
	Here, $\tau$ is the difference between the identity and the shift map, i.e.\ $\tau$ is given by $(m_n)_{n\in \BN} \mto \big(m_n - f_n(m_{n+1}) \big)_{n\in \BN}$ on sections.

	From \eqref{Eq - Inverse limit short exact sequence} we obtain a long exact sequence of condensed group cohomology
	\begin{equation*}%\label{Eq - Inverse limit long exact sequence}
		\ldots \lra\ul{H}^k \big( G , \textstyle\varprojlim_{n\in \BN} M_n \big) \lra \ul{H}^k \Big( G, \prod_{n\in \BN} M_n \Big) \xrightarrow{\ul{H}^k(\tau)} \ul{H}^k \Big( G, \prod_{n\in \BN} M_n \Big)  \lra \ldots .
	\end{equation*}
	It induces $\ul{H}^0 \big( G,  \textstyle\varprojlim_{n\in \BN} M_n  \big)  \cong \Ker\big( \ul{H}^0(\tau) \big)$ and short exact sequences, for $k \in \BN$,
	\begin{equation*}
		0 \lra \Coker \big( \ul{H}^k(\tau) \big) \lra \ul{H}^{k+1} \big( G,  \textstyle\varprojlim_{n\in \BN} M_n  \big) \lra \Ker \big( \ul{H}^{k+1}(\tau) \big) \lra 0 .
	\end{equation*}
	Moreover, we have $\ul{H}^k \big( G, \prod_{n\in \BN} M_n \big) \cong \prod_{n\in \BN} \ul{H}^k \big( G, M_n)$ because $R\ul{\Hom}$ preserves limits in the second entry.
	Under this identification $\ul{H}^k(\tau)$ is the difference between the identity and the shift map again.
	We therefore have 
	\begin{equation*}
		\Ker \big( \ul{H}^{k}(\tau) \big) \cong \varprojlim_{n\in \BN} \ul{H}^k (G, M_n) \qquad\text{and}\qquad \Coker \big( \ul{H}^k(\tau) \big)  \cong R^1\!\varprojlim_{n\in \BN} \ul{H}^k (G, M_n)
	\end{equation*}
	which shows the claim.
\end{proof}

Now, let $G$ be a condensed group and $H$ a condensed subgroup of $G$.
Following Zou \cite[Def.\ 3.0.7]{Zou24CatFormFarguesConjTori}, we define for a condensed $H$-module $M$ the \emph{coinduction}
\begin{align*}
	\mathrm{coind}^G_H(M) \defeq \ul{\Hom}_{\BZ[H]} (\BZ[G],M)
\end{align*}
and considers it as a left condensed $G$-module via precomposing with the inversion of $G$.

\begin{lemma}\label{Lemma - Coinduction and condensation}
	Let $G$ be a locally profinite group with an open subgroup $H$.
	Let $M$ be an ultrametrisable Polish abelian group endowed with the trivial $H$-action.
	Then there is an isomorphism of condensed $\ul{G}$-modules 
	\begin{equation*}
		\mathrm{coind}^\ul{G}_\ul{H} (\ul{M}) \cong \ul{C(G/H, M)}
	\end{equation*}
	where $g\in G$ acts on $f\in C(G/H, M)$ by $g \cdot f \defeq f(g^{-1} \blank)$.
\end{lemma}
\begin{proof}
	For $S\in \Prof$, by \cite[X.3.4 Cor.\ 2]{Bourbaki66GenTop2} and fully faithfulness of $X \mto \ul{X}$ on the category of compactly generated Hausdorff spaces, we have natural isomorphisms
	\begin{align*}
		C\big( S, C  (G/H, M ) \big)
		&\cong C \big(G/H \times S, M \big) 
		\cong \Hom_{\mathrm{Cond(Set)}} \big( \ul{G/H}  \times\ul{S}  , \ul{M} \big) .
	\end{align*}
	Since the functor $X \mto \BZ[\ul{H}][X]$ is left adjoint to the forgetful functor from condensed $\BZ[\ul{H}]$-modules to condensed sets, we obtain
	\begin{align*}
		\Hom_{\mathrm{Cond(Set)}} \big(  \ul{G/H}\times \ul{S} , \ul{M} \big) 
		&\cong \Hom_{\BZ [\ul{H}] } \big( \BZ [ \ul{H} ][  \ul{G/H}  \times \ul{S}], \ul{M} \big) \\
		&\cong \Hom_{\BZ [\ul{H}] } \big( \BZ [ \ul{G} ] \otimes_{\BZ[\ul{H}]}   \BZ[\ul{H}] [ \ul{S} ], \ul{M} \big)  \\
		&\cong \Hom_{\BZ [\ul{H}] } \big( \BZ [ \ul{H} ][\ul{S}] , \ul{\Hom}_{\BZ[\ul{H}]} ( \BZ[ \ul{G} ] , \ul{M} ) \big) \\
		&\cong \Hom_{\mathrm{Cond(Set)}} \big( \ul{S} , \ul{\Hom}_{\BZ[\ul{H}]} ( \BZ[ \ul{G} ] , \ul{M} ) \big) .
	\end{align*}
	For the second isomorphism here, we have used that $X \mto \BZ[\ul{H}][X]$ is symmetric monoidal and that there is an isomorphism $\BZ[\ul{H}][\ul{G/H}] \cong \BZ[\ul{G}]$ of right $\BZ[\ul{H}]$-modules as $G/H$ is discrete \cite[Lemma 3.0.9]{Zou24CatFormFarguesConjTori}.
	The third isomorphism is the tensor-hom adjunction, cf.\ \cite[Prop.\ A.21]{Tang24ProfinSolidCohom}, and the last one comes from adjunction for $X \mto \BZ[\ul{H}][X]$ again.
	In total, we arrive at an isomorphism $	\ul{C (G/H, M)}	\cong \mathrm{coind}^\ul{G}_\ul{H} (\ul{M}) $ by the Yoneda lemma.
	Moreover, one verifies that the above isomorphisms are $\ul{G}(S)$-equivariant.
\end{proof}

\begin{proposition}[{Shapiro's lemma, \cite[Lemma 3.0.10]{Zou24CatFormFarguesConjTori}}]\label{Prop - Shapiro's lemma}
	Let $G$ be a condensed group with a condensed subgroup $H$, and let $M$ be a condensed $H$-module.
	If $\BZ[G]$ is projective over $\BZ[H]$, then there are canonical isomorphisms, for all $n\in \BN$, 
	\begin{align*}
		\ul{H}^n \big(G, \mathrm{coind}^G_H(M) \big) \cong \ul{H}^n (H, M) .
	\end{align*}
\end{proposition}

The condition on $\BZ[G]$ and $\BZ[H]$ is satisfied for example when $G/H$ is discrete \cite[Lemma 3.0.9]{Zou24CatFormFarguesConjTori}. 
\\

Now suppose that the condensed subgroup $H$ is \emph{normal}, i.e.\ $H(S) \subset G(S)$ is normal for all $S\in \Prof$.
For a condensed $G$-module $M$, the condensed $H$-group cohomology of $M$ then carries a residual $G/H$-action.
We have a version of the Hochschild--Serre spectral sequence in this situation.

\begin{proposition}[{Hochschild--Serre spectral sequence, cf.\ \cite[Prop.\ 3.3.6]{BarthelSchlankStapletonWeinstein25RatK(n)LocSphere}, \cite[Rmk.\ 3.0.16]{Zou24CatFormFarguesConjTori}}]\label{Prop - Hochschild-Serre spectral sequence}
	Let $G$ be a condensed group with a normal condensed subgroup $H$, and let $M$ be a condensed $G$-module.
	There is a spectral sequence 
	\begin{equation}\label{Eq - Hochschild-Serre spectral sequence}
		E_2^{i,j} = \ul{H}^i \big( G/H , \ul{H}^j ( H,M ) \big) \,\Rightarrow\, \ul{H}^{i+j} ( G, M ) .
	\end{equation}
	In particular, there is a $5$-term exact sequence of condensed abelian groups
	\begin{align*}
		0 \ra \ul{H}^1 \big( G/H , M^H \big) \ra \ul{H}^1 ( G, M) \ra \ul{H}^1 ( H, M )^{G/H} \ra \ul{H}^2 \big( G/H , M^H \big) \ra \ul{H}^2 (G, M) .
	\end{align*}
\end{proposition}

We would like to deduce this spectral sequence as the Grothendieck spectral sequence associated to the composition of functors $(\blank)^{G/H}\circ (\blank)^H $.
However, the category of condensed $H$-modules does not necessarily have enough injective objects\footnote{For example, the only injective object of $\mathrm{Cond(Ab)}$ is the zero object.}.
To circumvent this problem, we pass to the category $\mathrm{Cond_\kappa(Ab)}_G$ of $\kappa$-condensed $G$-modules, cf.\ \cite[Rmk.\ 1.3]{ClausenScholze19CondMath}, where $\kappa$ is an uncountable strong limit cardinal such that $G$ is a $\kappa$-condensed group.
For a $\kappa$-condensed $G$-module $M$, we define
\begin{equation*}
	R \ul{\Gamma}_\kappa (G,M) \defeq R\ul{\Hom}_{\kappa, \BZ[G]} (\BZ, M) \qquad\text{and}\qquad \ul{H}^n_\kappa (G,M) \defeq R^n \ul{\Gamma}_\kappa (G,M) .
\end{equation*}

\begin{lemma}\label{Lemma - kappa-condensed and condensed group cohomology}
	Let $G$ be a condensed group and $M$ a condensed $G$-module.
	Then there exists an uncountable strong limit cardinal $\kappa$ such that $G$ and $M$ are $\kappa$-condensed and, for all $\kappa' \geq \kappa$,
	\begin{equation*}
		\ul{H}^n (G,M) = \ul{H}^n_{\kappa'} (G,M)\,,\quad\text{for all $n\in \BN$.}
	\end{equation*}
\end{lemma}
\begin{proof}
	Let $\kappa_{-1}$ be an uncountable strong limit cardinal such that $G$ and $M$ are $\kappa_{-1}$-condensed sets.
	Let $P_\bullet$ be a projective resolution of $\BZ[G]$ in $\mathrm{Cond_{\kappa_{-1}}(Ab)}_G$ so that $\ul{H}^n_{\kappa_{-1}} (G,M)$ is the cohomology of the complex $\ul{\Hom}_{\kappa_{-1}, \BZ[G]} (P_\bullet, M)$.

	When viewing $\mathrm{Cond_{\kappa_{-1}}(Ab)}_G$ as subcategory of $\mathrm{Cond_{\kappa'}(Ab)}_G$, for $\kappa'\geq \kappa_{-1}$, or $\mathrm{Cond(Ab)}_G$ this remains a projective resolution of $\BZ[G]$, see the proof of \cite[Thm.\ 5.1]{Land22CondMath}.
	Therefore, $P_\bullet$ also computes $\ul{H}^n_{\kappa'} (G,M)$ and $\ul{H}^n (G,M)$.

	As seen in the proof of \cite[Prop.\ 5.5]{Land22CondMath}, for all $i \in \BN$, there exists $\kappa_i \geq \kappa_{-1}$ such that 
	\begin{equation*}
		\ul{\Hom}_{\kappa_i, \BZ[G]} (P_i, M) = \ul{\Hom}_{\kappa', \BZ[G]} (P_i, M) = \ul{\Hom}_{\BZ[G]} (P_i, M) \,,
	\end{equation*}
	for all $\kappa'\geq \kappa_i$.
	We define $\kappa \defeq \sup_{i\in \BN} \kappa_i$.
	For $\kappa' \geq \kappa$ we then have
	\begin{align*}
		\ul{\Hom}_{\kappa', \BZ[G]} (P_\bullet, M) = \ul{\Hom}_{ \BZ[G]} (P_\bullet, M) 
	\end{align*}
	which proves the claim.  
\end{proof}

\begin{proof}[{Proof of \Cref{Prop - Hochschild-Serre spectral sequence}}]
	By \Cref{Lemma - kappa-condensed and condensed group cohomology} and taking the supremum we find an uncountable strong limit cardinal $\kappa$ such that $\ul{H}^n = \ul{H}^n_\kappa$ for all group cohomology terms appearing in \eqref{Eq - Hochschild-Serre spectral sequence}.
	Since the inclusion functor of $\mathrm{Cond_\kappa(Ab)}$ into $\mathrm{Cond(Ab)}$ is exact (see the proof of \cite[Thm.\ 5.1]{Land22CondMath}), it thus suffices to work with $\mathrm{Cond_\kappa(Ab)}_G$ and show the spectral sequence for $\kappa$-condensed group cohomology.

	In this situation, we can apply the classical reasoning for the Hochschild--Serre spectral sequence, see for example \cite[6.8.2]{Weibel94IntroHomolAlg}:
	For $M \in \mathrm{Cond_\kappa(Ab)}_G$ we have 
	%$\ul{\Gamma}_\kappa (G, \blank) = \ul{\Gamma}_\kappa (G/H, \blank) \circ \ul{\Gamma}_\kappa (H, \blank)$
	\begin{align*}
		\ul{\Hom}_{\kappa, \BZ[G/H]} \big( \BZ, \ul{\Hom}_{\kappa, \BZ[H]} ( \BZ, M ) \big) 
		&\cong \ul{\Hom}_{\kappa, \BZ[G/H]} \big( \BZ, \ul{\Hom}_{\kappa, \BZ[G]} \big( \BZ[G/H] , M \big) \big) \\
		&\cong \ul{\Hom}_{\kappa, \BZ[G]} \big( \BZ \otimes_{\kappa, \BZ[G/H]} \BZ[G/H] , M \big) \\
		&\cong \ul{\Hom}_{\kappa, \BZ[G]} ( \BZ,  M ) 
	\end{align*}
	by \cite[Rmk.\ 3.0.2]{Zou24CatFormFarguesConjTori} and the tensor-hom adjunction \cite[Prop.\ A.21]{Tang24ProfinSolidCohom}.
	Furthermore, $\mathrm{Cond_\kappa(Ab)}_G$ is the category of modules on the ringed site $\big(\ast_{\kappa\text{-pro\'et}}, \BZ[G]\big)$ and similarly for $G/H$.
	These categories thus have enough injective objects.

	Finally, $(\blank)^H$ is right adjoint to the forgetful functor $\mathrm{Cond_\kappa(Ab)}_G \ra \mathrm{Cond_\kappa(Ab)}_{G/H}$ induced by $G \ra G/ H$ which is exact. 
	Therefore, $(\blank)^H$ preserves injective objects and the Grothendieck spectral sequence associated to $(\blank)^{G/H} \circ (\blank)^H = (\blank)^G$ exists.
\end{proof}
$\,$

Finally, we consider a family $G_i$, $i\in I$, of condensed groups which all have a common condensed subgroup $H$.
Following the terminology of \cite[Sect.\ I.1.2]{Serre80Trees} for the setting of abstract groups, we call the associated colimit $\ast_H \, G_i$ the \emph{condensed amalgamated sum} of the $G_i$ along $H$.
It arises as the sheafification of the presheaf $S \mto \ast_{H(S)}\, G_i(S)$ on $\Prof$.

\begin{proposition}\label{Prop - Short exact sequence for condensed amalgamated sum}
	Let $G_1$ and $G_2$ be condensed groups with a common condensed subgroup $H$ and let $G \defeq G_1 \ast_H G_2$ denote their condensed amalgamated sum.
	Then the sequence 
	\begin{equation*}
	\arraycolsep=1.4pt
	\def\arraystretch{1.2}
	\begin{array}{rclcrcl}
		0 \lra \BZ[G] \otimes_{\BZ[H]} \BZ[G] &\lra&  \big( \BZ[G] \otimes_{\BZ[G_1]} \BZ[G] \big) &\oplus&  \big( \BZ[G] \otimes_{\BZ[G_2]} \BZ[G] \big) &\lra& \BZ[G] \lra 0 \\
													a \otimes b &\lto& ( a \otimes b ,  a \otimes b)	\,,							&&(a_1 \otimes b_1, a_2 \otimes b_2) &\lto& a_1 b_1 - a_2 b_2
	\end{array}
	\end{equation*}
	of $\BZ[G]$-$\BZ[G]$-bimodules is exact.
\end{proposition}
\begin{proof}
	This short exact sequence can be derived from its analogue for abstract groups. 
	Let $G'$ denote the presheaf $S \mto G_1(S) \ast_{H(S)} G_2(S)$.
	Then, \cite[Thm.\ 2]{Dic77} implies that 
	\begin{align*}
			0 \lra &\, \BZ[G'(S)] \otimes_{\BZ[H(S)]} \BZ[G'(S)] \\
					&\lra \! \Big( \BZ[G'(S)] \otimes_{\BZ[G_1(S)]} \BZ[G'(S)]\Big)\! \oplus\!  \Big( \BZ[G'(S)] \otimes_{\BZ[G_2(S)]} \BZ[G'(S)] \Big) \! \lra \BZ[G'(S)] \lra 0 
	\end{align*}
	is exact for all $S \in \Prof$.
	It follows that	the sequence of the proposition is an exact sequence of presheaves when $G$ is replaced by $G'$.
	Therefore, the associated sequence of sheaves is exact as well.
	The claim now follows because $G$ is the sheafification of $G'$ and thus $\BZ[G]$ the sheafification of $S \mto \BZ[G'(S)]$.
\end{proof}

\begin{corollary}[{Mayer--Vietoris sequence\footnote{
		For the amalgamated sum of abstract groups, such a ``Mayer--Vietoris-type'' long exact sequence of group cohomology is the content of \cite[Thm.\ 2.3]{Swan69GrpsCohomDim} for example.
		}}]\label{Cor - Mayer-Vietoris sequence}
	Let $G \defeq G_1 \ast_H G_2$ be the amalgamated sum of condensed groups $G_1$ and $G_2$ along a common condensed subgroup $H$.
	Let $M$ be a condensed $G$-module.
	Then the canonical restriction maps induce a functorial long exact sequence of condensed abelian groups
	\begin{equation*}
		\ldots \lra \ul{H}^n ( G , M ) \lra \ul{H}^n (G_1 , M) \oplus \ul{H}^n (G_2 , M) \lra \ul{H}^n ( H, M) \lra \ldots .
	\end{equation*}
\end{corollary}
\begin{proof}
	We take the tensor product of the short exact sequence of \Cref{Prop - Short exact sequence for condensed amalgamated sum} with the trivial left $\BZ[G]$-module $\BZ$ from the right:
	\begin{equation}\label{Eq - Resolution of trivial representation}
		0 \lra \BZ[G] \otimes_{\BZ[H]} \BZ \lra  \big( \BZ[G] \otimes_{\BZ[G_1]} \BZ \big) \oplus  \big( \BZ[G] \otimes_{\BZ[G_2]} \BZ \big) \lra \BZ \lra 0 .
	\end{equation}
	This sequence of left $\BZ[G]$-modules remains exact since $\BZ[G]$ in particular is a flat right $\BZ[G]$-module. 
	For a condensed $G$-module $M$, we moreover have
	\begin{align*}
		R \ul{\Hom}_{\BZ[G]} \big( \BZ[G] \otimes_{\BZ[H]} \BZ , M) &\cong 		R \ul{\Hom}_{\BZ[G]} \big( \BZ[G] \otimes_{\BZ[H]}^L \BZ , M) \\
			&\cong R \ul{\Hom}_{\BZ[H]} \big( \BZ , R \ul{\Hom}_{\BZ[G]}(\BZ[G] , M) \big) \\
			&\cong R \ul{\Hom}_{\BZ[H]} \big( \BZ , M \big)
	\end{align*}
	using that $\BZ[G]$ is a flat right $\BZ[H]$-module and the derived version of the tensor-hom adjunction \cite[Prop.\ A.21]{Tang24ProfinSolidCohom}.
	We obtain analogous isomorphisms for $G_1$ and $G_2$.
	Therefore, \eqref{Eq - Resolution of trivial representation} induces a distinguished triangle
	\begin{equation*}
		R \ul{\Gamma} ( G, M ) \lra R \ul{\Gamma}( G_1, M  ) \oplus R \ul{\Gamma}  ( G_2, M  ) \lra R \ul{\Gamma}  ( H, M  ) \lra R \ul{\Gamma}  ( G, M  )[1]
	\end{equation*}
	which yields the claimed long exact sequence.
\end{proof}

Given topological groups $G_i$, $i\in I$, with a common topological subgroup $H$, one can likewise form their \emph{amalgamated sum} $\ast_H\, G_i$ in the category of topological groups.
It is constructed as the colimit of the underlying abstract groups (see \cite[Sect.\ I.1.2]{Serre80Trees}) endowed with the final topology with respect to the canonical homomorphisms $f_i \colon G_i \ra \ast_H \, G_i$.
We remark that the $f_i$ are injective by Thm.\ 1 of {\it loc.\ cit.}.

We want to compare the condensation of $\ast_H \, G_i$ with the condensed amalgamated sum of the $\ul{G_i}$ along $\ul{H}$ in certain good situations.

\begin{proposition}\label{Prop - Condensation of pushout}
	In the above setting, additionally assume that $\ast_H \, G_i$ is Hausdorff, the $f_i$ are topological embeddings and $H$ is an open subgroup of $\ast_H \, G_i$.
	Then the canonically induced homomorphism
	\begin{equation*}
		\ast_\ul{H} \,\, \ul{G_i} \lra \ul{\ast_H \, G_i}
	\end{equation*}
	is an isomorphism of condensed groups.
\end{proposition}
\begin{proof}
	The condensed group $\ast_\ul{H}\,\, \ul{G_i} $ is the sheafification of $S \mto \ast_{C(S,H)} \, C(S,G_i)$ where the latter colimit is taken in the category of groups.
	Hence, it suffices to show that, for all $S \in \Prof$, the induced map
	\begin{equation*}
		\alpha \colon  \ast_{C(S,H)} \,C(S,G_i) \lra C \big( S, \ast_H \, G_i \big)
	\end{equation*}
	is injective and locally surjective.
	We fix such profinite $S$.

	In the following, we use that elements of amalgamated sums are uniquely represented by reduced words. 
	We refer to {\it loc.\ cit.} for the relevant definitions and precise statements.
	In this context, we choose a system of right coset representatives $R_i$ of $G_i/ H$ with $1 \in R_i$, for all $i\in I$. 
	Since $H$ is open in $G_i$, the functions $C(S,R_i) \subset C(S,G_i)$ then constitute right coset representatives of $C(S,G_i)/ C(S,H) \cong C(S, G_i/H)$.

	For $\psi \in \ast_{C(S,H)} \, C(S,G_i)$ there exist unique $n \geq 0$, a type $(i_1,\ldots,i_n) \in I^n$ and a reduced word $(\eta,\phi_{i_1},\ldots,\phi_{i_n})$ of this type with $\eta \in C(S,H)$ and $\phi_{i_k} \in C(S, R_{i_k})$, for $k=1,\ldots, n$, such that $	\psi = \eta \, \phi_{i_1} \cdots \phi_{i_n} $.
	It follows that, for all $s\in S$, 
	\begin{equation*}
		\alpha(\psi)(s) = \eta(s) \, \phi_{i_1}(s) \cdots \phi_{i_n}(s) 
	\end{equation*}
	is a representation of $\alpha(\psi)(s)$ by a reduced word as an element of $\ast_H \, G_i$.
	If $\psi \in \Ker(\alpha)$, the unique\-ness of such a representation therefore implies that $\eta$ and all $\phi_{i_k}$ are equal to the trivial function.
	Hence, $\psi$ is the identity element of $\ast_{C(S,H)} \, C(S,G_i)$.
	This show that $\alpha$ is injective.

	To prove that $\alpha$ is locally surjective, we consider $\phi \in C(S, \ast_H \, G_i)$.
	Because $H $ is open in $\ast_H \, G_i$ and by passing to a disjoint open covering of $S$, we may assume that $\phi$ takes values in only one coset $Hg$ of $\ast_H \, G_i$.
	Then there exist $n\geq 0$, a type $(i_1,\ldots,i_n) \in I^n$ and a reduced word $(h,r_{i_1},\ldots, r_{i_n})$ of this type with $h \in H$ and $r_{i_k} \in R_{i_k}$ such that $g = h\, r_{i_1} \cdots r_{i_n}$.
	We let $\phi_{i_k} \in C(S, R_{i_k})$ denote the constant function with value $r_{i_k}$ and define $\eta \in C(S,H)$ to be $\eta(s)\defeq   \phi(s)\, r_{i_n}^{-1} \cdots r_{i_1}^{-1}$.
	Then $\psi \defeq \eta \, \phi_{i_1} \cdots \phi_{i_n} $ is a preimage of $\phi$ under $\alpha$.	
\end{proof}

\end{document}